\documentclass[11pt]{ip-journal}
\makeatletter
\def\l@section{\@tocline{1}{10pt}{1pc}{}{}}
\def\l@subsection{\@tocline{2}{0pt}{1pc}{4.6em}{}}
\def\l@subsubsection{\@tocline{3}{0pt}{1pc}{7.6em}{}}
\renewcommand{\tocsection}[3]{%
  \indentlabel{\@ifnotempty{#2}{\makebox[2.3em][l]{%
    \ignorespaces#1 #2.\hfill}}}\textbf{#3}}
\renewcommand{\tocsubsection}[3]{%
  \indentlabel{\@ifnotempty{#2}{\hspace*{2.3em}\makebox[2.3em][l]{%
    \ignorespaces#1 #2.\hfill}}}#3}
\renewcommand{\tocsubsubsection}[3]{%
  \indentlabel{\@ifnotempty{#2}{\hspace*{4.6em}\makebox[3em][l]{%
    \ignorespaces#1 #2.\hfill}}}#3}
\makeatother

\usepackage{amsmath,amssymb,amscd,amsthm,graphicx,enumerate,tikz-cd}

\usepackage{hyperref}
\hypersetup{breaklinks=true}

\usepackage{soul}

\usepackage{enumitem} 
\newlist{condenum}{enumerate}{1} 
\setlist[condenum]{label=\bfseries Condition \arabic*.,  ref=\arabic*, wide}

\newcounter{mtheorem}
\newtheorem{mtheorem}[mtheorem]{Theorem}

\newtheorem{mcor}[mtheorem]{Corollary}

\numberwithin{equation}{section}
\theoremstyle{plain}



\usepackage[utf8]{inputenc}
\usepackage{amssymb}
\usepackage{amsthm}
\usepackage{mathrsfs}
\usepackage{amsfonts}
\usepackage{amsmath}
\usepackage{mathtools}
\usepackage{makecell}
\usepackage{xcolor}
\usepackage[margin=1.25in]{geometry}

\usepackage{lipsum}
\makeatletter
\def\ps@pprintTitle{%
 \let\@oddhead\@empty
 \let\@evenhead\@empty
 \def\@oddfoot{}%
 \let\@evenfoot\@oddfoot}
\makeatother

\newcommand{\R}{\mathbb{R}}

\newcommand{\Rm}{\textnormal{Rm}}

\newcommand{\Ric}{\textnormal{Ric}}

\newcommand{\rii}{\rightarrow\infty}
\newcommand{\ri}{\rightarrow}

\newcommand{\diam}{\textnormal{diam}}

\numberwithin{equation}{section}

\newtheorem{theorem}{Theorem}[section]
\newtheorem{lem}[theorem]{Lemma}
\newtheorem{remark}[theorem]{Remark}
\newtheorem{prop}[theorem]{Proposition}

\newtheorem{cor}[theorem]{Corollary}
\newtheorem{con}[theorem]{Conjecture}
\theoremstyle{definition}
\newtheorem{defn}[theorem]{Definition}
\newtheorem*{theorem*}{Theorem}
\newtheorem{eg}[theorem]{Example}

\usepackage{xpatch}
\makeatletter
\xpatchcmd{\tableofcontents}{\contentsname \@mkboth}{\small\contentsname \@mkboth}{}{}
\xpatchcmd{\listoffigures}{\chapter *{\listfigurename }}{\chapter *{\small\listfigurename }}{}{}
\makeatother

\makeatletter
\def\blfootnote{\xdef\@thefnmark{}\@footnotetext}
\makeatother

\begin{document}

\begin{abstract}
In $1996$, H.-D.~Cao constructed a $U(n)$-invariant steady gradient K\"ahler-Ricci soliton on $\mathbb{C}^{n}$ and asked whether every steady gradient K\"ahler-Ricci soliton of positive curvature on $\mathbb{C}^{n}$ is necessarily $U(n)$-invariant (and hence unique up to scaling). Recently, Apostolov-Cifarelli answered this question in the negative for $n=2$. Here, we construct a family of  $U(1)\times U(n-1)$-invariant, but not $U(n)$-invariant, complete steady gradient K\"ahler-Ricci solitons with strictly positive curvature operator on real $(1,\,1)$-forms (in particular, with strictly positive sectional curvature) on $\mathbb{C}^{n}$ for $n\geq3$, thereby answering Cao's question in the negative for $n\geq3$. This family of steady Ricci solitons interpolates between Cao's $U(n)$-invariant steady K\"ahler-Ricci soliton and the product of the cigar soliton and Cao's $U(n-1)$-invariant steady K\"ahler-Ricci soliton. This provides the K\"ahler analog of the Riemannian flying wings construction of the third-named author.
In the process of the proof, we also demonstrate that the almost diameter rigidity of $\mathbb{P}^{n}$ endowed with the Fubini-Study metric does not hold even if the curvature operator is bounded below by $2$ on real $(1,\,1)$-forms.
\end{abstract}

\title[A family of K\"ahler flying wing steady solitons]{A family of K\"ahler flying wing steady Ricci solitons}

\author[P.-Y.~Chan]{Pak-Yeung Chan}
\address[]{Department of Mathematics, National Tsing Hua University, Hsin-Chu, Taiwan}
\email{pychan@math.nthu.edu.tw}

\author[R.~J.~Conlon]{Ronan J.~Conlon}
\address[]{Department of Mathematical Sciences, The University of Texas at Dallas, Richardson, TX 75080, USA}
\email{ronan.conlon@utdallas.edu}

\author[Y.~Lai]{Yi Lai}
\address[]{Department of Mathematics, Stanford University, Stanford, CA 94305, USA}
\email{yilai@stanford.edu}

\maketitle

\section{Introduction}
\subsection{Overview}
Ricci solitons are self-similar solutions of the Ricci flow that serve as generalizations of Einstein manifolds. They play an important role in the singularity analysis of the Ricci flow. Specifically, a \emph{Ricci soliton} is a triple $(M,\,g,\,X)$, where $M$ is a Riemannian manifold endowed with a complete Riemannian metric $g$ and a complete vector field $X$, such that
\begin{equation}\label{e: soliton}
\Ric_{g}+\frac{1}{2}\mathcal{L}_{X}g=\lambda\,g
\end{equation}
for some $\lambda\in\mathbb{R}$. A Ricci soliton is called \emph{steady} if $\lambda=0$, \emph{expanding}
if $\lambda<0$, and \emph{shrinking} if $\lambda>0$. If $X=\nabla^{g} f$ for some smooth real-valued function $f$ on $M$,
then we say that $(M,\,g,\,X)$ (or $(M,\,g,\,f)$) is a \emph{gradient} Ricci soliton, and we call $f$ the (\emph{soliton}) \emph{potential function}. In this case, the soliton equation \eqref{e: soliton}
becomes $$\Ric_{g}+\nabla^{2}f = \lambda\,g.$$
If the potential function moreover has a critical point $p\in M$, then we denote the soliton by $(M,g,f,p)$.

In \eqref{e: soliton}, if $g$ is K\"ahler and $X$ is real holomorphic, then we say that $(M,\,g,\,X)$ is a \emph{K\"ahler-Ricci soliton}. Let $\omega$ denote the K\"ahler
form of $g$. If $(M,\,g,\,X)$ is in addition gradient, then \eqref{e: soliton} may be rewritten as
\begin{equation*}
\rho_{\omega}+i\partial\bar{\partial}f=\lambda\,\omega,
\end{equation*}
where $\rho_{\omega}$ is the Ricci form of $\omega$ and $f$ is the potential. 
In this article, we are concerned with complete steady gradient K\"ahler-Ricci solitons. 

In real dimension $2$, the only example of a non-flat steady gradient Ricci soliton is Hamilton's cigar soliton \cite{cigar} which is rotationally symmetric. For real dimension $n\ge3$, the only $n$-dimensional non-flat rotationally symmetric (i.e., $O(n)$-symmetric) steady gradient Ricci soliton is the Bryant soliton \cite{bryant}. More recently, the third-named author constructed a family of $\mathbb{Z}_2\times O(n-1)$-symmetric, non-rotationally symmetric, steady gradient Ricci solitons on $\mathbb{R}^{n}$
for $n\ge3$. These solitons are collapsed for $n=3$ and are called ``flying wings'', reflecting the fact that they are asymptotic to two-dimensional sectors at infinity.
For $n\geq4$, they are non-collapsed. Moreover, in even dimensions $n\ge4$ they are not K\"ahler, since the curvature operator is strictly positive everywhere; cf.~Section 2 or \cite{DZ20}. 

In the K\"ahler world, for any complex dimension $n\geq1$, Cao \cite{Cao1996} constructed a $U(n)$-invariant steady gradient K\"ahler-Ricci soliton on $\mathbb{C}^{n}$ and on the canonical bundle $K_{\mathbb{P}^{n}}$ of complex projective space $\mathbb{P}^n$, thereby generalizing Hamilton's cigar soliton to higher dimensions. Further generalizations were subsequently obtained by Dancer-Wang \cite{DancerWang2011}, Yang \cite{Yang12}, Biquard-Macbeth \cite{BM17}, the second-named author and Deruelle \cite{CD20}, and more recently by Sch\"afer \cite{Sch21, Sch20}. Cao-Hamilton \cite{CH00} showed that the underlying complex manifold of a complete steady gradient K\"ahler-Ricci soliton with strictly positive Ricci curvature admitting a critical point of the scalar curvature must be diffeomorphic to $\R^{2n}$. Bryant \cite{Bry08}
and Chau-Tam \cite{CT05} independently improved this result by showing that the underlying complex manifold of such a steady gradient K\"ahler-Ricci soliton must in fact be biholomorphic to $\mathbb C^{n}$.
As it turns out, Cao's steady gradient K\"ahler-Ricci soliton on $\mathbb C^n$ has strictly positive sectional curvature. In light of this fact, he conjectured the following.
\begin{con}[Cao's Conjecture \cite{Cao1996}]\label{Cao conj}
   A complete steady gradient K\"ahler-Ricci soliton with positive sectional curvature on $\mathbb{C}^n$ must be isometric (up to scaling) to Cao's $U(n)$-invariant steady gradient K\"ahler-Ricci soliton on $\mathbb{C}^n$.
\end{con}

As pointed out by Cao \cite{Cao1996}, this conjecture is true when $n=1$ as the real Killing vector field $J\nabla f$ provides the $U(1)$-symmetry. In this case, the soliton is the cigar soliton. In higher dimensions, it is natural to investigate the local version of this uniqueness problem, namely whether or not Cao's soliton can be perturbed to another steady gradient K\"ahler-Ricci soliton nearby. To this end, Chau-Schn\"urer \cite{CS05} demonstrated that Cao's soliton on $\mathbb{C}^n$ is dynamically stable under sufficiently small perturbations of the K\"ahler potential which have fast decay at infinity. Uniqueness of Cao's soliton under suitable $C^1$ asymptotic conditions at spatial infinity was proved by Cui \cite{CuiThesis16} via Brendle's Killing vector field method \cite{brendlesteady3d, Brendle_jdg_high} (see also \cite{CF16, CD20, Sch20}).

In the more rigid \emph{shrinking} soliton case, Ni \cite{ni2005ancient} showed that $\mathbb{P}^n$ endowed with the $U(n+1)$-invariant Fubini-Study metric is the unique shrinking gradient K\"ahler-Ricci soliton with strictly positive bisectional curvature. This condition is strictly weaker than strictly positive sectional curvature and strictly positive curvature operator on real $(1,\,1)$-forms. In the more flexible \emph{expanding} soliton case, there is a one-parameter family of $U(n)$-invariant expanding gradient K\"ahler-Ricci solitons with positive curvature on $\mathbb C^n$ constructed by Cao \cite{Cao1997jdg}. Furthermore, the second-named author and Deruelle \cite{con-der} proved that there exist 
continuous families of asymptotically conical expanding gradient K\"ahler-Ricci solitons with positive curvature on $\mathbb{C}^{n}$. Heuristically speaking, steady K\"ahler-Ricci solitons exhibit behavior residing on the cusp of shrinking and expanding solitons. They are particularly delicate due to volume collapsing phenomena \cite{DZ20}. It is therefore tempting to understand the general uniqueness of positively curved steady K\"ahler-Ricci solitons.

To this end, Apostolov-Cifarelli \cite{apostolov2023hamiltonian} recently constructed counterexamples to Conjecture \ref{Cao conj} for $n=2$. More precisely, they used Hamiltonian two-forms and toric geometry to construct a one-parameter family of $U(1)\times U(1)$-invariant positively curved steady gradient K\"ahler-Ricci solitons on $\mathbb{C}^{2}$. They also constructed complete steady gradient K\"ahler-Ricci solitons on $\mathbb{C}^n$ with more general symmetries for $n\ge 3$. However, it is unclear whether or not
their examples exhibit positive curvature in these dimensions, hence Conjecture \ref{Cao conj} remains open for $n\ge 3$.

\subsection{Main results}

\subsubsection{Existence of steady K\"ahler-Ricci solitons}

In our first result, we construct complete steady gradient K\"ahler-Ricci solitons on $\mathbb{C}^{n}$ with strictly positive curvature operator on real $(1,\,1)$-forms that are not $U(n)$-invariant. The condition of $\Rm>0$ on real $(1,\,1)$-forms in particular implies strictly positive sectional curvature. Thus, these steady solitons provide counterexamples to Conjecture \ref{Cao conj} in any dimension $n\ge2$.

\begin{mtheorem}\label{t: existence of new}
Let $n\ge2$ and set $c_n=\tfrac{1}{2n(n+1)}$. 
 Then for all $\alpha\in[0,c_n]$, there exists a $U(1)\times U(n-1)$-invariant complete steady gradient K\"ahler-Ricci soliton $(M,g,f,p)$ on $\mathbb{C}^n$ with strictly positive curvature operator on real $(1,\,1)$-forms such that $R(p)=1$ and the lowest sectional curvature at $p$ is equal to $\alpha$, where $R$ denotes the scalar curvature of the metric $g$.
\end{mtheorem}

The family of steady solitons in this theorem
interpolates between the product of the cigar soliton and Cao's $U(n-1)$-invariant steady K\"ahler-Ricci soliton ($\alpha=0$)
and Cao's $U(n)$-invariant steady K\"ahler-Ricci soliton ($\alpha=c_n$). These solitons serve as the K\"ahler analog of the $\mathbb Z_2\times O(n-1)$-symmetric $n$-dimensional Riemannian steady Ricci solitons from \cite{Lai2020_flying_wing,lai20223d}, and so we call them K\"ahler flying wings.  As demonstrated in Appendix \ref{s: appendix B} (cf.~Corollary \ref{noniso}), by comparing the respective soliton vector fields, we show that the steady solitons of Theorem \ref{t: existence of new} are not isometric to those of Apostolov-Cifarelli \cite{apostolov2023hamiltonian}. In light of the results from \cite{ChenZhu2005, DZ20, ni2005ancient}, the steady gradient K\"ahler-Ricci solitons of Theorem \ref{t: existence of new} are all collapsed, have zero asymptotic volume ratio, and have volume growth rate bounded below by $r^n$, where $n$ is the complex dimension of the underlying manifold.

More generally, constructing explicit examples of complete non-compact K\"ahler manifolds with strictly positive sectional curvature in higher dimensions has long been an important problem in K\"ahler geometry.
Surprisingly, not much progress was made until the mid-90s. In earlier work, Klembeck \cite{Kle77} constructed $U(n)$-invariant complete non-compact K\"ahler manifolds on $\mathbb{C}^n$ with strictly positive bisectional curvature $BK>0$. When $n=1$, the same example was also constructed by Hamilton \cite{cigar} independently and is known as the cigar soliton.
However, for $n\ge 2$, these examples do not have nonnegative 
sectional curvature 
\cite[Example 1]{WZ11}. To the best of our knowledge, Cao's $U(n)$-invariant expanding and steady gradient K\"ahler-Ricci solitons \cite{Cao1996, Cao1997jdg} are the first examples of complete non-compact K\"ahler manifolds in higher dimensions with strictly positive sectional curvature in the literature (see \cite{WZ11}). Wu-Zheng \cite{WZ11} systematically studied $U(n)$-invariant K\"ahler manifolds with strictly positive sectional curvature and provided more non-trivial examples with $U(n)$-symmetry, exotic volume growth, and scalar curvature decay. We refer the reader to \cite{WZ11, YZ13} for a more detailed historical account in this direction. Since positive curvature is an open condition, one may easily generate other K\"ahler manifolds with strictly positive sectional curvature via a small compact perturbation of the K\"ahler potential of the aforementioned metrics. Theorem \ref{t: existence of new} generalizes previous examples in \cite{apostolov2023hamiltonian, con-der} by giving positively curved sub-$U(n)$-symmetric examples in higher dimensions that cannot be obtained by small perturbations of the K\"ahler potential of Cao's examples.

\subsubsection{Non-almost diameter rigidity of $\mathbb{P}^{n}$}

In the course of the proof of Theorem \ref{t: existence of new}, we obtain a non-almost diameter rigidity result for $\mathbb{P}^{n}$. The classical Myers theorem states that a complete Riemannian manifold $(M,\,g)$ with $\Ric\ge (n-1)g$ must have diameter $\text{diam}(M,\,g)\le \pi$. Cheng \cite{Cheng1975} proved that $\text{diam}(M,\,g)= \pi$ if and only if $M$ is isometric to the standard sphere of radius $1$. It is natural to ask whether or not ``almost diameter'' rigidity holds, meaning whether or not a Riemannian manifold $(M,\,g)$ with $\Ric\ge (n-1)g$ having diameter close to $\pi$ is also close to the standard sphere of radius $1$ in a precise sense. However, the almost diameter rigidity does not hold even topologically. Indeed, there are metrics $g$ with $\Ric_{g}\ge (n-1)g$ and $\operatorname{diam}(g)\ge \pi-\varepsilon$ for arbitrarily small $\varepsilon>0$ on $\mathbb P^n$ for $n\ge2$ by Anderson \cite{An90}, 
and $S^k\times S^{n-k}$ for any $k\ge2$ and $n-k\ge 3$ by Otsu \cite{Ot91}. 
On the other hand, the almost diameter rigidity does hold under additional conditions: with a suitable negative sectional curvature lower bound, Perelman proved that it is homeomorphic to $S^n$ \cite{perelman1995diameter}.
Very recently, Ren-Rong \cite{RR2023} showed that the manifold must be $\delta$-bi-H\"older close to the standard sphere for any $\delta>0$
if local universal covers are sufficiently close to the Euclidean ball under a uniform scale; see also \cite{cheeger1995almost,cheeger1996shape,Maximo23} for further discussion of the almost rigidity assuming $\Ric_{g}\ge(n-1)g$. 
Finally, the almost diameter rigidity does not hold under the even stronger curvature condition $\Rm\ge1$. The third-named author constructed examples with $\Rm\ge1$ that are arbitrarily close in the Gromov-Hausdorff sense to an interval of length $\pi$ \cite{Lai2020_flying_wing}. 

In the K\"ahler case, Li-Wang \cite{LiWang2005} proved that a complete K\"ahler manifold with bisectional curvature $BK\ge 2$ must have diameter bounded above by that of half of the Fubini-Study metric, that is, $\frac{\pi}{2}$. Recently, Datar-Seshadri \cite{DS2023} improved a rigidity result of Liu-Yuan \cite{GY2018} (see also \cite{ TamYu12}) by showing that if the diameter is equal to $\frac{\pi}{2}$, then the K\"ahler manifold is holomorphically isometric to $\mathbb{P}^n$. The next natural step then is to investigate the almost diameter rigidity in the K\"ahler case. In contrast to the Riemannian case, the K\"ahler structure imposes extra restrictions on the geometry of the manifold.
By the results of Mori and Siu-Yau on the Frankel conjecture \cite{mori1979projective,siu1980compact}, the positive bisectional curvature assumption already guarantees that the closed K\"ahler manifold is biholomorphic to $\mathbb P^n$. It is therefore interesting to ask if closeness to maximal diameter implies closeness to $\mathbb{P}^{n}$ in a more restrictive way, say the Gromov-Hausdorff sense, under the curvature condition. As a byproduct of our construction, we exhibit a degeneration of a family of K\"ahler metrics with bisectional curvature $BK\ge 2$ to a one-dimensional interval with optimal diameter equal to $\frac{\pi}{2}$. This demonstrates that the almost diameter rigidity of $\mathbb{P}^{n}$ in the K\"ahler case does not hold in general without further assumptions. More precisely, we show that almost diameter rigidity of $\mathbb{P}^{n}$ does not hold under an even stronger curvature condition.

\begin{mtheorem}\label{t: GH convergence on CPn}
    Let $n\geq 1$. Then for all $\varepsilon>0$, there exists a $U(n)$-invariant K\"ahler metric $g$ on $\mathbb{P}^n$ having curvature operator $\Rm$ bounded below by $2$ on {\sl real} $(1,1)$-forms everywhere (in particular, the holomorphic bisectional curvature BK $\ge 2$) such that 
    \[
    d_{GH}\left((\mathbb{P}^n,d_g),[0,\tfrac{\pi}{2}]\right)\le\varepsilon.
    \]
\end{mtheorem}

In Theorem \ref{t: GH convergence on CPn}, it is clear by volume comparison that the volume of $g$ goes to $0$ as $\varepsilon\to0$. One may wonder if the almost diameter rigidity holds under an additional uniform volume lower bound. Indeed as mentioned in \cite{RR2023}, counterexamples in the Riemannian case with uniform volume lower bound were provided by Anderson \cite{An90} and Otsu \cite{Ot91} independently. In the K\"ahler case, we provide counterexamples in Corollary \ref{p: heat kernel} by finding a sequence of $U(n)$-invariant K\"ahler metrics on $\mathbb{P}^n$ with $\Rm\ge 2$ on {\sl real} $(1,1)$-forms with volumes uniformly bounded from below, whose diameters converge to $\frac{\pi}{2}$, but whose metrics are not close to half of the Fubini-Study metric in the Gromov-Hausdorff sense.

The precise definition of curvature operator $\Rm\ge2$ on real $(1,\,1)$-forms is given below in Definition \ref{equiv}. In particular, we show that this curvature condition implies that the holomorphic bisectional curvature $BK\ge 2$. The curvature condition $\Rm>0$
on complex-valued $(1,\,1)$-forms is also known as {\sl strictly positive complex curvature operator} and was studied in \cite{CaoChow1986, ChenZhu2005, Siu1980, WZ11}. 
Here, we study the curvature condition $\Rm>2$
on {\sl real} $(1,1)$-forms, and show that it is equivalent to the condition that the K\"ahler cone over the corresponding Sasaki manifold has $\Rm>0$ on real $(1,\,1)$-forms in the {\sl transverse} directions. Moreover, these K\"ahler cones can be smoothed out by expanding K\"ahler-Ricci solitons with $\Rm>0$ on real $(1,\,1)$-forms by a result of the second-named author and Deruelle \cite{con-der}.  Theorem \ref{t: GH convergence on CPn} therefore implies:
\begin{mcor}\label{cor exp}
    Let $n\geq2$. Then for all $\varepsilon>0$, there exists a $U(1)\times U(n-1)$-invariant expanding gradient K\"ahler-Ricci soliton on $\mathbb{C}^n$ with $\Rm>0$
    on {\sl real} $(1,1)$-forms whose link $(S^{2n-1},h)$ satisfies 
    \[
    d_{GH}\left((S^{2n-1},d_h),[0,\tfrac{\pi}{2}]\right)\le\varepsilon.
    \]
\end{mcor}

\subsection{Outline of proofs}
We begin by recalling some classical methods that have been used to construct Ricci solitons and explain their
relevance to our construction.

\noindent\textbf{ODE methods:}~If the metric is assumed to satisfy certain symmetries, then both the Einstein and Ricci soliton equation reduce to a family of ODEs. In the Riemannian case, assuming rotational symmetry, Hamilton constructed the two-dimensional cigar soliton \cite{cigar} and Bryant \cite{bryant} constructed the $n$-dimensional Bryant soliton for all $n\ge3$. Appleton \cite{appleton2017family} constructed four-dimensional $U(2)$-invariant, non-collapsed, non-K\"ahler steady solitons on the line bundles $\mathcal{O}_{\mathbb{P}^{1}}(k)$, $k>2$, over $\mathbb{P}^{1}$. 

In the K\"ahler case, assuming $U(n)$-symmetry, Cao \cite{Cao1996} and Koiso \cite{koiso1990rotationally} found shrinking K\"ahler-Ricci solitons on twisted
projective line bundles over $\mathbb{P}^{n-1}$
for $n\ge2$. Cao \cite{Cao1997jdg} also constructed a one-parameter family of complete expanding
K\"ahler-Ricci solitons on $\mathbb C^n$, and steady K\"ahler-Ricci solitons on $\mathbb C^n$ and on the blow-up of $\mathbb{C}^n/\mathbb{Z}_n$ at the origin. Feldman-Ilmanen-Knopf \cite{feldman2003rotationally} then constructed the corresponding blow-down shrinking K\"ahler-Ricci soliton; see \cite{calabi1979metriques, CaoHD, dancer2008steady, DancerWang2011, eguchi1978asymptotically,gastel2004family,ivey1994newsteady,pedersen1999quasi,wang2004kahler} for more examples using ODEs.

\noindent\textbf{Continuity methods:}~In \cite{Siepmann2013Thesis}, Siepmann used the continuity method to construct expanding K\"ahler-Ricci solitons
coming out of Ricci-flat K\"ahler cones. Deruelle \cite{De15} extended the continuity method to the Riemannian case by constructing expanding gradient Ricci solitons coming out of positively curved cones.
In \cite{con-der}, the second-named author and Deruelle extended the aforesaid work of Siepmann by
using the continuity method to construct expanding 
K\"ahler-Ricci solitons emanating from K\"ahler cones with $\Rm>0$ on real $(1,\,1)$-forms. The key ingredient in the continuity method is to show that a deformation of the cone metric can be lifted to a deformation of the expanding gradient Ricci soliton. This relies on the invertibility of the linearized operator of the expanding soliton equation, which is true under the assumption of suitable positive curvature conditions. Recently, Bamler-Chen \cite{bamler2023degree} developed a new continuity method using degree theory that only requires nonnegative scalar curvature. In this situation, the linearized operator is not necessarily invertible. We also refer the reader to the following related works for expanding solitons: \cite{cao2023complete,chan2023curvature,chodosh,CF16,CDS19,deruelle2022initial,lee2022three,lott2017note,schulze2013expanding,Siepmann2013Thesis}. 

\noindent\textbf{Collapsing methods:}~In \cite{Lai2020_flying_wing}, the third-named author observed an intricate relationship between expanding and steady gradient Ricci solitons. Namely, for any sequence of expanding Ricci solitons with collapsing asymptotic volume ratio, zooming in on the points with the highest curvature, a steady Ricci soliton can be observed. A major step is to find a sequence of cone metrics with collapsing asymptotic volume ratios which themselves can be lifted to expanding gradient Ricci solitons by the continuity method. Finding a collapsing sequence of cone metrics is equivalent to finding a sequence of metrics on the link of the cone with collapsing volumes. Depending on the curvature conditions of the expanding solitons, the collapsing links should also satisfy certain curvature conditions.

In the Riemannian case \cite{Lai2020_flying_wing}, the links need to satisfy $\Rm>1$ for the cones over them to be lifted to expanding Ricci solitons with $\Rm>0$ \cite{De15}. Here, we require the link to be a Sasaki metric on $S^{2n-1}$ over a K\"ahler metric on $\mathbb P^{n-1}$ with $\Rm>2$ on real $(1,\,1)$-forms. We show that this is equivalent to the corresponding cone having $\Rm>0$ on real $(1,\,1)$-forms. These cones can therefore be lifted to expanding gradient K\"ahler-Ricci solitons with $\Rm>0$ on real $(1,\,1)$-forms by work of the second-named author and Deruelle \cite{CD20}. In the following, we explain the construction of the desired metrics on $\mathbb P^{n-1}$ and outline the proof.

In Section 2, we include some preliminaries on K\"{a}hler and Sasaki geometry, as well as the $U(1)\times U(n-1)$-invariant K\"{a}hler cone metric on $\mathbb{C}^n$ induced by a $U(n-1)$-invariant metric on $\mathbb{P}^{n-1}$. The details of the curvature computations are contained in Appendix A.

In Section 3, we construct a sequence of smooth K\"ahler
metrics on $\mathbb{P}^{n-1}$ with $\Rm\ge2$ on real $(1,\,1)$-forms outside a small region that converge to the interval $[0,\tfrac{\pi}{2}]$ in the Gromov-Hausdorff sense. We do this by first writing the Fubini-Study metric on $\mathbb{P}^{n-1}$ as a doubly warped product over the interval $[0,\frac{\pi}{2}]$. Then, by scaling down the two warping functions, we obtain a sequence of singular metrics collapsing to the interval $[0,\frac{\pi}{2}]$ with singularities at $0$ and $\frac{\pi}{2}$, and satisfying $\Rm\ge2$ on real $(1,\,1)$-forms on the smooth part. 
By cutting off the conical singularity at $0$ on arbitrarily small scales and gluing back a portion of a suitable Cao's expanding soliton, we can approximate the singular metrics by smooth K\"ahler metrics on $\mathbb{P}^{n-1}$ with $\Rm\ge2$ outside of an arbitrarily small neighborhood of $0$ 
and with $\Rm>0$ everywhere.

In Section 4, we take a limit of Ricci flows starting from these approximating metrics, and obtain a Ricci flow starting from each singular metric which smooths out the singularities. 
We show that the Ricci flow satisfies $\Rm\ge 2$ everywhere at all positive times. This yields the almost non-rigidity result of Theorem \ref{t: GH convergence on CPn}.

In Section 5, we lift the sequence of K\"ahler metrics on $\mathbb P^{n-1}$ in Theorem \ref{t: GH convergence on CPn} 
to a sequence of expanding K\"ahler-Ricci solitons with $\Rm>0$ on real $(1,\,1)$-forms.
We will show that they converge to a steady soliton asymptotic to a two-dimensional sector of angle $\frac{\pi}{2}$, and the soliton splits off a cigar factor. Then, using the same interpolation construction of the third-named author from \cite{Lai2020_flying_wing}, we obtain the family of solitons of Theorem \ref{t: existence of new}. In Appendix \ref{s: appendix B}, we show that these solitons are not isometric to those of \cite{apostolov2023hamiltonian}.

\subsection*{Acknowledgements}
P.-Y.~C. is supported by EPSRC grant EP/T019824/1,  by the Yushan Young Fellow Program of the Ministry of Education (MOE) Taiwan (MOE-108-YSFMS-0004-012-P1), and by the NSTC grant 113-2115-M-007-014-MY2. 
R.C.~was supported by NSF grant DMS-1906466 during the preparation of this article and is currently supported by a Simons Travel Grant. Y.L.~is supported by NSF grant DMS-2203310 and EPSRC grant EP/T019824/1. She would also like to acknowledge the University of Warwick for its hospitality during the time some of this article was being completed. 
The authors wish to thank Vestislav Apostolov,
Charles Cifarelli, Man-Chun Lee, Lei Ni, Felix Schulze, and Peter Topping for helpful discussions. We also thank Guofang Wei for providing us with additional references on the almost rigidity, and Charles Cifarelli for providing us with the proofs of Lemma \ref{hello} and Proposition \ref{bye}. This project began when the authors met at the workshop entitled \emph{Ricci flow and related topics} at the University of Warwick in March 2023. The authors would like to express their deep gratitude to Peter Topping for organizing the workshop. The authors would also like to thank the anonymous referees for their valuable feedback and suggestions.

\section{Preliminaries}\label{prelim}

\subsection{K\"ahler cones and Sasaki metrics}\label{sasakii}
In this subsection, we recall several useful notions and definitions, in particular that of a K\"ahler cone and Sasaki metric. For basics in K\"ahler geometry, we refer the reader to Huybrechts \cite{huybrechts}. For a more comprehensive reference for Sasakian geometry, we refer the reader to Boyer-Galicki \cite{book:Boyer}.

We begin with:
\begin{defn}[Riemannian cone]
Let $(S,\,g)$ be a compact connected Riemannian manifold. The \emph{Riemannian cone} $C_{0}$
 with \emph{link} $S$ is defined to be $\R_{+} \times S$ with metric $g_{0} = dr^2 \oplus r^2 g$ up to isometry. The radius function $r$ is then characterized intrinsically as the distance from the tip in the metric completion. We normally identify $S$ with the level set $\{r=1\}$.
\end{defn}
Then we have:
\begin{defn}[K\"ahler cone]\label{defn: kahler cones}
A \emph{K{\"a}hler cone} is a Riemannian cone $(C_0,g_0)$ such that $g_0$ is K{\"a}hler, together with a choice of $g_0$-parallel complex structure $J_0$. This will in fact often be unique up to sign. We then have a K{\"a}hler form $\omega_0(X,Y) = g_0(J_0X,Y)$, and $\omega_0 = \frac{i}{2}\partial\bar{\partial} r^2$ with respect to $J_0$.
\end{defn}

The vector field $r\partial_{r}$ on a K\"ahler cone is real holomorphic and $\xi:=J_{0}r\partial_r$ is real holomorphic and Killing \cite[Appendix A]{Yau}. This latter vector field is known as the \emph{Reeb field}. The closure of its flow in the isometry group of the link of the cone generates the holomorphic isometric action of a real torus on $C_{0}$ that fixes the tip of the cone. We call a K{\"a}hler cone ``quasiregular'' if this action is an
$S^1$-action (and, in particular, ``regular'' if this $S^1$-action is free), and ``irregular'' if the action generated is that of a real torus of rank $>1$.

Given a K\"ahler cone $(C_{0},\,\omega_{0}=\frac{i}{2}\partial\bar{\partial}r^{2})$ with radius function $r$, it is true that
\begin{equation}\label{conemetric}
\omega_{0}=rdr\wedge\eta +\frac{1}{2}r^{2}d\eta,
\end{equation}
where
\begin{equation}\label{contact}
\eta=i(\bar{\partial}-\partial)\log r=d^{c}\log(r)
\end{equation}
restricts to a contact form on the link $\{r=1\}$ of $C_{0}$ and we define $d^{c}:=i(\bar{\partial}-\partial)$. Clearly, any K\"ahler cone metric on $L^{\times}$, the contraction of the zero section of a negative line bundle $L$ over a projective manifold with some positive multiple of the radial vector field equal to the Euler vector field on $L\setminus\{0\}$, is regular. In fact, as the following theorem states, this property characterises all regular K\"ahler cones.
\begin{theorem}[{\cite[Theorem 7.5.1]{book:Boyer}}]\label{regulars}
Let $(C_{0},\,\omega_{0})$ be a regular K\"ahler cone with K\"ahler cone metric $\omega_{0}=\frac{i}{2}\partial\bar{\partial}r^{2}$, radial function $r$, and radial vector field $r\partial_{r}$. Then:
\begin{enumerate}[label=\textnormal{(\roman{*})}, ref=(\roman{*})]
\item $C_{0}$ is biholomorphic to the blowdown $L^\times$ of the zero section of a negative line bundle $L$ over a projective manifold $D$, with $a\cdot r\partial_{r}$ equal to the Euler field on $L\setminus\{0\}$ for some $a>0$.
\item Let $p:L\to D$ denote the projection. Then, writing $\omega_{0}$ as in \eqref{conemetric}, we have that $\frac{1}{2}d\eta=p^{*}\omega^{T}$ for some K\"ahler form $\omega^{T}$ on $D$ with $[\omega^{T}]=2\pi a\cdot c_{1}(L^{*})$.
\end{enumerate}
\end{theorem}
\noindent Evidently, the flow generated by the vector fields $\{ r\partial_{r},\,\xi\}$ produces the standard $\mathbb{C}^{*}$-action on the fibers of $L$. We give the K\"ahler form $\omega^{T}$ a special name.

\begin{defn}
The K\"ahler form $\omega^{T}$ on $D$ from Theorem \ref{regulars}(ii) is called the \emph{transverse K\"ahler form} of $\omega_{0}$ on $D$, with the corresponding K\"ahler metric $g^{T}$ called the \emph{transverse K\"ahler metric}. 
\end{defn}
\noindent In light of \eqref{conemetric}, the K\"ahler cone metric $g_{0}$ associated to $\omega_{0}$ takes the form
\begin{equation}\label{metricc}
g_{0}=dr^{2}+r^{2}(\eta^{2}+p^{*}g^{T}).
\end{equation}

As the next example demonstrates, Theorem \ref{regulars} is reversible in that one can always endow $L^\times$ with the structure of a regular K\"ahler cone metric.

\begin{eg}[{\cite[Theorem 7.5.2]{book:Boyer}}]\label{regular-eg}
Let $L$ be a negative line bundle over a projective manifold $D$ and let $p:L\to D$ denote the projection. Then $L$ has a hermitian metric $h$ with negative curvature. Locally, $h$ is defined by a smooth nonnegative real-valued function which, by abuse of notation, we also write as $h$. This is just the norm with respect to $h$ of the unit section in a local trivialization of $L$. The real $(1,\,1)$-form $i\partial\bar{\partial}\log h$ then defines a (global) K\"ahler form on $D$. 

Set $r^{2}=(h|w|^{2})^{a}>0$ for any $a>0$, with $w$ the coordinate on the fiber. Then $\frac{r^{2}}{2}$ defines the K\"ahler potential of a K\"ahler cone metric on $L^{\times}$, the contraction of the zero section of $L$, with K\"ahler form $\omega_{0}=\frac{i}{2}\partial\bar{\partial}r^{2}$ and radial vector field $r\partial_{r}$ a scaling of the Euler vector field on $L\setminus\{0\}$ by $\frac{1}{a}$.
Finally, writing $\omega_{0}$ as in \eqref{conemetric}, we have that $\frac{1}{2}d\eta=p^{*}\omega^{T}$, where $\omega^{T}=a\cdot i\partial\bar{\partial}\log h$ is the transverse K\"ahler form, a K\"ahler form on $D$ with $[\omega^{T}]=2\pi a\cdot c_{1}(L^{*})$.
\end{eg}

As a specific example, we have:
\begin{eg}\label{flatt}
In Example \ref{regular-eg}, one can consider the holomorphic line bundle \linebreak $\pi:\mathcal{O}_{\mathbb{P}^{n-1}}(-1)\to\mathbb{P}^{n-1}$ endowed with the hermitian metric $h$ whose curvature form is $-\omega_{FS}$, that is, negative the Fubini-Study metric on $\mathbb{P}^{n-1}$ \cite[Example 4.3.12]{huybrechts}. For any $a>0$, consider the K\"ahler cone metric defined by 
$r^{2}=(h|w|^{2})^{a}$, with $w$ the coordinate on the fiber. The corresponding K\"ahler cone via the usual identification of the blowdown of the zero section of $\mathcal{O}_{\mathbb{P}^{n-1}}(-1)$
with $\mathbb{C}^{n}$ resulting from the construction is $\mathbb{C}^{n}$ endowed with the K\"ahler cone metric $\omega_{0}=\frac{i}{2}\partial\bar{\partial}r^{2}=\frac{i}{2}\partial\bar{\partial}(|z|^{2a})$, with the Reeb vector field $\xi$ a scaling by $\frac{1}{a}$ of that whose flow rotates the Hopf fibers with period $2\pi$. In this case, the transverse K\"ahler form is given by $\frac{a}{2}\cdot\omega_{FS}$. Clearly, when $a=1$, we obtain the flat metric on $\mathbb{C}^{n}$.
\end{eg}

One may deform a K\"ahler cone to generate more 
examples in the following way.

\begin{defn}[Type II deformation]\label{d: type two deformation}
Let $(C_{0},\,\omega_{0}=\frac{i}{2}\partial\bar{\partial}r^{2})$ be a K\"ahler
cone with complex structure $J_{0}$
and let $\varphi:C_{0}\to\mathbb{R}$ be a smooth
real-valued function satisfying $\mathcal{L}_{r\partial_{r}}\varphi=\mathcal{L}_{J_{0}r\partial_{r}}\varphi=0$ with $\tilde{\omega}_{0}=\frac{i}{2}\partial\bar{\partial}(r^{2}e^{2\varphi})>0$. Then $(C_{0},\,\tilde{\omega}_{0})$, a K\"ahler cone   
with radius function $\tilde{r}:=re^{\varphi}$ and radial vector field $\tilde{r}\partial_{\tilde{r}}=r\partial_{r}$, is called a \emph{deformation of type II} (\emph{of $(C_{0},\,\omega_{0})$}).
\end{defn}
\noindent Let $\tilde{\eta}=i(\bar{\partial}-\partial)\log\tilde{r}$. Then by \eqref{conemetric}, $\tilde{\omega}_{0}$ may be written as
\begin{equation*}
\begin{split}
\tilde{\omega}_{0}&=\frac{i}{2}\partial\bar{\partial}(r^{2}e^{2\varphi})=\tilde{r}d\tilde{r}\wedge\tilde{\eta} +\frac{1}{2}\tilde{r}^{2}d\tilde{\eta}=\tilde{r}d\tilde{r}\wedge\tilde{\eta} +\frac{1}{2}\tilde{r}^{2}(d\eta+i\partial\bar{\partial}\varphi).
\end{split}
\end{equation*}
We refer the reader to \cite[Section 7.5.1]{book:Boyer} or \cite[Proposition 4.2]{FOW09} for more details.

The link of a K\"ahler cone is called a ``Sasaki'' manifold.

\begin{defn}[Sasaki manifolds]\label{Sasakii}
A compact (odd real dimensional) Riemannian manifold $(S,\,g)$ is \emph{Sasaki} if and only if the Riemannian cone over $(S,\,g)$ is a K\"ahler cone.
\end{defn}

A Sasaki manifold $(S,\,g)$ is naturally a contact manifold with contact form $\eta$ and Reeb vector field given by the restriction of \eqref{contact} and $\xi$ to $\{r=1\}\cong S$, respectively.
We will only consider \emph{regular} Sasaki manifolds, i.e., those Sasaki manifolds for which the corresponding K\"ahler cone $C_{0}$ is regular. Then via Theorem \ref{regulars}(ii), we have a map $p:C_{0}\to D$ onto a compact K\"ahler manifold $(D,\,g^{T})$, where $g^{T}$ is the transverse K\"ahler metric whose K\"ahler form $\omega^{T}$ satisfies $\frac{1}{2}d\eta=p^{*}\omega^{T}$. Moreover, restricting $p$ to $S\cong\{r=1\}$, we get a map $p:(S,\,g)\to(D,\,g^{T})$ which is a Riemannian submersion, as $g=\eta^{2}+p^{*}g^{T}$ thanks to \eqref{metricc}. This realizes $S$ as the total space of an $S^{1}$-bundle over $D$, the fibers of which are precisely the orbits of the flow of $\xi$.

\begin{eg}\label{flat2}
In the K\"ahler cone described in Example \ref{flatt}, the corresponding Sasaki manifold $(S,\,g)$ is the $(2n-1)$-sphere $S=S^{2n-1}$ with
$g$ the round metric of curvature $1$, $p:(S^{2n-1},\,g)\to(\mathbb{P}^{n-1},\,g^{T})$ is the Hopf fibration, and $g^{T}=\frac{1}{2}g_{FS}$
is the Fubini-Study metric on $\mathbb{P}^{n-1}$ normalized so that $\Ric(g_{FS})=ng_{FS}$.
The contact form $\eta$ in this case is the restriction of the one-form $d^{c}\log r$ to $S^{2n-1}\subset\mathbb{R}^{2n}$.
\end{eg}
For more details on Sasaki manifolds, we refer the reader to \cite{book:Boyer}.

\subsection{Doubly warped product metrics}\label{T3}

In this subsection, we introduce doubly warped product metrics and 
highlight the key features of such metrics that we need.

For $n\geq2$, let $(M,\,g,\,\eta,\,\xi)$ be a (possibly irregular) $(2n-1)$-dimensional Sasaki manifold as defined in Definition \ref{Sasakii}, with Sasaki metric $g$, contact one-form $\eta$, Reeb vector field $\xi$, and transverse K\"ahler metric $g^{T}$. For a given connected open interval $I=(0,\,L)\subset\mathbb{R},\,L>0,$ we define on the real $2n$-dimensional manifold $\widehat{M}:=M\times I$ a doubly warped product Riemannian metric $\hat{g}$ by
\begin{equation} \label{cohomo in App}
\hat{g}:=ds^{2}+a^2(s)\eta^{2}+b^2(s)g^{T},
\end{equation}
where $s$ is the coordinate on $(0,\,L)$. Without loss of generality, we assume that $a(s),\,b(s)>0$. The archetypical example of this construction is the K\"ahler cone itself.

\begin{eg}\label{cone}
Set $a(s)=b(s)=s$ and $L=\infty$. Then one obtains on $\widehat{M}$ the K\"ahler cone over the Sasaki manifold $(M,\,g)$.
\end{eg}

We endow $\widehat{M}$ with a complex structure in the following way. Since $(M,\,g)$ is Sasaki, we know that the cone $C_{0}=M\times\mathbb{R}_{+}$ over $M$ is K\"ahler,
and is in particular a complex manifold with complex structure we denote by $J_{0}$. Let $r$ denote the coordinate on the $\mathbb{R}_{+}$-factor of $C_{0}$. We define a map $\phi:(0,\,L)\to(0,\,\infty)$ as the unique solution to the ODE:
\begin{equation}\label{oode}
\left\{ \begin{array}{ll}
a(s)\phi'(s)=\phi(s), & \\
\phi\left(\frac{L}{2}\right)=1. & \\
\end{array} \right.
\end{equation}
This is given explicitly by $\phi(s)=e^{\int_{\frac{L}{2}}^{s}\frac{du}{a(u)}}$. Then, since $\phi'(s)>0$ for $s>0$, $\phi$ defines a diffeomorphism onto its image. We define a map  $\Phi:\widehat{M}\to C_{0}$ by $\Phi(x,\,s)=(x,\,\phi(s))$. This is also a diffeomorphism onto its image, and so we define a complex structure on $\widehat{M}$ by $\widehat{J}:=\Phi^{*}J_{0}$.
By construction $\Phi_{*}(a(s)\partial_{s})=r\partial_{r}$, and so
$\widehat{J}(a(s)\partial_{s})=\xi$.

It turns out that $\hat{g}$ is hermitian with respect to $\widehat{J}$
with fundamental form given by 
\begin{equation*}
\widehat{\omega}=\hat{g}(\widehat{J}(\cdot),\,\cdot)
=a(s)ds\wedge\eta+b(s)^{2}\omega^{T}.
\end{equation*}
We give necessary and sufficient conditions for when $(\widehat{M},\,\hat{g},\,\widehat{J})$ is K\"ahler.

\begin{lem}\label{kahlerr}
$(\widehat{M},\,\hat{g},\,\widehat{J})$ is K\"ahler if and only if $a(s)=b(s)b'(s)$. If this is the case, then the K\"ahler form $\widehat{\omega}$ of 
$(\widehat{M},\,\hat{g},\,J_{0})$ is given by
$$\widehat{\omega}=dd^{c}\left(\frac{1}{2}\int_{\frac{L}{2}}^{s}\frac{b(u)^{2}}{a(u)}\,du\right),$$
where $d^{c}=i\left(\bar{\partial}-\partial\right)$.
\end{lem}

\begin{proof}
The fundamental form $\widehat{\omega}$ of $\hat{g}$ is given by
\begin{equation}\label{kahlerform}
\widehat{\omega}=a(s)ds\wedge\eta+b(s)^{2}\omega^{T}.
\end{equation}
Clearly $d\widehat{\omega}=0$ if and only if $a(s)=b(s)b'(s)$. Here we use the fact that $d\eta=2\omega^{T}$. 

Regarding the last statement, it is clear from \eqref{kahlerform} that
$$\widehat{\omega}=d\left(\frac{1}{2}b(s)^{2}\eta\right).$$
Now, in light of \eqref{oode}, we can write
$$\eta=d^{c}\log(s)=d^{c}\log(\phi(s))=\left(\frac{\phi'(s)}{\phi(s)}\right)d^{c}s=\frac{d^{c}s}{a(s)}.$$
The desired expression follows.
\end{proof}

This leads us on to our next examples. The first illustrates the fact that the Fubini-Study metric on $\mathbb{P}^{n}$ can be realized as a doubly warped product.

\begin{eg}[{\cite[p.17]{petersen}}]\label{fubini}
Let $(M,\,g,\,\eta,\,\xi)$ be the round Sasaki structure on
$S^{2n-1}$, set $a(s)=\sin(s)\cos(s)=\frac{1}{2}\sin(2s)$ and $b(s)=\sin(s)$, and let $I=\left(0,\,\frac{\pi}{2}\right)$. Then the doubly warped product metric closes up at $0$ by adding a point and at $\frac{\pi}{2}$ by adding a $\mathbb{P}^{n-1}$. The resulting metric is one half of the Fubini-Study metric on
$\widehat{M}=\mathbb{P}^{n}$.
In light of Lemma \ref{kahlerr}, the K\"ahler form $\omega_{FS}$ of the Fubini-Study metric may be written as
$$\frac{1}{2}\omega_{FS}=dd^{c}\left(\frac{1}{2}\int_{\frac{\pi}{4}}^{s}\tan(u)\,du\right)=\frac{1}{2}dd^{c}\left(\ln\sec(s)\right)\quad\textrm{on $M\times\left(0,\,\frac{\pi}{2}\right)$.}$$
\end{eg}

The next example yields complete doubly warped product K\"ahler metrics on $\mathbb{P}^{n}$.

\begin{eg}\label{projective}
Again working with the round Sasaki structure on the 
$(2n-1)$-sphere $(S^{2n-1},\,g)$, realized as an $S^{1}$-bundle
$p:(S^{2n-1},\,g)\to(\mathbb{P}^{n-1},\,g^{T})$ over $\mathbb{P}^{n-1}$ via the Hopf fibration, with transverse metric $g^{T}=\frac{1}{2}g_{FS}$, one half of the Fubini-Study metric $g_{FS}$ on $\mathbb{P}^{n-1}$ normalized so that $\operatorname{Ric}(g_{FS})=ng_{FS}$ (cf.~Example \ref{flat2}), the doubly warped product construction yields K\"ahler metrics on $(0,\,L)\times S^{2n-1}$ which close up at $0$ by adding a point and at $L$ by adjoining a $\mathbb{P}^{n-1}$ to give a complete K\"ahler metric on $\mathbb{P}^{n}$ if and only if $a$ and $b$ can be extended smoothly to $0$ and $L$ such that
\begin{equation}\label{e: smoothness for a}
    a^{(\textnormal{even})}(0)=0,\qquad a'(0)=1,\qquad a^{(\textnormal{even})}(L)=0,\qquad a'(L)=-1,
\end{equation}
and
\begin{equation}\label{e: smoothness for b}
    b^{(\textnormal{even})}(0)=0,\qquad b'(0)=1,\qquad b^{(\textnormal{odd})}(L)=0,\qquad b(L)>0.
    \end{equation}
(The smoothness condition at $s=L$ was also considered by Tran \cite[Lemma 3.3]{Tran2023}.) One can verify that 
one half of the Fubini-Study metric on $\mathbb{P}^{n}$, 
considered in Example \ref{fubini} where $L=\frac{\pi}{2}$, $a(s)=\frac{1}{2}\sin(2s)$, and $b(s)=\sin(s)$, satisfies these conditions. These conditions have the following inferences for the functions $a$ and $b$. Using \cite[Lemma 2.1]{milnor}, we can write them in the following way:
\begin{equation}\label{boundaryy}
\begin{split}
a(x)&=xg_{1}(x)\qquad\textrm{for $x\in[0,\,L)$, $g_{1}(0)=a'(0)=1$, and $g_{1}'(0)=\frac{1}{2}a''(0)=0$,}\\
a(x)&=(L-x)g_{2}(L-x)\qquad\textrm{for $x\in(0,\,L]$, $g_{2}(0)=-a'(L)=1$, and $g_{2}'(0)=\frac{1}{2}a''(L)=0$,} \\
b(x)&=xg_{3}(x)\qquad\textrm{for $x\in[0,\,L)$, and $g_{3}(0)=b'(0)=1$ and $g_{3}'(0)=\frac{1}{2}b''(0)=0$,}\\
b(x)&=b(L)+(L-x)g_{4}(L-x)\qquad\textrm{for $x\in(0,\,L]$,
$g_{4}(0)=-b'(L)=0$},\\
&\qquad\qquad\qquad\qquad\qquad\qquad\qquad\qquad\qquad\qquad\qquad\qquad\qquad\textrm{and $g_{4}'(0)=\frac{1}{2}b''(L)=0$,}\\
\end{split}
\end{equation}
where $g_{i}(x),\,i=1,\ldots,4,$ are smooth functions.

Assume that $a$ and $b$ are chosen so that $[0,\,L]\times S^{2n-1}$
is smooth at $0$. Let $t_i\to0$ be a sequence of positive numbers converging to zero. Then the rescaled metrics\linebreak $ds^{2}+(t_i^{-1}a(t_i s))^{2}\eta\otimes \eta+(t_i^{-1}b(t_i s))^{2}g^{T}$ converge smoothly locally to the Euclidean metric \linebreak $ds^{2}+s^{2}\eta\otimes \eta+s^{2}g^{T}$
on $\mathbb{C}^{n}$ (realized as a Riemannian cone over $S^{2n-1}$; cf.~\eqref{metricc}). The smooth convergence immediately implies that $a(0)=b(0)=0$ and $a'(0)=b'(0)=1$.
In fact, by rewriting the metric under the Cartesian product as in \cite[Chapter 1, Sections 3.4 and 4.3]{petersen}, one can see that the pair of conditions \eqref{e: smoothness for a} and \eqref{e: smoothness for b} are necessary and sufficient for smoothness at $0$ and $L$.
\end{eg}

\subsection{$U(1)\times U(n-1)$-K\"ahler cone metrics on $\mathbb{C}^{n}$ coming from $U(n-1)$-invariant metrics on $\mathbb{P}^{n-1}$}
In this section, we will demonstrate how to construct $U(1)\times U(n-1)$-K\"ahler cone metrics on $\mathbb{C}^{n}$ from the following data.
The main result of this section is Proposition \ref{iceland}.

Let $n\geq3$ and recall Example \ref{projective} in the case of the round $(2n-3)$-sphere $(S^{2n-3},\,g)$ realized as an $S^{1}$-bundle $p:(S^{2n-3},\,g)\to(\mathbb{P}^{n-2},\,g^{T})$ over $(\mathbb{P}^{n-2},\,g^{T})$ with $2g^{T}=g_{FS}$, the Fubini-Study metric on $\mathbb{P}^{n-2}$ normalized so that $\operatorname{Ric}(g_{FS})=(n-1)g_{FS}$. We consider doubly warped product K\"ahler metrics on $S^{2n-3}\times(0,\,L),\,L>0$, of the form
\begin{equation}\label{dwp}
\bar{g}=ds^{2}+\bar{a}(s)^{2}\eta^{2}+\bar{b}(s)^{2}g^{T}.
\end{equation}
Here $s$ is the coordinate on the interval $(0,\,L)$ and $\eta$ is the contact form associated to the round Sasaki structure on $S^{2n-3}$ (see Example \ref{flat2}). Write $\omega^{T}$ for the K\"ahler form associated to $g^{T}$. We have that 
$\bar{a}(s)=\bar{b}(s)\bar{b}'(s)$ by Lemma \ref{kahlerr} because $\bar{g}$ is K\"ahler, and in addition that \eqref{e: smoothness for a} and \eqref{e: smoothness for b} hold true so that, as described in Example \ref{projective}, $\bar{g}$ closes up smoothly at the endpoints of $(0,\,L)$ to give a warped product K\"ahler metric $\bar{g}$ on $\mathbb{P}^{n-1}$. Without loss of generality, we can assume that $\bar{a}(s)>0$. Then, as written in Lemma \ref{kahlerr}, the K\"ahler form $\bar{\omega}$ of $\bar{g}$ can be written as
$$\bar{\omega}=\bar{a}(s)ds\wedge\eta+\bar{b}(s)^{2}\omega^{T}=dd^{c}\left(\frac{1}{2}\int_{\frac{L}{2}}^{s}\frac{\bar{b}(u)^{2}}{\bar{a}(u)}\,du\right),$$
and the corresponding volume form can be computed as 
$$\bar{\omega}^{n-1}=(n-1)\cdot\bar{a}(s)\cdot\bar{b}(s)^{2n-4}ds\wedge\eta\wedge(\omega^{T})^{n-2}.$$
In particular,
\begin{equation}\label{coeff}
\begin{split}
\int_{\mathbb{P}^{n-1}}\bar{\omega}^{n-1}&=c(n)\int_{0}^{L}\bar{a}(s)\cdot\bar{b}(s)^{2n-4}\,ds=c(n)\int_{0}^{L}\bar{b}'(s)\cdot\bar{b}(s)^{2n-3}\,ds
=\frac{c(n)\bar{b}(L)^{2n-2}}{2n-2}>0,
\end{split}
\end{equation}
where $c(n)>0$ is a dimensional constant and where we have used \eqref{e: smoothness for b} and the fact that $\bar{a}(s)=\bar{b}(s)\bar{b}'(s)$.

We begin by proving some preliminary lemmas before stating and proving the main result.

\begin{lem}\label{ireland}
The map $\phi:(0,\,L)\to(0,\,\infty)$ defined by $$\phi(s)= e^{\int^{s}_{\frac{L}{2}}\frac{du}{\bar{a}(u)}}$$ is invertible.   
\end{lem}

\begin{proof}
Since $\bar{a}(s)>0$ by assumption, we see directly that $\phi'(s)>0$ for $s\in(0,\,L)$ so that $\phi(s)$ is strictly increasing. In addition, in light of \eqref{boundaryy}, 
with $|g_{2}(L-x)-1|\leq\frac{1}{2}$ for $x$ in an interval of the form $(0,\,\varepsilon)$ for some $\varepsilon>0$ sufficiently small, we see that for $s\in(L-\varepsilon,\,L)$,
$$\int_{\frac{L}{2}}^{s}\frac{du}{\bar{a}(u)}
=\int_{\frac{L}{2}}^{L-\varepsilon}\frac{du}{\bar{a}(u)}+\int_{L-\varepsilon}^{s}\frac{du}{\bar{a}(u)}\geq C+\frac{2}{3}
\int_{L-\varepsilon}^{s}\frac{du}{L-u}$$
which tends to $+\infty$ as $s\to L^{-}$. Similar behavior occurs as $s\to0^{+}$, hence $\phi(s)$ is indeed a diffeomorphism. 
\end{proof}

Let $(w_{1},\ldots,w_{n})$ henceforth denote holomorphic coordinates on $\mathbb{C}^{n}$. 

\begin{lem}\label{l: s bar}
Define a function $\hat{s}:\mathbb{C}^{n}\setminus\{0\}\to[0,\,L]$ by
\begin{equation*}
\hat{s} (w_{1},\ldots,w_{n})=\left\{ \begin{array}{ll}
 0 & \textrm{if $w_{2}=\ldots=w_{n}=0$ (and $w_{1}\neq0$)},\\
\phi^{-1}\left(\frac{\sqrt{\sum_{j=2}^{n}|w_{j}|^{2}}}{|w_{1}|}\right)
& \textrm{if $w_{1}\neq0$ and $w_j\neq 0$ for some $j\ge 2$},\\
L & \textrm{if $w_{1}=0$.}\\
\end{array} \right.
\end{equation*}
Then $\hat{s}$ is continuous and invariant under the diagonal $\mathbb{C}^{*}$-action on $\mathbb{C}^{n}$.
\end{lem}

\begin{proof}
The function $\hat{s}$ is clearly invariant under the diagonal $\mathbb{C}^{*}$-action on $\mathbb{C}^{n}$. As such, it descends to a function that we denote by $\tilde{s}:\mathbb{P}^{n-1}\to[0,\,L]$ on $\mathbb{P}^{n-1}$. To complete the proof of the lemma, it suffices to show that $\tilde s$ is continuous in the induced homogeneous coordinates $[w_{1}:\ldots:w_{n}]$ on $\mathbb{P}^{n-1}$.

To this end, since $\phi$ extends to a continuous function with $\phi(0)=0$, 
$\tilde{s}$ is clearly continuous on the open subset $\{[w_{1}:w_{2}:\ldots:w_{n}]\in\mathbb{P}^{n-1}\,|\,w_{1}\neq0\}$.
Let $2\leq j\leq n$ and consider the chart
$\{[w_{1}:w_{2}:\ldots:w_{n}]\in\mathbb{P}^{n-1}\,|\,w_{j}\neq0\}$.
In this chart, we see that for $w_{1}\neq0$,
$$\tilde{s}([w_{1}:\ldots:w_{n}])=\phi^{-1}\left(\left|\frac{w_{j}}{w_{1}}\right|\sqrt{1+\sum_{\substack{k=2 \\ k\neq j}}^{n}|w_{k}|^{2}}\right).$$
Since the quantity inside the parentheses tends to $+\infty$ as $w_{1}\to0$, we see that $\tilde s$ is continuous, as required.
\end{proof}

\begin{lem}\label{uk}
Define a function $\widehat{\varphi}:\mathbb{C}^{n}\setminus\{0\}\to\mathbb{R}$ by
$$\widehat{\varphi}(w)=2\int_{\frac{L}{2}}^{\hat{s}(w)}\frac{\bar{b}(u)^{2}}{\bar{a}(u)}\,du-\bar{b}(L)^{2}\cdot\ln\left(1+e^{2\int_{\frac{L}{2}}^{\hat{s}(w)}\frac{du}{\bar{a}(u)}}\right),$$
where $\hat{s}$ is the function defined in Lemma \ref{l: s bar}. Then $\widehat{\varphi}$ is smooth and invariant under the diagonal $\mathbb{C}^{*}$-action on $\mathbb{C}^{n}\setminus\{0\}$.   
\end{lem}

\begin{proof}
Since $\widehat{\varphi}$ is a function of $\hat{s}$ only, and $\hat{s}$ is invariant under the diagonal $\mathbb{C}^{*}$-action, $\widehat{\varphi}$ has the same property. As such, it descends to a function $\tilde{\varphi}=\tilde{\varphi}(\tilde{s}):\mathbb{P}^{n-1}\to\mathbb{R}$, where $\tilde{s}$ is as in the proof of the previous lemma. 
As we shall soon see, after a suitable reparametrization, 
$i\partial \bar \partial\tilde{\varphi}$ encodes the difference between
$\bar\omega$ and a suitable multiple of $\omega_{FS}$.
To complete the proof of the lemma, it suffices to show that $\tilde{\varphi}$ is smooth in the induced homogeneous coordinates $[w_{1}:\ldots:w_{n}]$ on $\mathbb{P}^{n-1}$.

To this end, with $s$ still denoting the coordinate on the $(0,\,L)$-factor, let $\tau:(0,\,L)\to\left(0,\,\frac{\pi}{2}\right)$ be the unique solution of the ODE 
\begin{displaymath}
\left\{ \begin{array}{ll}
\bar{a}(s)\tau'(s)=\frac{1}{2}\sin(2\tau(s)), & \\
\tau\left(\frac{L}{2}\right)=\frac{\pi}{4}. & \\
\end{array} \right.
\end{displaymath}
Then $$\tau(s)=\arctan\left(e^{\int_{\frac{L}{2}}^{s}\frac{du}{\bar{a}(u)}}\right)=\arctan(\phi(s))$$ and hence is a diffeomorphism. 

Next recall from Example \ref{fubini} one half of the Fubini-Study metric on $S^{2n-3}\times\left(0,\,\frac{\pi}{2}\right)$, the K\"ahler form of which is given by 
\begin{equation}\label{tired}
\frac{1}{2}\omega_{FS}=dd^{c}\left(\frac{1}{2}\int_{\frac{\pi}{4}}^{t}\tan(u)\,du\right)=\frac{1}{2}dd^{c}\left(\ln\sec(t)\right)
\end{equation}
with $t$ the coordinate on the interval factor. $\tau(s)$
has the additional property that $\tau_{*}(\bar{a}(s)\partial_{s})=\frac{1}{2}\sin(2t)\partial_{t}$ and so we have an induced biholomorphism 
\begin{equation}\label{beta}
\begin{split}
T&:(S^{2n-3}\times\left(0,\,L\right),\,\bar{\omega})\to\left(S^{2n-3}\times \left(0,\,\frac{\pi}{2}\right),\,\frac{1}{2}\omega_{FS}\right),\\
 T(x,\,s)&=(x,\,\tau(s))=(x,\,\arctan(\phi(s))),
\end{split}
\end{equation}
which extends over the endpoints in the obvious way to give an automorphism of $\mathbb{P}^{n-1}$:
\begin{equation*}
\begin{split}
T&:\left(\left(S^{2n-3}\times\left(0,\,L\right)\right)\cup\{0\}\cup\left(\mathbb{P}^{n-2}\times\left\{L\right\}\right),\,\bar{\omega}\right)\\
&\qquad\qquad\qquad\to\left(\left(S^{2n-3}\times \left(0,\,\frac{\pi}{2}\right)\right)\cup\{0\}\cup\left(\mathbb{P}^{n-2}\times\left\{\frac{\pi}{2}\right\}\right),\,\frac{1}{2}\omega_{FS}\right).
\end{split}
\end{equation*}
Notice that $T^{*}t=\tau(s)=\arctan(\phi(s))$ and that  
\begin{equation*}
\begin{split}
T^{*}\left(\frac{1}{2}\omega_{FS}\right)&=dd^{c}\left(\frac{1}{2}\int_{\frac{\pi}{4}}^{\tau(s)}\tan(u)\,du\right)=\frac{1}{2}dd^{c}\left(\ln\sec(\tau(s))\right)\\
&=\frac{1}{4}dd^{c}\left(\ln\left(1+\left(e^{\int_{\frac{L}{2}}^{s}\frac{du}{\bar{a}(u)}}\right)^{2}\right)\right)=\frac{1}{4}dd^{c}\left(\ln\left[1+(\phi(s))^{2}\right)\right].
\end{split}
\end{equation*}

Now $\bar{\omega}$ lies in a positive multiple of the K\"ahler class of $T^{*}\left(\frac{1}{2}\omega_{FS}\right)$, and so there exists a smooth real-valued function $\varphi:\mathbb{P}^{n-1}\to\mathbb{R}$, defined up to addition of a constant, and $c>0$ such that $\bar{\omega}-cT^{*}\left(\frac{1}{2}\omega_{FS}\right)=i\partial\bar{\partial}\varphi$.
In light of \eqref{coeff}, we see that $c=\bar{b}(L)^{2}$.
On $S^{2n-3}\times(0,\,L)$, we can write
\begin{equation}\label{lovely}
\begin{split}
  \frac{1}{2}dd^{c}\varphi&=i\partial\bar{\partial}\varphi=\bar{\omega}-\bar{b}(L)^{2}T^{*}\left(\frac{1}{2}\omega_{FS}\right)\\
   &=\frac{1}{2}dd^{c}\left(\int_{\frac{L}{2}}^{s}\frac{\bar{b}(u)^{2}}{\bar{a}(u)}\,du-\frac{\bar{b}(L)^{2}}{2}\ln\left(1+e^{2\int_{\frac{L}{2}}^{s}\frac{du}{\bar{a}(u)}}\right)\right).
\end{split}
\end{equation}
Using the expansions of \eqref{boundaryy}, it is easy to show that the function inside the parentheses on the right-hand side
has a $C^{2}$-continuation over $s=0$, i.e., the function and its first and second derivatives with respect to $s$ extend as continuous functions over $s=0$. The function also admits a $C^{2}$-continuation over $s=L$. To see this, again use the expansions of \eqref{boundaryy} and rewrite the function in the following way:
\begin{equation*}
\begin{split}
\int_{\frac{L}{2}}^{s}\frac{\bar{b}(u)^{2}}{\bar{a}(u)}\,du&-\frac{\bar{b}(L)^{2}}{2}\ln\left(1+e^{2\int_{\frac{L}{2}}^{s}\frac{du}{\bar{a}(u)}}\right)\\
&\qquad\qquad\qquad\qquad=\int_{\frac{L}{2}}^{s}\left(
\frac{\bar{b}(u)^{2}-\bar{b}(L)^{2}}{\bar{a}(u)}\right)\,du
-\frac{\bar{b}(L)^{2}}{2}\ln\left(1+
e^{-2\int_{\frac{L}{2}}^{s}\frac{du}{\bar{a}(u)}}\right).
\end{split}
\end{equation*}
In addition, its first derivative with respect to $s$ vanishes at $s=0,\,L$, and so the gradient and Laplacian of this function both extend as continuous functions to the whole of $\mathbb{P}^{n-1}$. With the Laplacian of this function vanishing, an integration by parts argument now implies that
\begin{equation}\label{lost}
\varphi-2\int_{\frac{L}{2}}^{s}\frac{\bar{b}(u)^{2}}{\bar{a}(u)}\,du+\bar{b}(L)^{2}\ln\left(1+e^{2\int_{\frac{L}{2}}^{s}\frac{du}{\bar{a}(u)}}\right)\qquad\textrm{is constant.}
\end{equation}
In particular, 
\begin{equation*}
2\int_{\frac{L}{2}}^{s}\frac{\bar{b}(u)^{2}}{\bar{a}(u)}\,du-\bar{b}(L)^{2}\ln\left(1+e^{2\int_{\frac{L}{2}}^{s}\frac{du}{\bar{a}(u)}}\right)
\end{equation*}
is smooth on $\mathbb{P}^{n-1}$ because $\varphi$ 
is.

We next define a biholomorphism
$$\Psi:\left(S^{2n-3}\times\left(0,\,\frac{\pi}{2}\right)\right)\cup\{0\}\to\{[w_{1}:w_{2}:\ldots:w_{n}]\in\mathbb{P}^{n-1}\,|\,w_{1}\neq0\}$$ in the following way.
Choose an automorphism of the round Sasaki structure
$$\lambda:S^{2n-3}\to\{z=(z_{1},\ldots,z_{n-1})\in\mathbb{C}^{n-1}\,|\,|z|=1\}$$ and set
\begin{equation*}
\left\{ \begin{array}{ll}
\Psi(0)=[1:0:\ldots:0], & \\
\Psi(x,\,t)=[1:z_{1}:\ldots:z_{n-1}]=[1:\lambda(x)\cdot\tan(t)] & \textrm{for $t>0$ and $x\in S^{2n-3}$.}\\
\end{array} \right.
\end{equation*}
Note that $\psi:\left(0,\,\frac{\pi}{2}\right)\to(0,\,\infty)$,
$\psi(t):=\tan(t),$ is a diffeomorphism so that $\Psi$ is one also. 
Moreover, $\psi(t)$ is the unique solution of the ODE
\begin{displaymath}
\left\{ \begin{array}{ll}
\frac{1}{2}\sin(2t)\psi'(t)=\psi(t), & \\
\psi(\frac{\pi}{4})=1, & \\
\end{array} \right.
\end{displaymath}
hence $\Psi$ in addition satisfies $\Psi_{*}\left(\frac{1}{2}\sin(2t)\partial_{t}\right)=r\partial_{r}$, where $r(z)^{2}=\sum_{i=1}^{n-1}|z_{i}|^{2}$. Thus, $\Psi$ is indeed a biholomorphism.
Notice by construction that $\Psi^{*}r=\psi\circ t=\tan(t)$ and so in light of \eqref{tired},
$\Psi$ extends in the obvious way to a holomorphic isometry $$\Psi:\left(S^{2n-3}\times\left(0,\,\frac{\pi}{2}\right)\right)\cup\{0\}\cup\left(\mathbb{P}^{n-2}\times\left\{\frac{\pi}{2}\right\}\right)\to\{[w_{1}:\ldots:w_{n}]\in\mathbb{P}^{n-1}\}$$
of $\mathbb{P}^{n-1}$ with respect to the Fubini-Study metric. 

Recall that $\Psi^{*}r=\tan(t)$ and that $T^{*}t=\tau(s)=\arctan(\phi(s))$.
Set
$$\Phi:=\Psi\circ T:\left(S^{2n-3}\times\left(0,\,L\right)\right)\cup\{0\}\cup\left(\mathbb{P}^{n-2}\times\left\{L\right\}\right)\to\{[w_{1}:w_{2}:\ldots:w_{n}]\in\mathbb{P}^{n-1}\}.$$
Explicitly, $\Phi$ is given by
\begin{equation}\label{fuck}
\left\{ \begin{array}{ll}
\Phi(0)=[1:0:\ldots:0], & \\
\Phi(x,\,s)=\left[1:\lambda(x)\cdot\phi(s)]=[1:\lambda(x)\cdot e^{\int^{s}_{\frac{L}{2}}\frac{du}{\bar{a}(u)}}\right]& \textrm{for $s>0$ and $x\in S^{2n-3}$.}\\
\end{array} \right.
\end{equation}
Since $\Phi^{*}r=\phi(s)$, we see directly that
$\Phi^{*}\tilde{s}=s$. Thus, in light of \eqref{lost}, we find that
\begin{equation}\label{fck}
\begin{split}
i\partial\bar{\partial}\left(\varphi-\Phi^{*}\tilde{\varphi}(\tilde{s})\right)=
i\partial\bar{\partial}\left(\varphi-\tilde{\varphi}(s)\right)=0,
\end{split}
\end{equation}
and so $\tilde{\varphi}$ is smooth, as required.
\end{proof}

The main result of this subsection is the following.
 
\begin{prop}\label{iceland}
Let $n\geq3$. 

{\rm (1)} Let $\bar{g}$ be a K\"ahler metric on $\mathbb P^{n-1}$ with K\"ahler form $\bar{\omega}$
invariant under the $U(n-1)$-action defined by matrix multiplication on the last $(n-1)$-homogeneous coordinates on $\mathbb P^{n-1}$. Then there exists $L>0$,
a smooth function $s:\mathbb{P}^{n-1}\to[0,\,L]$, and smooth functions $\bar{a},\bar{b}:[0,\,L]\to[0,\infty)$
with $\bar{a}(s)=\bar{b}(s)\bar{b}'(s)$
satisfying \eqref{e: smoothness for a} and \eqref{e: smoothness for b} such that $\bar{g}$ can be written as a doubly warped product 
\begin{equation*}
   \bar{g}=ds^{2}+\bar{a}^2(s)\eta^{2}+\bar{b}^2(s)g^{T} 
\end{equation*}
on $S^{2n-3}\times(0,\,L)$. Here, $(S^{2n-3},\,g^{T},\,\eta)$ is the standard round Sasaki structure on $S^{2n-3}$ as described in Example \ref{flat2}.

{\rm (2)} Let $\widehat{\varphi}:\mathbb{C}^{n}\setminus\{0\}\to\mathbb{R}$ be as in Lemma \ref{uk} and let $\hat{g}$ be a K\"ahler cone metric on $\mathbb C^n\setminus\{0\}$ defined by the radial function $\hat{r}(w)=|w|^{\bar{b}(L)^2}e^{\widehat{\varphi}(w)}$ (cf.~Definition \ref{d: type two deformation}).
Let $\widehat{\omega}^{T}$ denote the transverse K\"ahler form on $\mathbb{P}^{n-1}$. Then:

 {\rm (a)} $\hat{g}$ is invariant under the $U(1)\times U(n-1)$-action on $\mathbb{C}^{n}$ defined by  
$$U(1)\times U(n-1)\ni(\alpha,\,A)\cdot(z,\,w)\in\mathbb{C}\times\mathbb{C}^{n-1}\mapsto(\alpha\cdot z,\,Aw).$$ 

{\rm (b)} The induced Sasaki metric $h$ on $\{\hat{r}=1\}\cong S^{2n-1}$ 
satisfies
\begin{equation}\label{diam sasa}
    \operatorname{diam}(\mathbb{P}^{n-1}, \hat{g}^T)\le\operatorname{diam}(S^{2n-1}, h)\leq  \operatorname{diam}(\mathbb{P}^{n-1}, \,\hat{g}^T)+\pi \bar{b}(L)^2
\end{equation}
and
\begin{equation}\label{vol sasa}
    \operatorname{vol}_{h}(S^{2n-1})\leq 2\pi\bar{b}(L)^2 \operatorname{vol}_{\hat{g}^{T}}(\mathbb{P}^{n-1}).
\end{equation}

{\rm (c)} $\widehat{\omega}^{T}$ depends smoothly on $\bar{g}$ and there exists a biholomorphism $\Phi$ such that $\widehat{\omega}^{T}=(\Phi^{-1})^*\bar{\omega}$.
\end{prop}

\begin{remark}\textnormal{Since $\widehat{\varphi}$ is smooth and invariant under the diagonal $\mathbb{C}^{*}$-action on $\mathbb{C}^{n}\setminus\{0\}$ by Lemma \ref{uk}, it descends to a well-defined smooth real-valued function $\varphi$ on projective space. The fact that $\hat{r}$ does indeed define the radial function of a K\"ahler cone metric on $\mathbb{C}^{n}\setminus\{0\}$ then follows from item (2c) of the proposition}.
\end{remark}

\begin{remark} 
\textnormal{A version of Proposition \ref{iceland} also holds when 
$n=2$. In this case, the K\"{a}hler metric $\bar{g}$ 
is a warped product on $(0,\,L)\times S^1$ of the form
$$\bar{g}=ds^2+\bar{a}^2(s)\eta^{2}$$
and $\bar{a}(s)$ satisfies \eqref{e: smoothness for a}.
We define another function $\bar{b}(s):=\sqrt{2\int_0^{s}\bar{a}(\tau)\,d\tau}$. 
Then an induction argument shows that $\bar{b}(s)$ satisfies
\eqref{e: smoothness for b}. With this, the proof of Lemmas \ref{kahlerr}, \ref{ireland}, \ref{l: s bar}, and \ref{uk}, as well as that of Proposition \ref{iceland}, can be carried through. We leave the details to the interested reader.}
\end{remark}

\begin{proof}
{\rm (1)} We begin by showing that $\bar{g}$ takes the form of a doubly warped product, as stated. Many of the ideas we use come from \cite{Cao1996}.

First recall from Example \ref{flat2} the round Sasaki structure on $S^{2n-3}$ with contact one-form $\eta$ and transverse K\"ahler form $\omega^{T}=\frac{1}{2}\omega_{FS}$, one half of the Fubini-Study metric $g_{FS}$ on $\mathbb{P}^{n-2}$ normalized so that $\operatorname{Ric}(g_{FS})=(n-1)g_{FS}$. 
Let $z=[z_{0}:\ldots:z_{n-1}]$ denote homogeneous
coordinates on $\mathbb{P}^{n-1}$. The fact that $\bar{g}$
is invariant under the stipulated action means that
the restriction of $\bar{g}$ to 
the open set $\{z_{0}\neq0\}\subseteq\mathbb{P}^{n-1}$
is invariant under the standard $U(n-1)$-action on the 
holomorphic coordinates $(z_{1},\ldots,z_{n-1})$
induced from the standard coordinate chart 
covering the aforesaid open set. Let $|z|=:r:\mathbb{C}^{n-1}\to\mathbb{R}$ and write $r^{2}=e^{t}$. Being $U(n-1)$-invariant, we know that $\bar{\omega}=\frac{i}{2}\partial\bar{\partial}\Phi(t)$ for $\Phi:\mathbb{R}\to\mathbb{R}$ a given smooth function. Thus, recalling that $d^{c}:=i(\bar{\partial}-\partial)$ and $\eta=d^{c}\log(r)=\frac{1}{2}d^{c}t$, we can write
\begin{equation*}
\begin{split}
\bar{\omega}&=\frac{i}{2}\partial\bar{\partial}\Phi(t)=\frac{1}{4}dd^{c}\Phi(t)=\Phi'(t)\frac{1}{4}dd^{c}t+\Phi''(t)
\frac{dt}{2}\wedge\frac{d^{c}t}{2}=\mu(t)\omega^{T}+\mu'(t)\frac{dt}{2}\wedge\eta,
\end{split}
\end{equation*}
where $\mu(t):=\Phi'(t)$ is a smooth function on $(-\infty,\,\infty)$. 

In defining a metric, we clearly have that both $\mu(t),\,\mu'(t)>0$. We claim that $\bar{\omega}$ is a doubly warped product K\"ahler metric of the desired form with $L=\diam(\mathbb{P}^{n-1},\,\bar{g})$. To see this, set
$$s(t):=\frac{1}{2}\int_{-\infty}^t\sqrt{\mu'(\tau)}\,d\tau\qquad\textrm{for all $t\in\mathbb{R}$}.$$ Then, since $\bar{\omega}$ is in particular rotationally symmetric, the Euclidean radial line emerging from the origin is also a minimizing geodesic. Hence $s(t)$ is well-defined as it is the distance from the origin to any given $z\in \mathbb{C}^{n-1}$ with $t=2\ln|z|$. 

Next observe that $s'(t)=\frac{\sqrt{\mu'(t)}}{2}>0$ and so
$t$ can be written as a function of $s$. Let 
$$s_1:=\lim_{t\to-\infty}s(t)\qquad\text{and}\qquad s_2:=\lim_{t\to\infty}s(t).$$
Then $\bar{g}$ is a warped product metric over the interval $(s_{1},\,s_{2})$ with 
$$\bar{a}(s):=\sqrt{\mu'(t(s))}\qquad\textrm{and}\qquad\bar{b}(s):=\sqrt{\mu(t(s))}.$$ 
We claim that it is of the desired form. Indeed, for all $s\in (s_{1},\,s_{2})$, we have that
$$\bar{b}(s)\bar{b}'(s)=\frac{\mu'(t(s))}{2}\frac{dt}{ds}=\frac{\mu'(t(s))}{2}\frac{2}{\sqrt{\mu'(t(s))}}=\bar{a}(s)$$
and
$$\bar{\omega}=\mu(t)\omega^{T}+\mu'(t)\frac{dt}{2}\wedge\eta=\bar{b}(s)^2\omega^{T}+\bar{a}(s)ds\wedge\eta.$$
By the completeness of the metric, we see that
$$s_{1}=\lim_{t\to-\infty}s(t)=0
\qquad\textrm{and}\qquad s_{2}=\lim_{t\to \infty}s(t)=\operatorname{diam}(\mathbb{P}^{n-1},\,\bar{g})=L,$$
and by the arguments as laid out in Example \ref{projective} that \eqref{e: smoothness for a} and \eqref{e: smoothness for b} hold true, as required.

We now prove the remaining items of the proposition. 

 {\rm (2a)} Recall the definition of $\hat{s}$ from Lemma \ref{ireland}. Since $\widehat{\varphi}$ depends only on $\hat{s}$ which, by its very definition, is invariant under the action prescribed, it is clear that $\widehat{\varphi}$, and hence $\hat{r}$, is also. 

 {\rm (2b)} There exists a Riemannian submersion $p:(S^{2n-1}, h)\mapsto (\mathbb{P}^{n-1}, \hat g^T)$ which is an $S^1$-bundle with $S^1$-fiber parametrized by the flow of the Reeb vector field $\frac{\xi}{\bar{b}(L)^2}$. Moreover, the Sasaki metric $h$ on $\{\hat r=1\}$ is given by
$$h=( \bar{b}(L)^2\eta+i(\bar{\partial}-\partial)\widehat{\varphi})\otimes( \bar{b}(L)^2\eta+i(\bar{\partial}-\partial)\widehat{\varphi})+p^{*}\hat g^{T},$$
where $\eta$ and $\xi$ are the standard contact form and Reeb vector field of the flat K\"ahler cone $\mathbb{C}^n\setminus \{0\}$. Using the fact that 
$$[i(\bar{\partial}-\partial)\widehat{\varphi}](\xi)
=d^{c}\widehat{\varphi}(\xi)=-\mathcal{L}_{J\xi}\widehat{\varphi}
-d^{c}(\xi\cdot\widehat{\varphi})=0$$ 
because $\widehat{\varphi}$ is invariant under the diagonal $\mathbb{C}^{*}$-action generated by $\xi$ and $J\xi$, we see that the length of the $S^1$-fiber of $p$
is bounded above by $2\pi\bar{b}(L)^2$. \eqref{diam sasa} thus follows. We next apply the coarea formula to $p$ to see that \eqref{vol sasa} holds.

{\rm (2c)} We know that in homogeneous coordinates $[w_{1}:\ldots:w_{n}]$ on $\mathbb{P}^{n-1}$ induced from holomorphic coordinates $(w_{1},\ldots,w_{n})$
on the ambient $\mathbb{C}^{n}$, $\widehat{\omega}^{T}$ is given by 
\begin{equation}\label{slow}
\widehat{\omega}^{T}=\frac{1}{2}dd^{c}\log\hat{r}=\frac{\bar{b}(L)^{2}}{2}\omega_{FS}+i\partial\bar{\partial}\tilde{\varphi},
\end{equation}
where $\tilde{\varphi}$ is as in the proof of Lemma \ref{uk}. As a holomorphic isometry, we then see that  
\begin{equation*}
\begin{split}
\Psi^{*}\widehat{\omega}^{T}&=\frac{\bar{b}(L)^{2}}{2}\omega_{FS}+i\partial\bar{\partial}(\Psi^{*}\tilde{\varphi}) \\
&=(T^{-1})^{*}\left(\bar{b}(L)^{2}T^{*}\left(\frac{1}{2}\omega_{FS}\right)+i\partial\bar{\partial}[(\Psi\circ T)^{*}\tilde{\varphi}]\right)\\
&=(T^{-1})^{*}\left(\bar{b}(L)^{2}T^{*}\left(\frac{1}{2}\omega_{FS}\right)+i\partial\bar{\partial}\varphi\right)=(T^{-1})^{*}\bar{\omega},
\end{split}
\end{equation*}
where we have used \eqref{fck} followed by \eqref{lovely} in the last two equalities. Hence $\widehat{\omega}^{T}=(\Phi^{-1})^{*}\bar{\omega}$, where $\Phi=\Psi\circ T$ as in \eqref{fuck}.

Finally, note that by \eqref{coeff}, $\bar{b}(L)$ is determined by, and hence depends smoothly upon, $\operatorname{vol}_{\bar{g}}(\mathbb{P}^{n-1})$. Hence in light of 
\eqref{slow}, showing that $\widehat{\omega}^{T}$ depends smoothly on $\hat{g}$ comes down to showing that $\tilde{\varphi}$, or equivalently $\Psi^{*}\tilde{\varphi}$, depends smoothly on $\bar{g}$ (as a function on the interval $[0,\,\frac{\pi}{2}]$).
To this end, observe from \eqref{lovely} that
\begin{equation*}
\begin{split}
i\partial\bar{\partial}\left(\Psi^{*}\tilde{\varphi}\right)&=
i\partial\bar{\partial}\left((T^{-1})^{*}\varphi\right)
=(T^{-1})^{*}\bar{\omega}-\frac{\bar{b}(L)^{2}}{2}\omega_{FS}.
\end{split}
\end{equation*}
Contracting this equation with $g_{FS}$, the Fubini-Study metric, we see that, after imposing that $\varphi\left(\frac{L}{2}\right)=0$, $\Psi^{*}\tilde{\varphi}$ is the unique solution of the following elliptic PDE:
\begin{equation}\label{laplace}
\left\{ \begin{array}{ll}
\Delta_{g_{FS}}u=\operatorname{tr}_{g_{FS}}\left((T^{-1})^{*}\bar{\omega}\right)-\frac{n\bar{b}(L)^{2}}{2}, & \\
u\left(\frac{\pi}{4}\right)=0.\\
\end{array} \right.
\end{equation}
Now, from \eqref{beta}, we read that the map $T^{-1}$ is determined by the map 
$$\tau^{-1}=\phi^{-1}\circ\tan:\left(0,\,\frac{\pi}{2}\right)\to\mathbb{R}$$
and from its very definition we see that $\phi$ is defined in terms of data that is determined by $\bar{g}$, namely
$L$ which is equal to $\operatorname{diam}(\mathbb{P}^{n-1},\,\bar{g})$, $s$ which is the distance from the point $[1:0:\ldots:0]\in\mathbb{P}^{n-1}$, and 
$\bar{a}(s)^{2}$ which is equal to $[\bar{g}(\xi,\,\xi)](s)$, the norm squared  of the Reeb vector field $\xi$ of the round Sasaki structure on $S^{2n-3}$. Thus, we deduce that $\phi^{-1}$ depends smoothly on $\bar{g}$. It is now clear that the right-hand side of \eqref{laplace}, and as a result its unique solution $u$, depends smoothly on $\bar{g}$. This concludes the proof of item (2c). 
\end{proof}

\subsection{Positive curvature conditions}

In this subsection, we introduce several notions of positive curvature and their relation on K\"ahler manifolds. Throughout, we adopt the following convention for the curvature tensor (see also \eqref{conven}):
\[
\hat R(X,Y)Z=\widehat{\nabla}_X\left(\widehat{\nabla}_YZ\right)-\widehat{\nabla}_Y\left(\widehat{\nabla}_XZ\right)-\widehat{\nabla}_{[X,\,Y]}Z.
\]
Let us recall the notion of the curvature operator and its lower bound on $(1,\,1)$-forms on a K\"ahler manifold (see \cite{CaoChow1986, ChenZhu2005, Siu1980}).

\begin{defn}
Fix any $p\in M$. We say that a K\"ahler manifold has \emph{curvature operator $\Rm$
strictly greater than (respectively bounded below by) $2\lambda\in \mathbb{R}$ on $(1,1)$-forms at $p$} if for any nonzero $(1,1)$-form $i\,\xi_{\alpha\bar \beta}\,dz^{\alpha}\wedge dz^{\bar \beta}$ at $p$,
\begin{equation}\label{stronger}
-\hat R_{\alpha\bar{\beta}\gamma\bar{\delta}}\xi^{\alpha\bar\beta}\overline{\xi^{\delta\bar{\gamma}}}> (\ge \text{  resp.})\,2\lambda(\hat g_{\alpha\bar{\beta}}\hat g_{\gamma\bar{\delta}}+\hat g_{\alpha\bar \delta}\hat g_{\gamma\bar \beta})\xi^{\alpha\bar\beta}\overline{\xi^{\delta\bar{\gamma}}},
\end{equation}
where $\xi^{\alpha\bar\beta}=\hat g^{\alpha\bar\varepsilon}\hat g^{\eta\bar\beta}\xi_{\eta\bar \varepsilon}$.
\end{defn}
Note that the negative sign before $\hat R$ in \eqref{stronger} (as well as \eqref{e: unusual} below) is due to the convention we adopt in \eqref{conven}.
We may reformulate the condition above in a simpler manner in terms of the curvature operator on \emph{real} $(1,\,1)$-forms. 

\begin{defn}[Real $(1,\,1)$-form]
Let $(M,\,J)$ be a complex manifold of dimension $n$ with complex structure $J$. 
A \emph{real $(1,\,1)$-form} $\sigma$ on $M$ is a real two-form $\sigma$ on $M$ satisfying either one of the following two equivalent conditions:  

{\rm{(1)}} $\sigma(J(\cdot),\,J(\cdot))=\sigma(\cdot\,,\,\cdot)$; 

{\rm{(2)}} In local holomorphic coordinates $(z_{1},\ldots,z_{n})$ on $M$, the complex bilinear extension
$\sigma_{\mathbb{C}}$ of $\sigma$ to $(TM\otimes\mathbb{C})^{\otimes 2}$ takes the form
$\sigma_{\mathbb{C}}=i\, u_{\alpha\bar \beta}\,dz^{\alpha}\wedge dz^{\bar \beta}$, where $(u_{\alpha\bar \beta})_{\alpha\beta}$ is hermitian, i.e., $\overline{u_{\alpha\bar \beta}}=u_{\beta\bar\alpha}$ for all $\alpha$ and $\beta$.  
\end{defn}

\noindent By the symmetry of the curvature tensor, it induces a symmetric linear operator \linebreak $\Rm:\Lambda^2(M)_p\to \Lambda^2(M)_p$ such that for any $u\in \Lambda^2(M)_p$,
\[
\Rm(u)_{ij}=\hat R_{ijkl}\hat{g}^{ka}\hat{g}^{lb}u_{ab}, \]
where $\Lambda^2(M)_p$ is the space of real two-forms at $p$.
We denote the space of real $(1,\,1)$-forms at $p$ by $\Lambda^{1,1}_{\mathbb{R}}(M)_p\subseteq \Lambda^2(M)_p$. It can be seen from the K\"ahlerity of the metric that  $\Rm\left(\Lambda^2(M)_p\right)\subseteq \Lambda^{1,1}_{\mathbb{R}}(M)_p$ 
and $\Rm\equiv 0$ on the orthogonal complement of $\Lambda^{1,1}_{\mathbb{R}}(M)_p$ in $ \Lambda^2(M)_p$ (see \cite{CaoChow1986, ChenZhu2005}). In particular, $\Rm$ has non-trivial kernel in complex dimension $n\ge 2$.

\begin{defn}\label{equiv}
A complex $n$-dimensional K\"ahler manifold is said to have \emph{curvature operator $\Rm$
strictly greater than (respectively bounded below by) $2\lambda\in \mathbb{R}$ on real $(1,\,1)$-forms at $p$}\footnote{
The notion of the curvature operator restricted to real $(1,1)$-forms is also referred to as the \emph{K\"ahler curvature operator} in \cite{PW21} and as the \emph{complex curvature operator} in \cite{WZ11}.} if for any nonzero real $(1,1)$-form $i\, u_{\alpha\bar \beta}\,dz^{\alpha}\wedge dz^{\bar \beta}$ at $p$,
\begin{equation}\label{e: unusual}
-\hat R_{\alpha\bar{\beta}\gamma\bar{\delta}}u^{\alpha\bar\beta}u^{\gamma\bar{\delta}}> (\ge \text{  resp.})\,2\lambda(\hat g_{\alpha\bar{\beta}}\hat g_{\gamma\bar{\delta}}+\hat g_{\alpha\bar \delta}\hat g_{\gamma\bar \beta})u^{\alpha\bar\beta}u^{\gamma\bar{\delta}},
\end{equation}
where $u^{\alpha\bar\beta}=\hat{g}^{\alpha\bar\varepsilon}\hat g^{\eta\bar\beta}u_{\eta\bar \varepsilon}$.
\end{defn}
It is not difficult to see that Condition \eqref{stronger} implies Condition \eqref{e: unusual}. Moreover, by considering the following decomposition of $(1,1)$-forms
\[
\xi_{\alpha\bar\beta}=\tfrac{\xi_{\alpha\bar\beta}+\overline{\xi_{\beta\bar\alpha}}}{2}+\tfrac{\xi_{\alpha\bar\beta}-\overline{\xi_{\beta\bar\alpha}}}{2}=a_{\alpha\bar\beta}+i\,b_{\alpha\bar \beta},
\]
where $a_{\alpha\bar\beta}:=\frac{1}{2}(\xi_{\alpha\bar\beta}+\overline{\xi_{\beta\bar\alpha}})$ and $b_{\alpha\bar\beta}:=\frac{1}{2i}(\xi_{\alpha\bar\beta}-\overline{\xi_{\beta\bar\alpha}})$, we have $\overline{a_{\alpha\bar\beta}}=a_{\beta\bar\alpha}$, $\overline{b_{\alpha\bar\beta}}=b_{\beta\bar\alpha}$, and so Condition \eqref{e: unusual} also implies Condition \eqref{stronger}, i.e., these two conditions are equivalent. 

By abuse of notation, we still denote the optimal $\lambda$ in \eqref{e: unusual} by $\lambda$, namely
\begin{equation}\label{lleqn}
    \lambda(p):=\inf_{i\,u_{\alpha\bar \beta}\,dz^{\alpha}\wedge dz^{\bar \beta}\in \Lambda^{1,1}_\mathbb{R}(M)_p\setminus\{0\}}\frac{-\sum_{\alpha,\beta, \gamma, \eta=1}^n\hat R_{\alpha\bar{\beta}\gamma\bar{\eta}}u^{\alpha\bar\beta}u^{\gamma\bar{\eta}}}{\sum_{\alpha,\beta, \gamma, \eta=1}^n 2(\hat g_{\alpha\bar{\beta}}\hat g_{\gamma\bar{\eta}}+\hat g_{\alpha\bar \eta}\hat g_{\gamma\bar \beta})u^{\alpha\bar\beta}u^{\gamma\bar{\eta}}}.
\end{equation}
As $\Rm\left(\Lambda^2(M)_p\right)\subseteq \Lambda^{1,1}_{\mathbb{R}}(M)_p$ and $\Rm\equiv 0$ 
on the orthogonal complement of $\Lambda^{1,1}_{\mathbb{R}}(M)_p$, nonnegative curvature operator on $(1,1)$-forms implies nonnegative curvature operator. It is also well known that strictly positive curvature operator on $(1,\,1)$-forms implies strictly positive sectional curvature \cite{ChenZhu2005, Siu1980}. Indeed,
let $X$, $Y$ be two linearly independent real vectors in $T_p M$ and let $v$ and $w$ be the corresponding complex vectors in $T_p^{1,0}M$ given by
\[
v=\frac{1}{2}\left(X-i\,J_0X\right)\qquad\text{and} \qquad w=\frac{1}{2}\left(Y-i\,J_0Y\right).
\]
Then the sectional curvature $K(\sigma)$ of the 2-plane $\sigma$ generated by $X$ and $Y$ is equal to (cf.~\cite{Cao1996, Siu1980})
\begin{equation*}
    \begin{split}
&K(\sigma)=\\
&\frac{-\sum_{\alpha,\beta, \gamma, \delta=1}^n \hat R_{\alpha\bar{\beta}\gamma\bar{\delta}}\left(v^\alpha \overline{w^{\beta}}-w^\alpha \overline{v^{\beta}}\right)\left(w^\gamma \overline{v^{\delta}}-v^\gamma \overline{w^{\delta}}\right)}{\sum_{\alpha,\beta, \gamma, \delta=1}^ng_{\alpha \bar\delta}g_{\gamma \bar\beta}\left[\left(v^\alpha \overline{w^{\beta}}-w^\alpha \overline{v^{\beta}}\right)\left(w^\gamma \overline{v^{\delta}}-v^\gamma \overline{w^{\delta}}\right)-\left(v^\alpha w^\gamma-w^\alpha v^\gamma\right)\left(\overline{v^\beta} \overline{w^{\delta}}-\overline{w^\beta} \overline{v^{\delta}}\right)\right]}.
\end{split}
\end{equation*}
Hence by taking $u^{\alpha\bar\beta}=i\left(v^\alpha \overline{w^{\beta}}-w^\alpha \overline{v^{\beta}}\right)$ and $\lambda=0$ in \eqref{e: unusual}, strictly positive curvature operator on $(1,1)$-forms implies strictly positive sectional curvature and in turn strictly positive bisectional curvature. In the special case that $M$ is $U(n)$-invariant with fixed point $p$, then for any orthonormal pair $X$ and $Y$ in $T_p M$, the sectional curvature of $\sigma =$ span$\{X,Y\}$ is given by
\begin{equation}\label{uncur}
    K(\sigma)=\lambda\left(1+3 g(X,JY)^2\right),
\end{equation}
where $\lambda$ is as in \eqref{lleqn}.  
Recall from \cite{LiWang2005, GY2018, TamYu12}  that the bisectional curvature $BK\ge 2\lambda$ at $p\in M$ if for all holomorphic vectors $v,w\in T_p^{1,0}M$,
\[
-\hat R_{\alpha\bar{\beta}\gamma\bar{\delta}}v^{\alpha}\overline{v^{\beta}}w^{\gamma}\overline{w^{\delta}}\ge 2\lambda (\hat g_{\alpha\bar{\beta}}\hat g_{\gamma\bar{\delta}}+\hat g_{\alpha\bar \delta}\hat g_{\gamma\bar \beta})v^{\alpha}\overline{v^{\beta}}w^{\gamma}\overline{w^{\delta}}.
\]
It is not difficult to see that $\Rm\ge 2\lambda$ 
on $(1,1)$-forms implies that $BK\ge 2\lambda$. More precisely:

\begin{lem}\label{BKvsRm}
Let $M$ be a K\"ahler manifold and let $p\in M$. Suppose that $\Rm\ge 2\lambda$ on $(1,\,1)$-forms at $p$, where $\lambda\in \mathbb{R}$. Then the holomorphic bisectional curvature is bounded below by $2\lambda$ at $p$.
\end{lem}

\begin{remark}
    \textnormal{The converse of this lemma is not true in general \cite[Theorem 4]{WZ11}, i.e., a bisectional curvature lower bound is a weaker condition than a lower bound on the curvature operator on $(1,\,1)$-forms.}
\end{remark}

\begin{proof}[Proof of Lemma \ref{BKvsRm}]
For any $v,w\in T_p^{1,0}M$, let $u^{\alpha\bar\beta}=i\left(v^\alpha \overline{w^{\beta}}-w^\alpha \overline{v^{\beta}}\right)$. Then  
by \eqref{e: unusual}, we have that
\begin{equation*}
\begin{split}
   -2\hat{R}(v,\overline{w},w,\overline{v})+\hat{R}(v,\overline{w},v,\overline{w})+\hat{R}(\overline{v},w,\overline{v},w)
   &\ge\\
      -2\lambda(\hat g_{\alpha\bar{\beta}}\hat g_{\gamma\bar{\delta}}+\hat g_{\alpha\bar \delta}\hat g_{\gamma\bar \beta})(v^{\alpha}\overline{w^{\beta}}v^{\gamma}\overline{w^{\delta}}&+w^{\alpha}\overline{v^{\beta}}w^{\gamma}\overline{v^{\delta}}-v^{\alpha}\overline{w^{\beta}}w^{\gamma}\overline{v^{\delta}}-w^{\alpha}\overline{v^{\beta}}v^{\gamma}\overline{w^{\delta}}).
   \end{split}
\end{equation*}
Similarly, replacing $w$ by $i w$, substituting $u^{\alpha\bar\beta}=i\left(v^\alpha \overline{i w^{\beta}}-iw^\alpha \overline{v^{\beta}}\right)$ into \eqref{e: unusual}, and adding the resulting expression to the inequality above, we derive that
\[ 
-4\hat{R}(v,\overline{v},w,\overline{w})=-4\hat{R}(v,\overline{w},w,\overline{v})\ge 8\lambda (\hat g_{\alpha\bar{\beta}}\hat g_{\gamma\bar{\delta}}+\hat g_{\alpha\bar \delta}\hat g_{\gamma\bar \beta})v^{\alpha}\overline{v^{\beta}}w^{\gamma}\overline{w^{\delta}}.
\]
Here, we have used the symmetry $\hat R_{\alpha\bar{\delta}\gamma\bar{\beta}}=\hat R_{\alpha\bar{\beta}\gamma\bar{\delta}}$ of a K\"ahler curvature operator (see for instance \cite[Lemma 2.10]{RFTandA1}).
\end{proof}

Henceforth, we consider a complex $(n-1)$-dimensional K\"ahler manifold of the form \eqref{dwp} and
fix a special real orthonormal frame $\{\hat e_{k}\}_{k=1}^{2n-2}$ such that $\hat e_{2n-3}:=\partial_r$ and for all $k=1,\dots,n-1$, $J_0\left(\hat e_{2k-1}\right)=\hat e_{2k}$. We also let $\{\hat E_\alpha\}_{\alpha=1}^{n-1}$ and $\{\hat E_{\bar{\alpha}}\}_{\alpha=1}^{n-1}$ denote the corresponding holomorphic and anti-holomorphic frames, namely
\[
 \hat{E}_{\alpha}:=\tfrac{1}{2}(\hat{e}_{2\alpha-1}-i\,\hat{e}_{2\alpha}),\qquad
 \hat{E}_{\bar{\alpha}}:=\tfrac{1}{2}(\hat{e}_{2\alpha-1}+i\,\hat{e}_{2\alpha}),\qquad \alpha=1,\ldots,n-1.
\]
We extend the curvature tensor complex linearly. Lengthy computations using
Cartan's formulas (see \eqref{curform} in Appendix A) give us that 
\begin{eqnarray*}
\hat R_{\alpha\bar{\beta}\gamma\bar{\delta}} &=& \frac{\tilde R_{\alpha\bar{\beta}\gamma\bar{\delta}}}{b(r)^{2}}+\frac{1}{2}\left(\frac{b'(r)}{b(r)}\right)^2\delta_{\alpha\delta}\delta_{\beta\gamma}
+\left(\frac{a(r)^{2}}{2b(r)^{4}}\right)
\delta_{\alpha\beta}\delta_{\gamma\delta}\quad  \textrm{if $1\leq\alpha,\,\beta,\,\gamma,\,\delta\leq n-2$},\\
\hat R_{\alpha\bar{\beta}\gamma\bar{\delta}} &=&\left(\frac{b''(r)}{2b(r)}\right)\delta_{\gamma,\,n-1}\delta_{\delta\alpha} \quad\textrm{if $1\leq\alpha\leq n-2$, $\beta=n-1$},\\
\hat R_{\alpha\bar{\beta}\gamma\bar{\delta}} &=& \frac{1}{2}\left[\left(\frac{b''(r)}{b(r)}\right)(\delta_{\gamma\delta}-\delta_{\gamma,\, n-1}\delta_{\delta,\, n-1})+
\left(\frac{a''(r)}{2 a(r)}\right)\delta_{\gamma,\,n-1}\delta_{\delta,\,n-1}\right] \quad \textrm{if $\alpha=n-1$, $\beta=n-1$},
\end{eqnarray*}
where $\tilde R$ is the curvature tensor of the transverse metric $g^T$.
With the above formulas, we study the curvature lower bound of $\hat g$.

\begin{lem}\label{conditions on ab}
  In the above situation, for any nonnegative real number $\lambda\ge 0$ and $r\in (0,\,L)$, $\hat g$ has curvature operator strictly greater than $2\lambda$ on $(1,1)$-forms at $(r,w_0)\in\linebreak (0,L)\times S^{2n-3}$
  if and only if all of the following conditions are satisfied:

 {\rm{(1)}} $(1-\lambda b^2(r)-(b'(r))^2)>0$;
 
       {\rm{(2)}} $-a''(r)>4\lambda a(r)$;
       
       {\rm{(3)}} $-b''(r)>\lambda b(r)$;
       
       {\rm{(4)}} $\displaystyle\frac{2(n-1)(1-\lambda b^2(r)-(b'(r))^2)}{(n-2) b^2(r)}\left(-\frac{a''(r)}{4a(r)}-\lambda\right)>\left(-\frac{b''(r)}{b(r)}-\lambda\right)^2.$   
\end{lem}

\begin{proof} Let $(r\,,w_0)\in (0,L)\times S^{2n-3}$.
By Lemma \ref{warp vs curv}, we have curvature operator strictly greater than $2\lambda$ on $(1,\,1)$-forms if and only if for any nonzero hermitian matrix $u^{\alpha\bar\beta}$, 
\begin{eqnarray*}
    0&<&\frac{2(1-\lambda b^2-(b')^2)}{b^2}\left(\frac14\left(\sum_{\alpha=1}^{n-2}u^{\alpha\bar\alpha}\right)^2+\frac14\sum_{\alpha, \beta=1}^{n-2}\left|u^{\alpha\bar\beta}\right|^2\right)+\left(-\frac{a''}{4a}-\lambda\right)\left|u^{n-1\overline{n-1}}\right|^2\\
    &&\quad +\left(-\frac{b''}{b}-\lambda\right)\sum_{\alpha=1}^{n-2}\left|u^{\alpha\overline{n-1}}\right|^2+\left(-\frac{b''}{b}-\lambda\right)u^{n-1\overline{n-1}}\sum_{\alpha=1}^{n-2}u^{\alpha\bar\alpha}.
\end{eqnarray*}
Thus, by choosing $u$ suitably, we see that the curvature lower bound implies Conditions (1)-(3). Furthermore, setting $u$ to be the diagonal matrix $u^{\alpha\bar\alpha}=\frac{2}{n-2}$, $u^{\alpha \overline{n-1}}=0=u^{n-1 \overline{\alpha}}=u^{\alpha\bar \beta}$ for $\alpha,\beta =1,\dots,n-2,$ with $\alpha\neq \beta$, and
\[
u^{n-1\overline{n-1}}=-\frac{\left(\frac{b''}{b}+\lambda\right)}{\left(\frac{a''}{4a}+\lambda\right)},
\]
gives us Condition (4).

Conversely, if Conditions (1)--(4) hold, then by the Cauchy-Schwarz inequality and completing the square, for any nonzero hermitian $u$, we have that
\begin{eqnarray*}
    0&\le&
    \left[\frac{2(n-1)(1-\lambda b^2-(b')^2)}{(n-2)b^2}-\frac{\left(-\frac{b''(r)}{b(r)}-\lambda\right)^2}{\left(-\frac{a''}{4a}-\lambda\right)}\right]\frac14\left(\sum_{\alpha=1}^{n-2}u^{\alpha\bar\alpha}\right)^2\\
    &&\quad+\left(-\frac{a''}{4a}-\lambda\right)\left[u^{n-1\overline{n-1}}+\frac{\left(-\frac{b''(r)}{b(r)}-\lambda\right)}{2\left(-\frac{a''}{4a}-\lambda\right)}\sum_{\alpha=1}^{n-2}u^{\alpha\bar\alpha}\right]^2+\left(-\frac{b''}{b}-\lambda\right)\sum_{\alpha=1}^{n-2}\left|u^{\alpha\overline{n-1}}\right|^2\\
    &\le&\frac{2(1-\lambda b^2-(b')^2)}{b^2}\left(\frac14\left(\sum_{\alpha=1}^{n-2}u^{\alpha\bar\alpha}\right)^2+\frac14\sum_{\alpha, \beta=1}^{n-2}\left|u^{\alpha\bar\beta}\right|^2\right)+\left(-\frac{a''}{4a}-\lambda\right)\left|u^{n-1\overline{n-1}}\right|^2\\
    &&\quad +\left(-\frac{b''}{b}-\lambda\right)\sum_{\alpha=1}^{n-2}\left|u^{\alpha\overline{n-1}}\right|^2+\left(-\frac{b''}{b}-\lambda\right)u^{n-1\overline{n-1}}\sum_{\alpha=1}^{n-2}u^{\alpha\bar\alpha}.
\end{eqnarray*}
The expression on the right-hand side of this inequality is positive unless $u$ is equal to $0$. This completes the proof of the lemma.
\end{proof}

We will only consider asymptotically conical
expanding gradient K\"ahler-Ricci solitons \cite{con-der, De15, De14} satisfying $\Rm\ge0$ in this article. This latter condition is equivalent to the link of the asymptotic cone satisfying $\Rm\ge1$. In this case, the potential function is a strictly convex function with quadratic growth, the underlying manifold is diffeomorphic
to $\mathbb{R}^n$,
and the level sets of the soliton potential function are diffeomorphic to $S^{n-1}$. 
We will consider asymptotically conical expanding gradient K\"ahler-Ricci solitons resulting from the
following existence result by the second-named author and Deruelle.

\begin{theorem}[{\cite[Theorems A \& E]{con-der}}]\label{Thm E}\label{lift} Let $g_0$ be a K\"ahler cone metric on $\mathbb{C}^n$ 
with radial function $r$ such that $a\cdot r\partial_r$ is the Euler vector field, where $a\in (0,1)$. 
Then there exists a unique expanding gradient K\"ahler-Ricci soliton $g$ with soliton vector field $r\partial_r$ such that
$$|(\nabla^{g_{0}})^{k}(g-g_{0})|_{g_{0}}=O(r^{-2-k})\qquad\textrm{for all $k\geq0$.}$$
Moreover, if the induced metric on the complex space $\mathbb{P}^{n-1}$ has curvature operator on $(1,\,1)$-forms strictly greater than $2$, then $g$ has strictly positive curvature operator on $(1,\,1)$-forms.
\end{theorem}
In Section $5$, we apply Theorem \ref{lift} to K\"ahler cones with base spaces given by certain $U(n-1)$-invariant metrics on $\mathbb{P}^{n-1}$ (to be constructed in Sections $3$ and $4$) to obtain asymptotically conical expanding K\"ahler-Ricci solitons.
\begin{remark}\textnormal{
The lower bound
of the induced metric on the base space $\mathbb{P}^{n-1}$ should have curvature operator on $(1,\,1)$-forms strictly greater than $2$
in the statement of \cite[Theorem E]{con-der} as written here. To account for this change, the constant $a$ in \cite[p.356, Proof of Theorem E]{con-der} should be a function of $t$ as in the proof of Theorem \ref{t: existence of new}.}
\end{remark}

\begin{remark}\label{1eigv}
\textnormal{In the standard complex coordinates $z_{k}=x_{k}+iy_{k}$ on $\mathbb{C}^n$, the expanding gradient K\"ahler-Ricci soliton from Theorem \ref{Thm E} has soliton vector field given by $X=c x_{\alpha}\tfrac{\partial}{\partial x_{\alpha}}+c y_{\alpha}\tfrac{\partial}{\partial y_{\alpha}}$
for some $c<0$. This has a unique critical point at $0\in \mathbb{C}^n$. Moreover at $0$, the Ricci curvature $\Ric_g$ of $g$ has only one distinct eigenvalue, i.e., $\Ric_{g}=\mu g$ for some $\mu\in\mathbb{R}$. To see this, we compute the Lie derivative $\mathcal{L}_{X}g$. Indeed, for $\alpha,\beta=1,2,\dots, n$,
we have that
\begin{equation*}
\begin{split}
   \mathcal{L}_{X}g\left(\tfrac{\partial}{\partial x_{\alpha}},\tfrac{\partial}{\partial x_{\beta}}\right)&= g\left(\nabla_{\tfrac{\partial}{\partial x_{\alpha}}}X, \tfrac{\partial}{\partial x_{\beta}}\right)+g\left(\tfrac{\partial}{\partial x_{\alpha}}, \nabla_{\tfrac{\partial}{\partial x_{\beta}}}X\right)\\
   &=cg\left(\delta_{\alpha\gamma}\tfrac{\partial}{\partial x_{\gamma}}+x_{\gamma}\nabla_{\tfrac{\partial}{\partial x_{\alpha}}}\tfrac{\partial}{\partial x_{\gamma}}+y_{\gamma}\nabla_{\tfrac{\partial}{\partial x_{\alpha}}}\tfrac{\partial}{\partial y_{\gamma}}, \tfrac{\partial}{\partial x_{\beta}}\right)\\
   &\qquad+cg\left(\tfrac{\partial}{\partial x_{\alpha}}, \delta_{\beta\gamma}\tfrac{\partial}{\partial x_{\gamma}}+x_{\gamma}\nabla_{\tfrac{\partial}{\partial x_{\beta}}}\tfrac{\partial}{\partial x_{\gamma}}+y_{\gamma}\nabla_{\tfrac{\partial}{\partial x_{\beta}}}\tfrac{\partial}{\partial y_{\gamma}}\right)\\
   &=2cg\left(\tfrac{\partial}{\partial x_{\alpha}},\tfrac{\partial}{\partial x_{\beta}}\right)
\end{split}
\end{equation*}
when evaluated at $0$, and so $\mathcal{L}_{X}g=2cg$. Thus, by \eqref{e: soliton}, we find that
\[
\Ric_{g}=-\tfrac{1}{2}\mathcal{L}_{X}g-\lambda g=-\left(c+\lambda\right)g\qquad\textrm{at $0$.}
\]
}\end{remark}

\section{Deformations of complex projective space}

The main goal of this section is to prove Proposition \ref{l: smooth metric with nonnegative Rm} which produces for each $n\ge3$ a sequence of smooth $U(n-1)$-invariant K\"ahler metrics on $\mathbb{P}^{n-1}$ with $\Rm\ge0$ on real $(1,\,1)$-forms everywhere, and $\Rm\ge2$ on $(1,1)$-forms outside of arbitrarily small subsets, which collapse to $\left[0,\tfrac{\pi}{2}\right]$ in the Gromov-Hausdorff sense. In the next section, we will improve this result to obtain a sequence of metrics that everywhere satisfy $\Rm\ge2$ on real $(1,\,1)$-forms. The existence of these metrics will imply the almost non-rigidity of diameter in Theorem \ref{t: GH convergence on CPn} and serve in the proof of Theorem \ref{t: existence of new} as the complex base of the asymptotic cone of the expanding gradient K\"ahler-Ricci solitons that degenerate to our desired steady K\"ahler-Ricci solitons.

To motivate and sketch the ideas of the construction in Proposition \ref{l: smooth metric with nonnegative Rm}, we demonstrate a singular example: a variation of the Fubini-Study metric. Upon substituting $L=\frac{\pi}{2}$,
$a(r)=\frac{\sin(2r)}{2k}$, $b(r)=\frac{\sin(r)}{\sqrt{k}},$ into \eqref{cohomo in App} to obtain the metric 
\begin{equation}\label{e: metric with conical singularity}
h_k:=dr^{2}+\frac{\sin^2(2r)}{4k^2}\eta\otimes \eta+\frac{\sin^2(r)}{k}g^{T}\qquad\textrm{on $\left(0,\,\frac{\pi}{2}\right)\times S^{2n-3}$},
\end{equation}
where $g^T$ is the transverse metric as in Example \ref{projective} with $n$ replaced by $(n-1)$, it is clear that $h_k$ collapses to $[0,\frac{\pi}{2}]$ as $k\to\infty$ in the Gromov-Hausdorff sense and satisfies $\Rm\ge2$ on real $(1,\,1)$-forms except for $r=0$ and $r=\frac{\pi}{2}$, where $a$ and $b$ don't fulfill the smooth boundary conditions for any $k>1$ (see Example \ref{projective}). However, we will build our desired metrics by gluing these singular metrics with suitable expanding solitons near $r=0$ to smooth out the conical singularities.

In preparation for the gluing construction, we recall some basic facts about expanding gradient K\"ahler-Ricci solitons, and in particular Cao's examples of such solitons. To this end, let $(M^n,g,f,p)$ be an expanding gradient Ricci soliton with $p$ a critical point of $f$.
Then $(M^n,g,f,p)$ satisfies
\begin{equation*}
    \Ric+\lambda\,g=\nabla^2 f
\end{equation*}
for some $\lambda>0$ and generates a canonical Ricci flow of expanding Ricci solitons defined by $g(t):=(2\lambda t)\phi^*_{t-\frac{1}{2\lambda}}g$, $f_t=\phi^*_{t-\frac{1}{2\lambda}}f$, $t\in(0,\infty)$, where $\{\phi_s\}_{s\in\left(-\frac{1}{2\lambda},\infty\right)}$ is the one-parameter family of diffeomorphisms generated by the time-dependent vector field $\frac{-1}{1+2\lambda s}\nabla f$ with $\phi_0$ the identity, and 
\begin{equation*}
    \Ric(g(t))+\frac{1}{2t}g(t)=\nabla^2 f_t.
\end{equation*}

For any $\alpha>1$, Cao proved that there exists a unique smooth $U(n-1)$-invariant expanding gradient K\"ahler-Ricci soliton $(\mathbb C^{n-1},g_{\alpha},f_{\alpha},p_{\alpha})$ with positive curvature operator on real $(1,\,1)$-forms and with $R(p_{\alpha})=\max R=1$ that is asymptotic to the following cone metric \cite{Cao1996, ChenZhu2005}:
\begin{equation*}
g_{\operatorname{cone},\,\alpha}=dr^{2}+\frac{r^2}{\alpha^2}\,\eta\otimes \eta+\frac{r^2}{\alpha}\,g^{T},
\end{equation*}
where $R$ is the scalar curvature of $g_\alpha$.
Under the radial coordinates, we can write the expanding soliton metric $g_{\alpha}$ as
\begin{equation}\label{e: expander alpha}
    g_{\alpha}=dr^{2}+a^2_{\alpha}(r)\,\eta\otimes \eta+b^2_{\alpha}(r)\,g^{T}.
\end{equation}
Let $g_{\alpha}(t)$ be the canonical Ricci flow associated to $g_{\alpha}$. Since the expanding Ricci soliton is asymptotically conical, we know that the flow $g_{\alpha}(t)$ converges smoothly locally to the cone metric $g_{\operatorname{cone},\,\alpha}$ away from the tip $*$ of the cone as $t\to0$; cf.~\cite[Theorem 4.3.1]{Siepmann2013Thesis}. In the following lemma, we show that this local convergence is uniform for all $\alpha$ in a compact subset $I\subset(1,\infty)$.

\begin{lem}[Uniform convergence to asymptotic cones]\label{l: uniform convergence}
    For any compact subset $I\subset(1,\infty)$, $k\in\mathbb{N}$, $\varepsilon>0$, and $D>1$, there exists $t_0>0$ such that for all $0<t<t_0$ and all $\alpha\in I$, 
    $$\|\nabla^k(g_{\alpha}(t)-g_{\operatorname{cone},\,\alpha})\|\le\varepsilon$$
    on $B(*,D)\setminus B(*,D^{-1})$, where the derivatives, norms, and metric balls are measured with respect to the cone metric $g_{\operatorname{cone},\,\alpha}$.
\end{lem}

\begin{proof}
Suppose that this was not the case. Then there exist $\varepsilon_0>0$, $D_0>0$, $k_0\in\mathbb{N}$, and sequences $t_i\to0$ and $\alpha_i\in I$ such that for the expanding K\"ahler-Ricci soliton $(\mathbb C^{n-1},g_{\alpha_i},f_i,p_i)$ and the canonical Ricci flow $g_{\alpha_i}(t)$, we have that
\begin{equation}\label{contrad b''}
      \|\nabla^{k_0}(g_{\alpha_i}(t_i)-g_{\operatorname{cone},\,\alpha_i})\|\ge\varepsilon_0
\end{equation}
on $B(*,D_0)\setminus B(*,D_0^{-1})$. Recall that $R_{g_{\alpha_i}}(p_i)=1$.

Assume that $g_{\alpha_i}=g_{\alpha_i}(C_i)$ for some positive constant $C_i>0$. We claim that there exists a positive constant $C$ such that $C_i<C$ and for all $t>0$,
  \begin{equation*}
  |\Rm_{g_{\alpha_i}(t)}|\le \frac{C}{t}.
  \end{equation*}
To prove this, we assume the contrary. Then there is a sequence $C_i\rightarrow\infty$ such that $|\Rm_{g_{\alpha_i}(t)}|\le \frac{C_i}{t}$ and $R(p_i,C_i)=1$. By the same limiting argument as in the proof of \cite[Lemma 2.3]{Lai2020_flying_wing}, we can take a smooth sequential limit pointed at $p_i$ of the flows $g_{\alpha_i}(t+C_i)$, $t\in(-C_i,\infty)$, to obtain a smooth Ricci flow for $t\in(-\infty,\infty)$ which is associated to a smooth steady K\"ahler-Ricci soliton $(\mathbb C^{n-1},g_{\infty},f_{\infty},p_{\infty})$. Since $I\subset(0,\infty)$ is compact, these expanding solitons are uniformly non-collapsed in the sense that there is some $\kappa>0$ such that $B_{g_{\alpha_i}}(p_i,r)\ge\kappa r^{2(n-1)}$ for all $r>0$, hence the limit steady Ricci soliton and the asymptotic cone are also $\kappa$-non-collapsed, and $R(p_{\infty},0)=1$. By Perelman’s curvature estimate for Ricci flows with nonnegative
curvature operator (see for example \cite[Corollary 45.1(b)]{KL} and also \cite[Theorem 2]{ni2005ancient} for the K\"ahler case), this implies that the steady Ricci soliton is flat. But this contradicts the fact that $R(p_{\infty},0)=1$ and so we have the claim. By Shi's estimates \cite{Shi1987derivative1}, it follows that
\begin{equation}\label{e: curvature derivatives}
  |\nabla^k\Rm_{g_{\alpha_i}(t)}|\le \frac{C_k}{t^{\frac{k+2}{2}}}
  \end{equation}
for some $C_k>0$ depending $k$ and uniform for all $i$.

By the curvature bounds \eqref{e: curvature derivatives}, we may assume, after taking a subsequence if necessary, that $\alpha_i\to\alpha_{\infty}\in I$ and $g_{\alpha_i}(t)$ converges smoothly locally on $(0,\,\infty)$ to a flow $g_{\infty}(t)$ on $(0,\,\infty)$. Moreover, if $\alpha_i\to\alpha_{\infty}$ as $i\to\infty$, then by the same argument as in \cite[Lemma 2.3]{Lai2020_flying_wing}, one can show that $g_{\infty}(t)=g_{\alpha_{\infty}}(t)$ is the canonical Ricci flow of the expanding Ricci soliton $g_{\alpha_{\infty}}$ asymptotic to the cone metric $g_{\operatorname{cone},\,\alpha_{\infty}}$. 
Since $I$ is compact in $(1,\infty)$, for each $k\in\mathbb N$, we can find $A_k>0$, uniform over all $i=1,...,\infty$, such that on each smooth Riemannian cone $g_{\operatorname{cone},\,\alpha_i}$, we have that
\begin{equation}\label{curv est}
    \sup_{x\,\in\,\mathbb C^{n-1}}r^{2+k}_{\operatorname{cone},\,\alpha_i}(x)|\nabla^k\Rm|(x)=\sup_{x\,\in\,\partial B(*,\,1)}|\nabla^k\Rm|(x)=A_k<\infty,
\end{equation}
where $r_{\operatorname{cone},\,\alpha_i}(x)=d_{g_{\operatorname{cone},\,\alpha_i}}(x,\,*)$ and the curvature and derivatives are with respect to $g_{\operatorname{cone},\,\alpha_{\infty}}$. Since the Ricci flows $g_{\alpha_i}(t)$ all converge smoothly locally to their asymptotic cones $g_{\operatorname{cone},\,\alpha_i}$ as $t\to0$ on $\mathbb C^{n-1}\setminus\{*\}$, $g_{\alpha_i}(t)$ can be extended smoothly to $t=0$ on $\mathbb C^{n-1}\setminus\{*\}$. 
Moreover, by \eqref{e: curvature derivatives}, \eqref{curv est}, and Shi's local derivative estimates \cite{RFTandA2}, we can deduce the following uniform curvature estimates for all $i=1,...,\infty$: for all $k\in\mathbb{N}$ and $D>0$, there exists $A_k(D)>0$ such that $|\nabla^k\Rm_{g_{\alpha_i}(t)}|\le A_k$ on $B(*,D)\setminus B(*,D^{-1})$ for all $t>0$.
Therefore, after passing to a subsequence if necessary, we may assume that the Ricci flows $g_{\alpha_i}(t)$ converge smoothly to $g_{\alpha_{\infty}}(t)$ on any compact subset of $\mathbb C^{n-1}\setminus\{*\}$, uniformly on any $[0,\infty)$. Thus, we have that 
\begin{equation*}
    \begin{split}
        \|\nabla^k(g_{\alpha_i}(t_i)-g_{\operatorname{cone},\,\alpha_i})\|&\le \|\nabla^k(g_{\alpha_i}(t_i)-g_{\alpha_{\infty}}(t_i))\|+\|\nabla^k(g_{\operatorname{cone},\,\alpha_{\infty}}-g_{\alpha_{\infty}}(t_i))\|\\
&\qquad\qquad\qquad\qquad+\|\nabla^k(g_{\operatorname{cone},\,\alpha_{\infty}}-g_{\operatorname{cone},\,\alpha_i})\|\\
        &<\tfrac{\varepsilon_0}{2}
    \end{split}
\end{equation*}
for all sufficiently large $i$ on $B(*,D)\setminus B(*,D^{-1})$, where the norms and derivatives are taken with respect to $g_{\operatorname{cone},\,\alpha_{\infty}}$. Choosing $D>D_0$, this contradicts \eqref{contrad b''} and we are done.
\end{proof}

In the next lemma, we show that for 
any singular metric from \eqref{e: metric with conical singularity}, there exists an 
expanding Ricci soliton from Cao's family such that we can glue a compact $U(n-1)$-invariant subset of it to a $U(n-1)$-invariant subset of the singular metric to obtain a metric on $\mathbb P^{n-1}$ which is $C^1$-smooth near $r=0$.

\begin{lem}\label{e: find the expanding soliton}
For all large $k\in\mathbb{N}$, there exists a sequence $\{(\mathbb{C}^{n-1},g_i,r_i,s_i)\}_{i=1}^{\infty}$, where $r_i$, $s_i\to 0$ as $i\to\infty$, and $(\mathbb{C}^{n-1},g_i)$ is a $U(n-1)$-invariant expanding gradient K\"ahler-Ricci soliton which takes the form
\[
g_i=dr^2+a_i^2(r)\eta\otimes \eta+b_i^2(r)g^T
\]
for some strictly positive smooth functions $a_i, b_i:(0,\infty)\to (0,\infty)$ satisfying $a_i=b_ib_i'$ and 
\begin{equation}\label{e: derivatives m=0,1,2}
b_i^{(m)}(s_i)=\frac{\sin^{(m)}(r_i)}{\sqrt{k}}\qquad \textnormal{for }m=0,1,2.
\end{equation}
\end{lem}

\begin{proof}
Let $100\le k\in\mathbb{N}$. By Lemma \ref{l: uniform convergence}, for each $i\in\mathbb{N}$ we can find a large $R_i>0$ such that the rescaled expanding soliton $R_i^{-2}g_{\alpha}$ (where $g_{\alpha}$ is from \eqref{e: expander alpha}) is $\frac{1}{100\,i^2}$-close to the cone metric $g_{\operatorname{cone},\,\alpha}$ 
near radius $1$  for all $\alpha\in[k-1,k+1]$. In particular, this implies that
$$-\frac{R_i^2b''_{\alpha}(R_i)}{b_{\alpha}(R_i)}+|b'_\alpha(R_i)-\tfrac{1}{\sqrt{\alpha}}|\le\frac{1}{i^2}.$$
Notice that these two quantities vanish on $g_{\operatorname{cone},\,\alpha}$. In particular, we see that
\[
0<\lambda_{i,\,\alpha}:=-\frac{b''_\alpha(R_i)}{b_\alpha(R_i)}\le \frac{1}{R_i^2i^2}\qquad\text{and}\qquad b'_\alpha(R_i)\in \left[\frac{1}{\sqrt{\alpha+\frac12}}, \frac{1}{\sqrt{\alpha-\frac12}}\right].
\]
Furthermore, by the smoothness conditions $b_\alpha(0)=0$, $b_\alpha'(0)=1$, and the concavity of $b_\alpha$, we know that $b_\alpha(R_i)\le R_i$. Consider the further rescaled expanding soliton $h_{\alpha}:=\lambda_{i,\,\alpha} g_\alpha$ which has the form
\[
h_\alpha=ds^2+\bar a^2_\alpha(s) \eta\otimes \eta+\bar b_\alpha^2(s)g^T.
\]
Then we have that
\[
-\frac{\bar b''_\alpha(s_{i,\,\alpha})}{\bar b_\alpha(s_{i,\,\alpha})}=1,\qquad\bar b_\alpha(s_{i,\,\alpha})\le s_{i,\,\alpha}\le\frac 1i,
 \]
where $s_{i,\,\alpha}=\sqrt{\lambda_{i,\,\alpha}} R_i\le\frac{1}{i}$,
and 
\begin{equation}\label{e: b' compare 1}
    \bar b_\alpha'(s_{i,\,\alpha})=b_\alpha'(R_i)\in\left[\frac{1}{\sqrt{\alpha+\frac12}}, \frac{1}{\sqrt{\alpha-\frac12}}\right].
\end{equation}

We can assume that $i>k$ is sufficiently large. Then
there is a unique $r_{i,\,\alpha}\in(0,\tfrac{\pi}{2})$ such that
\begin{equation}
\label{e: def r}
    -\bar b''_\alpha(s_{i,\,\alpha})=\bar b_\alpha(s_{i,\,\alpha})=\frac{\sin(r_{i,\,\alpha})}{\sqrt{k}}\le\frac{1}{i},
\end{equation}
and we have that $r_{i,\,\alpha}\le \frac{2\sqrt{k}}{i}\le \frac{2}{\sqrt{i}}$ and also
\begin{equation}\label{e: cos range}
\frac{1}{\sqrt{k+1/2}}<\frac{\cos(r_{i,\,\alpha})}{\sqrt{k}}<\frac{1}{\sqrt{k-1/2}}. 
\end{equation}
In particular, \eqref{e: def r} implies that assertion \eqref{e: derivatives m=0,1,2} holds for $m=0,2,$ with our choice of $r_{i,\,\alpha}$, $s_{i,\,\alpha},$ and the expanding soliton $h_{\alpha}$ for any $\alpha\in[k-1,k+1]$.
We will further determine a value of $\alpha$ so that \eqref{e: derivatives m=0,1,2} also holds for $m=1$. To do this, we first observe that by combining \eqref{e: b' compare 1} and \eqref{e: cos range}, $\bar b'_\alpha(s_{i,\,\alpha})> \frac{\cos(r_{i,\,\alpha})}{\sqrt{k}}$ for $\alpha=k-1$ and $\bar b'_\alpha(s_{i,\,\alpha})< \frac{\cos(r_{i,\,\alpha})}{\sqrt{k}}$ for $\alpha=k+1$.
It is easy to see that $r_{i,\,\alpha}$ and $s_{i,\,\alpha}$ are both continuous in $\alpha\in [k-1,k+1]$, so by the intermediate value theorem there exists an $\alpha=\alpha(i,k)\in (k-1,k+1)$ such that $b'_\alpha(s_{i,\,\alpha})=\frac{\cos(r_{i,\,\alpha})}{\sqrt{k}}$. This implies that
\eqref{e: derivatives m=0,1,2} also holds for $m=1$.
\end{proof}

For each singular metric $h_k=dr^{2}+\frac{\sin^2(2r)}{4k^2}\eta\otimes \eta+\frac{\sin^2(r)}{k}g^{T}$, $r\in[0,\frac{\pi}{2}]$, from \eqref{e: metric with conical singularity},
by gluing $h_k$ to the expanding solitons obtained in Lemma \ref{e: find the expanding soliton} at arbitrarily small scales near $r=0$, we obtain a sequence of metrics that are $C^1$-smooth near $r=0$ and converge to $h_k$.
In the next proposition, we will further modify these metrics to obtain a sequence of smooth K\"ahler metrics which converge to $h_k$, and satisfy $\Rm\ge0$ on real $(1,\,1)$-forms everywhere and $\Rm\ge2$ on $(1,1)$-forms outside of arbitrarily small subsets.

\begin{prop}[Smoothing out cone points by gluing expanding solitons]\label{l: smooth metric with nonnegative Rm} Let $n\ge 3$ and let $k\in\mathbb{N}$ be any fixed large natural number. Then there exists a sequence of strictly positive real numbers $\delta_i\to0$ and a sequence of real numbers $\eta_i\to 0$ as $i\to\infty$, and a sequence of smooth $U(n-1)$-invariant K\"ahler manifolds $(\mathbb{P}^{n-1},\,g_i)$ which can be written as $g_i=dr^2+a^2_i(r)\eta\otimes\eta+ b^2_i(r)g^T$, 
$r\in[0,\frac{\pi}{2}+\eta_i]$, satisfying the following:
    \begin{enumerate}[label=\textnormal{(\alph{*})}, ref=(\alph{*})]
        \item\label{i: Rm>0} $\Rm_{g_i}>0$ on real $(1,\,1)$-forms everywhere;
        \item\label{i: Rm>2}  $\Rm_{g_i}\ge 2$ on real $(1,\,1)$-forms on $r^{-1}([\delta_i,\tfrac{\pi}{2}-\delta_i])$; 
        \item\label{i: converge to h_k}
         For every $\varepsilon>0$, $g_i$ smoothly converges to $h_k$ on $r^{-1}([\varepsilon,\frac{\pi}{2}-\varepsilon])$ as $i\to\infty$.
    \end{enumerate}
\end{prop}

\begin{proof}
Fix $k$ large, for each $i$ let $r_i,\,s_i>0$ be the sequences from Lemma \ref{e: find the expanding soliton} converging to zero, and let $b_1:[0,\infty)\to[0,\infty)$ be the warping function of the corresponding expanding Ricci soliton, where we omit the dependence of $b_1$ (and all the following functions $a_1,a_2,a_3,a_4,$ and $b,b_2,b_3,b_4$) on $i$.  Let $b:[0,\tfrac{\pi}{2}+s_i-r_i]\to[0,\infty)$ be defined by
    \begin{equation*}
        b(r) =\left\{\begin{array}{lr}
        b_1(r) & \text{for } r\in[0,s_i],\\
     b_2(r):=\frac{1}{\sqrt{k}}\sin(r-s_i+r_i)& \text{for } r\in[s_i,\tfrac{\pi}{2}+s_i-r_i].\\
        \end{array}\right.
    \end{equation*}
    Then $b$ is $C^2$ by Lemma \ref{e: find the expanding soliton}. We also write $a_2:=b_2b_2'=\tfrac{1}{2k}\sin(2(r-s_i+r_i))$.

Let $\varphi:(-\infty,\infty)\to[0,\infty)$ be a non-decreasing smooth function such that $\varphi(r)=0$ on $(-\infty,0]$ and $\varphi(r)=1$ on $[t,\infty)$, where $t\in(0,\min\{s_i,r_i\})$ is some sufficiently small number whose value we will choose later. In particular, $\varphi^{(k)}(0)=0=\varphi^{(k)}(t)$ for all $k\in \mathbb{N}$.
Define
    \begin{equation*}
        B_i(r):=\varphi(s_i-r)b'''_1(r)+(1-\varphi(s_i-r))b'''_2(r)\qquad\textrm{for $r\in [0,\tfrac{\pi}{2}+s_i-r_i]$}.
    \end{equation*}
Then $B_i$ is smooth. Let $b_3:[0,\tfrac{\pi}{2}+s_i-r_i]\to\mathbb R$ be the unique smooth function solving
    \begin{equation}\label{e: defn of b}
        \left\{\begin{array}{lr}
        b_3(0)=0,\quad b_3'(0)=1,\quad b_3''(0)=0, & \\
     b_3'''(r)= B_i(r). &\\
        \end{array}\right.
    \end{equation}
    Because $B_i\equiv b_1'''$ on $[0,s_i-t]$ and $B_i\equiv b_2'''$ on $[s_i,\tfrac{\pi}{2}+s_i-r_i]$, we know that $b_3\equiv b_1$ on $[0,s_i-t]$ and that $b_3$ satisfies the smoothness condition \eqref{e: smoothness for b} at $0$. Henceforth, for fixed $r_i,s_i$, we shall write $o(1)$ to denote all constants that go to zero as $t\to0$. 
   For small $t$ and $r\in [0,s_i]$, we have that 
    \[
    |b_3''(r)-b''(r)|=\left|\int_0^r\left(B_i(\tau)-b_1'''(\tau)\right)d\tau\right|\leq \int_{s_i-t}^{s_i}\left|B_i(\tau)-b_1'''(\tau)\right|d\tau\leq Ct=o(1),
    \]
    where $C>0$ is a constant independent of $t$ such that $\max_{[s_i-t,s_i]}(|b_1'''|+|b_2'''|)\leq C$. By integrating this estimate, we see that $\|b_3-b\|_{C^2[0,s_i]}\le o(1)$ on $[0,s_i]$.
    Since $b'''\equiv b_3'''$ on $[s_i,\frac{\pi}{2}+s_i-r_i]$, it follows that $\|b_3-b\|_{C^2[s_i,\tfrac{\pi}{2}+s_i-r_i]}\le o(1)$, and thus
    \begin{equation}\label{e: C2 closeness}
        \|b_3-b\|_{C^2[0,\tfrac{\pi}{2}+s_i-r_i]}\le o(1).
    \end{equation}
  Moreover, since $b_3'''\equiv b_2'''$ on $[s_i,\frac{\pi}{2}+s_i-r_i]$, by letting $a_3=b_3b_3'$, we also see that
  \begin{equation}\label{e: C2 closeness of a}
        \|a_3-a_2\|_{C^2[s_i,\tfrac{\pi}{2}+s_i-r_i]}\le o(1).
    \end{equation}

In general, $a_3$ and $b_3$ may not satisfy the boundary conditions \eqref{e: smoothness for a} and \eqref{e: smoothness for b} at $L\approx \frac{\pi}{2}$, so we need to modify them near $r=\frac{\pi}{2}$.
For $\alpha>0$ to be chosen later, define a smooth function $A_{i,\alpha}:\mathbb R\to\mathbb R$ by
\begin{equation*}
    A_{i,\alpha}(r)=\varphi\left(\tfrac{\pi}{2}-r_i-r\right)\frac{-a_3''}{a_3}+\left(1-\varphi\left(\tfrac{\pi}{2}-r_i-r\right)\right)\alpha^2
\end{equation*}
and let $a_4:[0,\infty)\to\mathbb R$ be the unique smooth function, omitting its dependence on $\alpha$, with $a_4=a_3$ on $[0,\tfrac{\pi}{2}-r_i-t]$ and solving
\begin{equation*}
        \left\{\begin{array}{lr}
        a_4\left(\tfrac{\pi}{2}-2r_i\right)=a_3\left(\tfrac{\pi}{2}-2r_i\right),\quad a_4'\left(\tfrac{\pi}{2}-2r_i\right)=a_3'\left(\tfrac{\pi}{2}-2r_i\right), & \\
     a_4''= -a_4\cdot A_{i,\alpha}.&\\
        \end{array}\right.
    \end{equation*}
Set $b_4(r):=\sqrt{\int_0^r2a_4(\tau)\,d\tau}$. Then $a_4=b_4b_4'$, and \eqref{e: C2 closeness of a} implies that
\begin{equation}\label{e: C2 closeness of a_4 t a_2}
   \|a_4-a_2\|_{C^2[s_i,\tfrac{\pi}{2}-r_i-t]}\le o(1).
\end{equation}

In the following, we will show that there exists some $\alpha=\alpha_\#$ such that 
$a_4$ and $b_4$ satisfy the smoothness conditions \eqref{e: smoothness for a} and \eqref{e: smoothness for b}, where $a_4=0$. 
Let $v_i(\alpha),\,c_i(\alpha)>0$ be two smooth functions such that $a_4\ge0$ on $[0,\tfrac{\pi}{2}-r_i+v_i(\alpha)]$ and
\begin{equation*}
    a_4(\tfrac{\pi}{2}-r_i+v_i(\alpha))=0,\qquad a_4'(\tfrac{\pi}{2}-r_i+v_i(\alpha))=-c_i(\alpha).
\end{equation*}
Then it suffices to find $\alpha_\#$ such that $c_i(\alpha_\#)=1$.
To do this, we consider the function 
$$w_{i,\alpha}(r):=c_{0,i}(\alpha)\cdot\alpha^{-1}\sin(\alpha(\tfrac{\pi}{2}-r_i+v_{0,i}(\alpha)-r)),$$ 
where $c_{0,i}(\alpha),v_{0,i}(\alpha)>0$ are two smooth functions in $\alpha$ such that
$w_{i,\alpha}\ge0$ on\linebreak $[\tfrac{\pi}{2}-r_i,\tfrac{\pi}{2}-r_i+v_{0,i}(\alpha)]$ and $w_{i,\alpha}$ solves the following IVP:
\begin{equation*}
        \left\{\begin{array}{lr}
        w\left(\tfrac{\pi}{2}-r_i\right)=a_2\left(\tfrac{\pi}{2}-r_i\right), \qquad
     w'\left(\tfrac{\pi}{2}-r_i\right)=a_2'\left(\tfrac{\pi}{2}-r_i\right), & \\
     w''= -w\cdot \alpha^2.&\\
        \end{array}\right.
\end{equation*}
In particular, for
$$\alpha_0=\sqrt{\frac{4k^2\left(1-\frac{1}{k^2}\cos^2(2s_i)\right)}{\sin^2{(2s_i)}}},\qquad v_0=\frac{\arccos(\tfrac{1}{k}(\cos(2s_i)))}{\alpha_0},$$
direct computation shows that $\alpha_0>2$ and $c_{0,i}(\alpha_0)=1$, and so
\begin{equation*}
    w_{i,\alpha_0}(\tfrac{\pi}{2}-r_i+v_0)=0,\qquad w_{i,\alpha_0}'(\tfrac{\pi}{2}-r_i+v_0)=-c_{0,i}(\alpha_0)=-1.
\end{equation*}
It is also easy to see that $c_{0,i}(\alpha_0-1)<1$ and $c_{0,i}(\alpha_0+1)>1$.
By \eqref{e: C2 closeness of a}, in letting $t\to0$, $a_4$ converges in the $C^1$-sense to $a_2$ on $[s_i,\frac{\pi}{2}-r_i]$ and to $w_{i,\alpha}$ on $[\frac{\pi}{2}-r_i,\frac{\pi}{2}-r_i+\max\{v_{0,i}(\alpha),v_i(\alpha)\}]$.
In particular, this implies that
$$c_i(\alpha)=c_{0,i}(\alpha)+o(1)\qquad\textrm{and}\qquad v_i(\alpha)=v_{0,i}(\alpha)+o(1).$$ Thus, for $t$ sufficiently small, we have that $c_i(\alpha_0-1)<1$ and $c_i(\alpha_0+1)>1$. Continuity and the intermediate value theorem now give us an $\alpha_\#\in(\alpha_0-1,\alpha_0+1)$ such that $\alpha_\#>1$ and $c_i(\alpha_\#)=1$.

Henceforth fixing this $\alpha_\#$, denote the corresponding $v_i(\alpha_\#)$ by $v_i$. Then \linebreak $g_i:=dr^2+a_4^2(r)\eta\otimes\eta+b_4^2(r)g^T$ is a $U(n-1)$-invariant smooth K\"ahler metric on $\mathbb{P}^{n-1}$. It is moreover clear from \eqref{e: C2 closeness} and \eqref{e: C2 closeness of a} that $$\|b_4-b_2\|_{C^2[2r_i,\frac{\pi}{2}-2r_i]}+\|a_4-a_2\|_{C^2[2r_i,\frac{\pi}{2}-2r_i]}\le o(1).$$ So $g_i$ clearly satisfies assertion (c) of the proposition.

It now suffices to show $g_i$ satisfies $\Rm>0$ on real $(1,1)$-forms on $[0,s_i]$, and $\Rm\ge 2+o(1)$ on real $(1,\,1)$-forms on $[s_i,\frac{\pi}{2}-r_i+v_i]$. If this is true, then there exists
a sequence $\varepsilon_i\to0$ such that the rescaled metrics $(1-\varepsilon_i)g_i$ satisfy all assertions of the proposition.
This is equivalent to checking that $a_4,\,b_4,$ satisfy inequalities (1)--(4) of Lemma \ref{conditions on ab} with $\lambda>0$ on $[0,s_i]$ and $\lambda=1+o(1)$ on $[s_i,\frac{\pi}{2}-r_i+v_i]$ (see \eqref{lamb def} for the geometric meaning of $\lambda$).

First, we verify $a_4,\,b_4,$ satisfy inequalities (1) and (3) with $\lambda>0$ on $[0,s_i]$ and $\lambda=1+o(1)$ on $[s_i,\frac{\pi}{2}-r_i-t]$.
Since the expanding Ricci soliton $dr^2+a_1^2(r)\eta\otimes\eta+b_1^2(r)g^T$ satisfies $\Rm>0$ on all $(1,\,1)$-forms, there exists $\delta_i>0$ which tend to $0$ as $i\to\infty$ such that $a_1,\,b_1,$ satisfy inequalities (1)--(4) of Lemma \ref{conditions on ab} on $[0,s_i]$. In particular, inequalities $(1)$ and $(3)$ only involve at most second order derivatives of $b_4=b_3$, so by \eqref{e: C2 closeness}, we see that $b_4$ satisfies $(1)$ and $(3)$ with $\lambda=\delta_i$ on $[0,\,s_i]$.
Similarly, since $h_k=dr^2+a_2^2(r)\eta\otimes\eta+b_2^2(r)g^T$ satisfies $\Rm\ge2$ on $(1,1)$-forms except at $r=0$ and $r=\frac{\pi}{2}$, it follows that $a_2,\,b_2,$ satisfy (1)--(4) with $\lambda=1$. So by \eqref{e: C2 closeness} we see that $a_4,\,b_4,$ satisfy (1) and (3) with $\lambda=1+o(1)$ on $[s_i,\tfrac{\pi}{2}-r_i-t]$.

Next, we show (2) and (4) on $[0,\tfrac{\pi}{2}-r_i-t]$.
Since $b'_2$ is decreasing and $b'_2(\tfrac{\pi}{2}-r_i+s_i)=0$, for sufficiently small $t$, it follows from \eqref{e: C2 closeness} that $b_3'\ge\tfrac12b'_2(\tfrac{\pi}{2}-r_i)>0$ on $[0,\tfrac{\pi}{2}-r_i]$. By
\eqref{e: defn of b} and \eqref{e: C2 closeness}, we therefore find that
$$-\frac{a_3''}{4a_3}=-\frac{3}{4}\frac{b_3''}{b_3}-\frac{1}{4}\frac{b_3'''}{b_3'}
=-\frac{3}{4}\frac{b''}{b}-\frac{1}{4}\frac{B_i}{b'}+o(1)\qquad\textrm{on $\left[0,\,\tfrac{\pi}{2}-r_i-t\right]$}.$$
On $[s_i-t,s_i]$, as $B_i$ takes values between $b_1'''(s_i)+o(1)$ and $b_2'''(s_i)+o(1)$ on $[s_i-t,s_i]$, we see that $-\frac{a_4''}{4a_4}$ is also bounded between $-\frac{a_1''}{4a_1}(s_i)+o(1)\ge\tfrac{\delta_i}{2}$ and $-\frac{a_2''}{4a_2}(s_i)+o(1)\ge1+o(1)$, and hence $-\frac{a_4''}{4a_4}\ge\frac{\delta_i}{4}$, which implies (2) with $\lambda=\tfrac{\delta_i}{2}$ on $[0,\,s_i]$. On $[s_i,\tfrac{\pi}{2}-r_i-t]$, because $-\frac{a_2''}{4a_2}=1$ and by \eqref{e: C2 closeness of a_4 t a_2}, we have $-\frac{a_4''}{4a_4}\ge1+o(1)$, which implies (2) with $\lambda=1+o(1)$. Similarly, we can verify inequality (4).

Secondly, we verify (1)--(4) on $[\tfrac{\pi}{2}-r_i-t, \tfrac{\pi}{2}-r_i]$ with $\lambda=1+o(1)$. Since $-\frac{a_3''}{4a_3}\le1+o(1)<\alpha_\#$ and $\alpha_\#\in(\alpha_0-1,\alpha_0+1)$, we see that
\begin{equation}\label{e: a_4 and a_3}
    \|a_4-a_3\|_{C^1[\tfrac{\pi}{2}-r_i-t,\tfrac{\pi}{2}-r_i]}=o(1),
\end{equation}
and also
\begin{equation}\label{e: a_4''}
    -\frac{a_4''}{4a_4}\ge -\frac{a_3''(\tfrac{\pi}{2}-r_i)}{4a_3(\tfrac{\pi}{2}-r_i)}+o(1).
\end{equation}
\eqref{e: a_4 and a_3} together with \eqref{e: C2 closeness of a} implies that
\begin{equation}\label{e: a_4 and a_2}
    \|a_4-a_2\|_{C^1[\tfrac{\pi}{2}-r_i-t,\tfrac{\pi}{2}-r_i]}+ \|b_4-b_2\|_{C^2[\tfrac{\pi}{2}-r_i-t,\tfrac{\pi}{2}-r_i]}=o(1),
\end{equation}
which immediately implies that (1) and (3) hold with $\lambda=1+o(1)$ on $[\tfrac{\pi}{2}-r_i-t,\tfrac{\pi}{2}-r_i]$. Moreover, together with \eqref{e: a_4''}, this implies (2) and (4).

Lastly, we verify (1)--(4) on $[\tfrac{\pi}{2}-r_i,\tfrac{\pi}{2}-r_i+v_i]$. First of all, $(2)$ follows from the inequalities
\begin{equation}\label{e: 4 LHS 2}
    -\frac{a_4''}{4a_4}\ge \frac{(\alpha_0-1)^2}{4}\ge \frac{\alpha_0^2}{8}\ge \frac{k^2}{16 r_i^2}
    \ge1.
\end{equation}
Let $w_i > 0$ denote generic constants with $w_i \to 0$ as $r_i \to 0$. The values of $w_i$ may vary from line to line in what follows.
Then we have
\begin{equation}\label{pi a/b}
b_4'(r)=\frac{a_4(r)}{b_4(r)}\le \frac{a_4(\tfrac{\pi}{2}-r_i-t)}{b_4(\tfrac{\pi}{2}-r_i-t)}\le w_i.
\end{equation}
It is also easy to see that $|b_4^2-b_4^2(\tfrac{\pi}{2}-r_i-t)|=w_i$, and so the bound
$|b_4^2(\tfrac{\pi}{2}-r_i-t)-\frac{1}{k}|\le w_i$ implies that $|b_4^2-\frac{1}{k}|\le w_i$. Using in addition \eqref{pi a/b}, we derive inequality (1) by
\begin{equation}\label{e: 4 LHS 1}
    \frac{1-(b_4')^2}{b_4^2}=\frac{1-\left(\frac{a_4}{b_4}\right)^2}{b_4^2}\ge k-w_i\ge k-1\ge1.
\end{equation}
Next, by using the fact that $-a_4'$ is increasing, \eqref{e: a_4 and a_2}, and $b_2'(r)\le w_i$, (3) follows from
\begin{eqnarray*}
    -\frac{b_4''}{b_4}=\frac{-a_4'+(b_4')^2}{b_4^2}
    \ge\frac{-a_2'(\tfrac{\pi}{2}-r_i-t)+(b_2')^2(\tfrac{\pi}{2}-r_i-t)}{b_2^2(\tfrac{\pi}{2}-r_i-t)}-w_i= 1-w_i.
\end{eqnarray*}
Finally, by using the inequalities $a_4'\ge a_4'(\tfrac{\pi}{2}-r_i+v_i)=-1$,
$|b_4^2-\frac{1}{k}|\le w_i$, and $0\le b_4'\le w_i$, we derive that
\begin{eqnarray}\label{e: 4 RHS}
 -\frac{b_4''}{b_4}= \frac{-a_4'+(b_4')^2}{b_4^2}&\le k+w_i\le k+1.
\end{eqnarray}
In conclusion, \eqref{e: 4 LHS 1}, \eqref{e: 4 LHS 2}, and \eqref{e: 4 RHS} together give us that
\begin{eqnarray*}
    \frac{2(n-1)(1-\lambda b_4^2(r)-(b_4'(r))^2)}{(n-2) b_4^2(r)}\left(-\frac{a_4''(r)}{4a_4(r)}-\lambda\right)&\ge& \frac{2(n-1)}{(n-2)}(k-1-\lambda)\left(\frac{k^2}{16 r_i^2}
    -\lambda\right)\\
    &>& (k+1-\lambda)^2\ge \left(-\frac{b_4''(r)}{b_4(r)}-\lambda\right)^2,
\end{eqnarray*}
where in the penultimate
inequality we have used the fact that $\lambda = 1 + o(1)$ and that $k$ is sufficiently large so that $2(k-3)(k^2 - 32) > 16(k+1)^2$. This yields $(4)$ for all large $k$, and thus completes the proof.
\end{proof}

\section{Ricci flows emanating from the singular metric}

Let $n\geq3$. The goal of this section is to show that there is a Ricci flow that smooths out the singular metric \eqref{e: metric with conical singularity}, namely
$$h_k=dr^{2}+\frac{\sin^2(2r)}{4k^2}\eta\otimes \eta+\frac{\sin^2(r)}{k}g^{T},$$
where $r(\cdot):\mathbb{P}^{n-1}\to[0,\frac{\pi}{2}]$ is the radial function, $r^{-1}(0)=\{p\}$, and $r^{-1}(\frac{\pi}{2})$ is biholomorphic to $\mathbb P^{n-2}$.
We construct this Ricci flow by taking 
limits of the Ricci flows coming out of the metrics $\{g^i_k\}_{i=1}^{\infty}$ from Proposition \ref{l: smooth metric with nonnegative Rm} that converge to $h_k$. These Ricci flows are smooth, have uniform existence time, and satisfy a uniform curvature decay $|\Rm|\le\frac{C}{t}$.
We will show that these smooth Ricci flows satisfy $\Rm\ge0$ everywhere and $\Rm\ge2$ outside of arbitrarily small subsets, so eventually the Ricci flows converge to a limit flow that satisfies $\Rm\ge 2$ on real $(1,\,1)$-forms \emph{everywhere}. 

The curvature condition $\Rm\ge2$ on real $(1,\,1)$-forms is equivalent to $\lambda\ge1$, where
$\lambda\in\R$ is half of the infimum of $\Rm$ on real $(1,\,1)$-forms given by \eqref{lleqn} in complex dimension $(n-1)$, namely
\begin{equation}\label{lamb def}
    \lambda:=\frac{1}{2}\inf_{u\in\Lambda^{1,1}_\mathbb{R}(M)\setminus\{0\}} \frac{\Rm(u,\bar{u})}{|u|_*^2},
\end{equation}
where for $u=i\,u_{\alpha\bar \beta}\,dz^{\alpha}\wedge dz^{\bar \beta}$, the norm $|u|_*$ is defined by 
$$|u|_*^2=\sum_{\alpha,\beta, \gamma, \eta=1}^{n-1}(\hat g_{\alpha\bar{\beta}}\hat g_{\gamma\bar{ \eta}}+\hat g_{\alpha\bar  \eta}\hat g_{\gamma\bar \beta})u^{\alpha\bar\beta}u^{\gamma\bar{ \eta}},$$
and we have $\Rm(u,\,\bar{u})=-\sum_{\alpha,\beta, \gamma, \eta=1}^{n-1}\hat R_{\alpha\bar{\beta}\gamma\bar{\eta}}u^{\alpha\bar\beta}u^{\gamma\bar{\eta}}$.
Using the fact that $u^{\alpha\bar\beta}=g^{\alpha\bar\varepsilon}g^{\eta\bar\beta}u_{\eta\bar \varepsilon}$ and $\overline{u_{\alpha\bar \beta}}=u_{\beta\bar\alpha}$ for any $\alpha$ and $\beta$, we have
\[
  |u|_*^2= \left|\sum_{\alpha, \beta=1}^{n-1}g_{\alpha\bar{\beta}}u^{\alpha\bar\beta}\right|^2+ u_{\alpha\bar \beta}g^{\alpha \bar \gamma}g^{\eta\bar\beta}\overline{u_{\gamma\bar\eta}}\ge  u_{\alpha\bar \beta}g^{\alpha \bar \gamma}g^{\eta\bar\beta}\overline{u_{\gamma\bar\eta}} \ge 0
\]
with equality if and only if $u=0$.

We first observe that $\lambda$ is a supersolution of the heat equation.

\begin{lem}\label{l: heat inequality}
Let $(M,g(t))$, $t\in[0,T]$, be a smooth K\"ahler-Ricci flow (not necessarily complete) of complex dimension $n-1$ with nonnegative curvature operator. Then 
\begin{equation*}
  \partial_t\lambda\ge\Delta \lambda
    \end{equation*}
holds in the following barrier sense: for any $(q,\tau)\in M\times (0,T)$, there exists a neighborhood $U\subset M\times (0,T)$ of $(q,\tau)$ and a $C^{\infty}$ (upper barrier) function $\phi:U\to\R$ such that $\phi\ge \lambda$ on $U$, with equality at $(q,\tau)$ and
\begin{equation*}
    \partial_t\phi\ge\Delta \phi\quad\textnormal{at }(q,\tau).
\end{equation*}
\end{lem}

\begin{proof}
Fix $(p_0,\,t_0)$ and suppose that the infimum in \eqref{lamb def} is attained at a nonzero $(1,\,1)$-form $w_0\in\Lambda^{1,1}_{p_0}$. We may assume by scaling that $|w_0|_{*,\,g_{t_0}}=1$. We first extend $w_0$ to a spacetime neighborhood $V\times(t_0-\varepsilon,t_0+\varepsilon)$ of $(p_0,t_0)$ in the following way. Let $\{e_k\}_{k=1}^{2n-2}$ be an orthonormal basis of $(T_{p_0}M,g_{t_0}|_{T_{p_0}M})$ such that $J(e_{2\alpha-1})=e_{2\alpha}$. By abusing notation and shrinking the neighborhood $V$ of $p_0$ if necessary, we extend $\{e_k\}_{k=1}^{2n-2}$ to a smooth local orthonormal frame near $p_0$ via parallel translation. Then $J(e_{2\alpha-1})=e_{2\alpha}$ on $V$ and we have at $(x_0,t_0)$ that
\begin{equation}\label{para frm}
\nabla_{e_j}e_k=0  \quad\text{and}\quad \nabla_X\nabla_X e_i=0 \quad\text{for any } X\in T_{p_0}M \text{ at  } p_0.
\end{equation}

Let $\nabla_t$ denote the natural space-time extension of the Levi-Civita connection $\nabla^{g(t)}$
so that it is compatible with the metric; that is to say, $\partial_t |X|^2_{g(t)}=2\langle \nabla_tX,X\rangle_{g(t)}$ for any time-dependent vector field $X$.
We extend $\{e_k\}_{k=1}^{2n-2}$ to a time-dependent local frame $\{e_k(t)\}_{k=1}^{2n-2}$ so that $e_k(x,t_0)= e_k(x)$ for all $x\in V$, and 
\begin{align}\label{Uh ODE}
\nabla_te_k(x,t) &=\partial_t(e_k(x,t))-\Ric(e_k(x,t))=0.
\end{align}
By ODE theory, after choosing a smaller $V$ if necessary, we can solve the system on\linebreak $(t_0-\varepsilon,t_0+\varepsilon)$ for some small $\varepsilon>0$ with $e_k(x,t)$ smooth in $(x,\,t)\in V\times(t_0-\varepsilon,t_0+\varepsilon) $.
It also follows from the fact that $J\circ \Ric=\Ric\circ J$ and the uniqueness of solutions of ODEs that $J(e_{2\alpha-1}(x,t))=e_{2\alpha}(x,t)$. Let $\{\theta_k(t)\}_{k=1}^{2n-2}$ denote the dual frame. We introduce the corresponding holomorphic and antiholomorphic frames:
\[
E_{\alpha}(x,t):=\frac{1}{2}\left(e_{2\alpha-1}(x, t)-i\,e_{2\alpha}(x,t)\right)\qquad\textrm{and}\qquad E_{\bar \alpha}(x,t):=\frac{1}{2}\left(e_{2\alpha-1}(x, t)+i\,e_{2\alpha}(x,t)\right);
\]
and their dual frames:
\[
\Theta_{\alpha}(x,t):=\theta_{2\alpha-1}(x, t)+i\,\theta_{2\alpha}(x,t)\qquad\textrm{and}\qquad \Theta_{\bar \alpha}(x,t):=\theta_{2\alpha-1}(x, t)-i\,\theta_{2\alpha}(x,t).
\]
Thanks to \eqref{para frm} and \eqref{Uh ODE}, at $(p_0,t_0)$ we have that
\begin{equation}\label{para Th}
    \nabla_t\Theta_{\alpha}=\nabla \Theta_{\alpha}=\nabla \Theta_{\bar\alpha}=\Delta \Theta_{\bar\alpha}=0.
\end{equation}
Since $w_0$ is a real $(1,\,1)$-form, we may assume that $w_0=i\,w_{\alpha\bar\beta}\Theta_{\alpha}(p_0,t_0)\wedge \Theta_{\bar\beta}(p_0,t_0)$ for some complex numbers $\overline{w_{\alpha\bar\beta}}=w_{\beta\bar\alpha}$, $\alpha, \beta=1,\dots, n-1$.
Let $w_t$ be the two-form given by
\[
w_t=i\,w_{\alpha\bar\beta}\,\Theta_{\alpha}(x,t)\wedge \Theta_{\bar\beta}(x,t)
\]
on $V\times(t_0-\varepsilon,t_0+\varepsilon)$.
Then by the K\"ahlerity of the metric, $w_t$ is $J$-invariant, i.e., $w_t(J(\cdot),\,J(\cdot))=w_t(\cdot\,,\, \cdot)$, and so it is of type $(1,\,1)$.
Moreover, by \eqref{para Th} at $(p_0,t_0)$, we have that
\begin{equation}\label{grad hess w}
   \nabla_t w_t=0,\quad\nabla w_t=0, \quad\text{and} \quad \Delta w_t=0.
\end{equation}

Next we define the smooth function $\varphi$ by
\[
\varphi(x,t):=\frac{1}{2}\frac{\Rm(w_t,\overline{w}_t)}{|w_t|_*^2}=\frac{1}{2}\Rm(w_t,\overline{w}_t),
\]
where the curvature and norms are all with respect to $g(t)$, and we use that $|w_t|_*\equiv1$.
It can be seen from the definition that $\varphi\ge \lambda$ on $V\times(t_0-\varepsilon,t_0+\varepsilon)$ and equality holds at $(p_0,t_0)$.
Recall that $$\nabla_t\Rm=\Delta\Rm+Q(\Rm)=\Delta\Rm+\Rm^2+\Rm^{\sharp},$$
and both $\Rm^2$ and $\Rm^\sharp$ are nonnegative operators when the curvature operator $\Rm$ is nonnegative. Hence by \eqref{grad hess w}, when evaluated at $(p_0,t_0)$, we have that
\begin{eqnarray*}
(\partial_t-\Delta)\varphi=\frac12(\partial_t-\Delta)\left(\Rm(w_t,\overline{w}_t)\right)&=&((\nabla_t-\Delta)\Rm) (w_t,\overline{w}_t) \\
&=& \Rm^2(w_t,\overline{w}_t)+\Rm^{\sharp}(w_t,\overline{w}_t)\ge 0.
\end{eqnarray*}
This proves the lemma.
\end{proof}

From Lemma \ref{l: heat inequality}, using the maximum principle, it is easy to see that the lower bound $\lambda\ge1$ is preserved along any smooth Ricci flow. In Lemma \ref{l: limit flows g_k} below, we will show that this also holds for the Ricci flow coming out of the singular metric $h_k$, where $\lambda\ge1$ only holds on the smooth part. This is achieved by a heat kernel estimate on the smooth Ricci flows that converge at positive times to the limit flow that smooths out $h_k$. 
The heat kernel estimate guarantees that the heat diffuses
evenly from all points
as $t$ increases from $0$,
thereby allowing one to ignore the lack of sufficiently positive 
curvature from the arbitrarily small regions. We first recall some basic facts about heat kernels.

Suppose that $(M,\,g(t))$ is a smooth complete Ricci flow for $t\in[a,\,b]$ with bounded curvature. Let $G(x,t;y,s)$ be the heat kernel of the heat equation $\partial_tu=\Delta u$, i.e., 
\begin{eqnarray*}
    (\partial_t-\Delta_{x,t})G(\cdot, \cdot; y,s)=0 &\text{  and  }& \lim_{t\to s^+}G(\cdot, t; y,s)=\delta_y(\cdot),\\
    (-\partial_s-\Delta_{y,s}+R)G(x,t;\cdot, \cdot)=0 &\text{  and  }& \lim_{s\to t^-}G(x, t; \cdot,s)=\delta_x(\cdot).
\end{eqnarray*}
In particular, we have $\int_M G(x,t;y,s)\,d_sy=1$ for all $s<t$; see \cite[Chapter 26]{RFTandA3} for more discussion of the heat kernel. We also recall the heat kernel upper bound given in terms of the lower bound on the pointed Nash entropy by Bamler in \cite[Theorem 7.1]{bamler2020entropy}:
\begin{equation}\label{e: heat kernel upper}
    G(x,t;y,s)\le \frac{C}{(t-s)^{\tfrac{n}{2}}}\exp(-\mathcal{N}_{(x,t)}(t-s)),
\end{equation}
where $C>0$ is a constant depending on $R_{\min}\cdot(t-s)$ and $R_{\min}$ is a lower bound of the scalar curvature. We furthermore recall the lower bound on the Nash entropy $\mathcal{N}_{(x,t)}(\cdot)$ in terms of the lower bound on the volume ratio in \cite[Theorem 8.1]{bamler2020entropy}:
\begin{equation}\label{e: nash entropy lower}
    \frac{\operatorname{vol}_{g(t)}(B_{g(t)}(x,r))}{r^n}\le C\exp(\mathcal{N}_{(x,t)}(r^2))\exp(C_0),
\end{equation}
where $C_0>0$ is a dimensional constant and $C>0$ is some constant depending on $R_{\min}r^2$.

Now we prove the main result of this section.
\begin{lem}\label{l: limit flows g_k}
Let $n\ge3$. Then for each fixed $k\in\mathbb{N}$, there exist $C_k,\,T_k>0$ and a Ricci flow $(\mathbb{P}^{n-1},g_k(t))$, $t\in[0,\,T_k]$, such that
\begin{enumerate}[label=\textnormal{(\arabic{*})}, ref=(\arabic{*})]
    \item $|\Rm|\le\frac{C_k}{t}$ and $g_k(t)\to h_k$ smoothly on any compact subset of $r^{-1}((0,\frac{\pi}{2}))$ as $t\to0$;
    \item The metrics $d_{g_k(t)}$ converge to $d_{h_k}$ as $t\to0$;
    \item $\lambda(x,t)\ge1$ for all $(x,t)\in\mathbb P^{n-1}\times (0,T_k]$.
\end{enumerate}
\end{lem}

\begin{proof}
For each fixed $k\in\mathbb{N}$, let $\{(\mathbb{P}^{n-1},g_k^i)\}_{i=1}^\infty$ be the sequence of metrics from Proposition \ref{l: smooth metric with nonnegative Rm}. Then $(\mathbb{P}^{n-1},g_k^i)$ converges smoothly to the singular metric $(\mathbb{P}^{n-1},h_k)$ on any compact subset in $r^{-1}((0,\frac{\pi}{2}))$ as $i\to\infty$. Let $(\mathbb{P}^{n-1},g_k^i(t))_{t\in[0,T_k^i)}$ be the Ricci flow starting from $g_k^i$, where $T_k^i>0$ is the maximal existence time.
It is clear that there exists $v_k>0$ such that $\operatorname{vol}_{h_k}(B_{h_k}(x,1))\ge v_k$ for all $x\in\mathbb P^{n-1}$, so that $\operatorname{vol}_{g_k^i}(B_{g_k^i}(x,1))\ge v_k$ everywhere for all large $i$. Since $g_k^i(0)$ in addition satisfies $\Rm\ge0$, it follows from \cite{BCRW, Lai2019}, the existence theorem for Ricci flows under almost nonnegative curvature and non-collapsing conditions, that there exist $0<T_k<1$ and $C_k>0$ such that $T_k^i\ge T_k$ for all $i$, and $|\Rm|\le\frac{C_k}{t}$ for all $t\in(0,T_k]$.
Moreover, using the initial volume bound and the curvature bound $|\Rm|\le\frac{C_k}{t}$, we can argue exactly as in \cite[Corollary 6.2]{Simon2012} to get 
\begin{equation}\label{e: volume lower bound}
      \operatorname{vol}_{g^i_k(t)}(B_{g^i_k(t)}(x,1))\ge \widetilde v_k
\end{equation}
for some $\widetilde v_k>0$ depending on $v_k, C_k$.

After passing to a subsequence, these flows converge smoothly to a Ricci flow $(\mathbb{P}^{n-1},g_k(t))$, $t\in(0,T_k]$, satisfying
$|\Rm|\le\frac{C_k}{t}$ for some $T_k,C_k>0$.
By Proposition \ref{l: smooth metric with nonnegative Rm}(c), the curvature of the metrics $g_k^i$ on any compact subset of $r^{-1}((0,\frac{\pi}{2}))$ has a uniform bound for all large $i$. By Perelman's pseudolocality theorem and Shi's derivative estimates (see for example \cite{KL}), this implies that the Ricci flows $\{g^i_k(t)\}$ uniformly smoothly
converge on any compact subset of $r^{-1}((0,\frac{\pi}{2}))$ as $t\to0$. The limit Ricci flow $g_k(t)$ therefore smoothly converges to $h_k$ on any compact subset of $r^{-1}((0,\frac{\pi}{2}))$ as $t\to0$.
Finally, convergence of metric spaces $(\mathbb P^{n-1},d_{g_k(t)})$ to $(\mathbb P^{n-1},d_{h_k})$ as $t\to0$ follows from standard distance distortion estimates; see for example \cite{BCRW, Lai2019, mollification}.     

Fix $(x,\,t)\in\mathbb P^{n-1}\times(0,T_k]$.
It remains to verify that $\lambda(x,t)\ge1$ for any $x\in\mathbb P^{n-1}$ and $t\in(0,T_k]$. 
To begin with, to estimate the heat kernel $G(x,t;y,s)$, by using \eqref{e: volume lower bound} in \eqref{e: nash entropy lower} we obtain $\mathcal{N}_{(x,t)}(t-s)\ge C^{-1}$ 
for all $s\in[0,\tfrac{t}{2}]$, where here and below $C>0$ denotes a generic constant that only depends on the point $(x,\,t)$. It follows from \eqref{e: heat kernel upper} that for all $(y,s)\in\mathbb P^{n-1}\times(0,\tfrac{t}{2}]$,
\begin{equation}\label{e: hk integral is small}
    G(x,t;y,s)\le C.
\end{equation}
Next, since $\lambda\ge1$ on $r^{-1}((0,\frac{\pi}{2}))\times\{0\}$ (see Lemma \ref{conditions on ab}) and $\lambda$ is continuous on $r^{-1}((0,\frac{\pi}{2}))\times[0,T_k]$, for any $\varepsilon>0$ we can find $\delta\in(0,\tfrac{t}{2})$ with $\delta\to0$ as $\varepsilon\to0$ such that for all $(y,s)\in r^{-1}((\delta,\frac{\pi}{2}-\delta))\times[0,\delta]$, we have 
\begin{equation}\label{e: continuity of lambda}
    \lambda(y,s)\ge 1-\varepsilon.
\end{equation}
By Lemma \ref{l: heat inequality}, $\lambda$ is a supersolution to the heat equation, and so we have the lower bound
    \begin{equation*}
    \begin{split}
        \lambda(x,t)&\ge\int_{\mathbb P^{n-1}} \lambda(y,\delta)\,G(x,t;y,\delta)\,d_{\delta}y\\
        &\ge(1-\varepsilon)\int_{r^{-1}(\delta,\tfrac{\pi}{2}-\delta)}G(x,t;y,\delta)\,d_{\delta}y\\
        &\ge(1-\varepsilon)\left(1-C\cdot \operatorname{vol}_{\delta}\left(r^{-1}\left([0,\delta]\cup\left[\tfrac{\pi}{2}-\delta,\tfrac{\pi}{2}\right]\right)\right)\right),
    \end{split}
    \end{equation*}
where in the second inequality
we have used \eqref{e: continuity of lambda}, and in the last inequality we have used \linebreak $\int_{\mathbb P^{n-1}}G(x,t;y,s)\,d_sy=1$ for all $s<t$, as well as \eqref{e: hk integral is small}.
Since the Hausdorff measure is weakly
continuous under the Gromov-Hausdorff convergence of $(\mathbb P^{n-1},d_{g_k(t)})$ to $(\mathbb P^{n-1},d_{h_k})$ as $t\to0$ (cf.~\cite{BGP, Cold97}), we deduce that $\operatorname{vol}_{\delta}\left(r^{-1}\left([0,\delta]\cup\left[\tfrac{\pi}{2}-\delta,\tfrac{\pi}{2}\right]\right)\right)\to0$ as $\delta\to0$.
Letting $\varepsilon\to0$, this implies that $\lambda(x,t)\ge1$.
\end{proof}

We conclude this section with the following corollary which implies Theorem \ref{t: GH convergence on CPn}.

\begin{cor}\label{p: heat kernel}
Let $n\ge2$. Then there exists a sequence of $U(n-1)$-invariant metrics $\{g_k\}_{k=1}^{\infty}$ on $\mathbb{P}^{n-1}$ 
satisfying $\lambda\ge1$ such that $(\mathbb{P}^{n-1},d_{g_k})$ converges in the Gromov-Hausdorff sense to the interval $[0,\tfrac{\pi}{2}]$.
\end{cor}

\begin{proof}
    For $n=2$, we have that $U(1)=SO(2)$, and so the sequence of $SO(2)$-invariant metrics is constructed in \cite{Lai2020_flying_wing}. We may therefore assume that $n\ge3$. Let $g_k(t)$ be the Ricci flow constructed in Lemma \ref{l: limit flows g_k} that smooths out the singular metric $h_k$. Since $h_k$ converges to $[0,\tfrac{\pi}{2}]$ and $g_k(t)$ converges to $h_k$ as $t\to0$ in the Gromov-Hausdorff sense, we can find a sequence $t_k\to0$ so that $g_k(t_k)$ satisfies the assertion of the corollary. 
\end{proof}

\section{Construction of the steady Ricci solitons}

In this section, we complete the construction of the family of $U(1)\times U(n-1)$-invariant steady K\"ahler-Ricci solitons for $n\ge 2$ and prove Theorem \ref{t: existence of new}. We present the proof in higher dimensions and point out the essential differences for complex dimension $2$.
We recall the flying wing construction. We begin by producing a sequence of interpolation families of degenerating smooth metrics on $\mathbb{P}^{n-1}$. In light of \cite{{con-der}}, there exists a sequence of families of expanding gradient K\"ahler-Ricci solitons asymptotic to the K\"ahler cones on $\mathbb{C}^n$ over these interpolation families of metrics on $\mathbb{P}^{n-1}$. The desired steady K\"ahler-Ricci solitons then arise as limits of some well-chosen elements from these families of expanding solitons.
In the proof, we use some standard notions and facts from Alexandrov geometry and metric comparison. We refer the reader to \cite[Section 2.4]{Lai2022_O(2)} for a more detailed exposition.

\begin{proof}[Proof of Theorem \ref{t: existence of new}]
By Corollary \ref{p: heat kernel}, we have a sequence of $U(n-1)$-invariant smooth K\"ahler metrics $(\mathbb{P}^{n-1},h_i)$ satisfying the following conditions:
\begin{enumerate}
    \item $\diam(\mathbb{P}^{n-1},\,h_i)\rightarrow\tfrac{\pi}{2}$ as $i\rii$;
   
    \item  \label{Rm bdd in const2} $\Rm(h_{i})> 2$ on real $(1,\,1)$-forms (or equivalently, sectional curvature $>4$ if $n=2$);
    \item $\lim_{i\rii}\operatorname{vol}_{h_i}(\mathbb{P}^{n-1})=0$.
\end{enumerate}
In particular, by \eqref{Rm bdd in const2}, 
$h_i$ has positive holomorphic bisectional curvature for all $i$. For each $i$, we run the normalized K\"ahler-Ricci flow $h_i(t)$ on $(\mathbb{P}^{n-1}, h_i)$. Then by the results of Collins-Sz\'ekelyhidi \cite{CS2016} and Tian-Zhang-Zhang-Zhu \cite{TZZZ2013}, the normalized flow exists for all time and converges smoothly to a K\"ahler-Einstein metric on $\mathbb{P}^{n-1}$ as $t\to\infty$, which is a positive multiple of the Fubini-Study metric $g_{FS}$ having the same volume as $h_i$, by the uniqueness of K\"ahler-Einstein metrics up to biholomorphism \cite{BM87}. 
Since there is a positive lower bound for $\Rm$ on real $(1,\,1)$-forms for all times along the normalized K\"ahler-Ricci flow and $\Rm>2$ for $t=0$, 
for each $i$, we can find a smooth positive function $v_i(t)$ such that $2v_i(t)$ is smaller than the lower bound of the curvature operator of $h_i(t)$ on real $(1,1)$-forms and $\lim_{t\to \infty}v_i(t)=1$ (see \cite{Lai2020_flying_wing}). Replacing $h_i(t)$ by $v_i^2(t)h_i(t)$ and introducing a suitable reparametrization $t=t(\mu)$ where $\mu\in [0,1]$ yields a smooth family of $U(n-1)$-invariant K\"ahler metrics $(h_{i,\,\mu})_{\mu\in[0,\,1]}$ on $\mathbb{P}^{n-1}$ with:
\begin{enumerate}
\item\label{p: endpoint} $h_{i,\,1}=h_i$ and $h_{i,\,0}=c_ig_{FS}$ for some $c_i>0$;
    \item\label{p: diameter} $\diam(\mathbb{P}^{n-1},\,h_{i,\,1})=\diam(\mathbb{P}^{n-1},\,h_i)\rightarrow\tfrac{\pi}{2}$ as $i\rii$;
    \item\label{Rm bdd in const} $\Rm(h_{i,\,\mu})> 2$ on real $(1,\,1)$-forms (or equivalently, sectional curvature $>4$ if $n=2$) for all $\mu\in[0,1]$;
    \item\label{p: volume collapsing} $\operatorname{vol}_{h_{i,\,\mu}}(\mathbb{P}^{n-1})\ri0$ as $i\to\infty$ uniformly for all $\mu\in[0,1]$.
\end{enumerate}

For each fixed $i$, since the metrics $h_{i,\,\mu}$ on $\mathbb P^{n-1}$ are $U(n-1)$-invariant and vary smoothly in $\mu$, we see from Proposition \ref{iceland} that there is a $U(1)\times U(n-1)$-invariant path of K\"ahler cone metrics $C_{i,\,\mu}$ on $\mathbb C^{n}\setminus\{0\}$ varying smoothly in $\mu$. The links of these cones are Sasaki metrics on $S^{2n-1}$ over $h_{i,\,\mu}$. The above conditions \eqref{p: diameter} and \eqref{p: volume collapsing} satisfied by $h_{i,\,\mu}$, together with \eqref{diam sasa} and \eqref{vol sasa} in Proposition \ref{iceland}, then imply that the length of the orbit of the flow of the Reeb vector field tends to $0$ as $i\to\infty$, and the Sasaki metrics over $h_{i,\,\mu}$ also satisfy conditions \eqref{p: diameter} and \eqref{p: volume collapsing}. Moreover, by condition (3) and Theorem \ref{cone vs base}, the K\"ahler cones $C_{i,\,\mu}$ have positive curvature operator on real $(1,\,1)$-forms in the transverse direction. 
Therefore by Theorem \ref{lift} and the local uniqueness theorem \cite[Theorem 5.3]{con-der}, there exists a unique deformation of the $U(n)$-invariant expanding K\"ahler-Ricci soliton $g_{i,\,0}$ of Cao \cite{Cao1997jdg} on $\mathbb C^{n}$ by a smooth path of $U(1)\times U(n-1)$-invariant expanding gradient K\"ahler-Ricci solitons $(M_{i,\,\mu},  g_{i,\,\mu},p_{i,\,\mu})$ with positive curvature operator on real $(1,\,1)$-forms asymptotic to the K\"ahler cone over $h_{i,\,\mu}$. Here, $p_{i,\,\mu}$ denotes the critical points of the soliton potential function and $R(p_{i,\,\mu})=1$. 

Since the Sasaki metrics over $h_{i,\,\mu}$ satisfy \eqref{p: volume collapsing}, it follows that the asymptotic volume ratio of the expanding solitons $g_{i,\,\mu}$ decreases to zero uniformly for all $\mu$ as $i\rii$, and so by the same argument as in \cite[Lemma 2.3]{Lai2020_flying_wing}, we can show that for any sequence $\mu_i\in[0,\,1]$, the $U(1)\times U(n-1)$-invariant expanding gradient K\"ahler-Ricci solitons $(M_{i,\,\mu_i}, g_{i,\,\mu_i}, p_{i,\,\mu_i})$ converge smoothly in the Cheeger-Gromov sense to a $U(1)\times U(n-1)$-invariant steady gradient K\"ahler-Ricci soliton. In particular, on one hand, for $\mu_i\equiv0$, since the K\"ahler cone is $U(n)$-invariant, the expanding solitons $g_{i,\,0}$ are also $U(n)$-symmetric by the uniqueness theorem in \cite{CF16}, therefore subsequentially converge to a limit steady soliton $(M_{\infty,0},g_{\infty,0},p_{\infty,0})$. By Cao's uniqueness result \cite[Proposition 2.1]{Cao1996}, this limit has to be Cao's positively curved $U(n)$-invariant steady soliton constructed in \cite{Cao1996}. 
On the other hand, for $\mu_i\equiv1$, 
since $\operatorname{vol}_{h_{i,\,\mu_i}}(\mathbb{P}^{n-1})\ri0$ as $i\to\infty$,
the expanding solitons $(M_{i,\,1},g_{i,\,1},p_{i,\,1})$ converge to the steady soliton $(M_{\infty,1},g_{\infty,1},p_{\infty,1})$ (see the proof of \cite[Lemma 2.3]{Lai2020_flying_wing}). Moreover, by Remark \ref{1eigv}, $\Ric_{g_{\infty,1}}$ has only one strictly positive distinct eigenvalue at the critical point of the steady soliton. 
Since the sequence of Sasaki metrics over $h_{i,\,1}$ also converges to the interval $[0,\tfrac{\pi}{2}]$ in the Gromov-Hausdorff sense, it follows that the asymptotic K\"ahler cones of the expanding solitons converge to the positive quadrant in the Euclidean plane $\R^2$.

We now show that the steady K\"ahler-Ricci soliton $(M_{\infty,1},g_{\infty,1},p_{\infty,1})$ must split as a product of a $(2n-2)$-dimensional $U(n-1)$-invariant steady K\"ahler-Ricci soliton and a two-dimensional $U(1)$-invariant steady K\"ahler-Ricci soliton. To this end, first note that since it exhibits $U(1)\times U(n-1)$ symmetry, there exists a $(2n-2)$-dimensional totally geodesic submanifold $N_1$ fixed by the $U(1)$-isometry and a two-dimensional totally geodesic submanifold $N_2$ fixed by the $U(n-1)$-isometry; cf.~Proposition \ref{iceland} for a precise description of the group action. Similarly, we write $N_{1,i},N_{2,i},$ for the fixed point sets in the expanding soliton $(M_{i,\,1},g_{i,\,1},p_{i,\,1})$ of the $U(1)$- and $U(n-1)$-action, respectively. 
Note that we can find a $U(1)$-invariant point and a $U(n-1)$-invariant point in the Sasaki metric over $h_{i,1}$ such that their distance is greater than $\frac{\pi}{2}-\varepsilon_i$, where here and below $\varepsilon_i$ denotes a general sequence such that $\varepsilon_i\to0$ as $i\to\infty$.
So, arguing as in \cite[Lemma 4.2]{Lai2020_flying_wing},
for any $x_{i}\in N_{1,i},\,y_{i}\in N_{2,i},$ with $d_{g_{i,\,1}}(x_{i},p_{i,\,1})=d_{g_{i,\,1}}(y_{i},p_{i,\,1})$, the comparison angle satisfies 
$\widetilde\measuredangle x_{i}p_{i,\,1}y_{i}\ge \frac{\pi}{2}-\varepsilon_i$. In particular, let $\{C_k\}_{k=0}^{\infty}$ be a sequence of positive numbers going to infinity as $k\to\infty$. Then for any points $x_{k,i}\in N_{1,i},\,y_{k,i}\in N_{2,i},$ with $d_{g_{i,\,1}}(x_{k,i},p_{i,\,1})=d_{g_{i,\,1}}(y_{k,i},p_{i,\,1})=C_k$, we have 
\begin{equation}\label{e: comparison angle at k}
    \widetilde\measuredangle x_{k,i}p_{i,\,1}y_{k,i}\ge \frac{\pi}{2}-\varepsilon_i.
\end{equation} 

On $(M_{\infty,1},g_{\infty,1},p_{\infty})$, on the one hand, for any points $x_{k,\infty}\in N_1,\,y_{k,\infty}\in N_2,$ with $$d_{g_{\infty,1}}(x_{k,\infty},p_{\infty,1})=d_{g_{\infty,1}}(y_{k,\infty},p_{\infty,1})=C_k,$$ let 
$\gamma_{1,k},\gamma_{2,k}:[0,C_k]\to M_{\infty,1}$ be the unit speed minimizing geodesics connecting $p_{\infty,1},x_{k,\infty},$ and $p_{\infty,1},y_{k,\infty}$, and let 
$\gamma_1,\gamma_2:[0,\infty)\to M_{\infty,1}$ be two unit speed geodesic rays starting from $p_{\infty,1}$ obtained as any subsequential limit of $\gamma_{1,k},\gamma_{2,k}$.
Then \eqref{e: comparison angle at k} implies that the asymptotic angle between $\gamma_1,\gamma_2,$ satisfies
\begin{equation}\label{e: bigger than pi/2}
    \measuredangle \gamma_1\gamma_2=\lim_{k\to\infty}\widetilde\measuredangle x_{k,\infty}p_{\infty,1}y_{k,\infty} \ge\tfrac{\pi}{2}.
\end{equation}
On the other hand, since the $U(1)\times U(n-1)$-action induces an isometric action on the tangent space at the unique fixed point
$p_{\infty,1}$ and induces an orthogonal decomposition $T_{p_{\infty,1}}M_{\infty,1}=V_1\perp V_2$, where $V_1,V_2,$ are the $(2n-2)$-dimensional and $2$-dimensional subspaces fixed by the $U(1)$- and $U(n-1)$-actions respectively, and $U(n-1)$ acts on $V_1$ transitively and $U(1)$ acts transitively on $V_2$,
suppose that $\gamma'_{1,k}(0)=v_{1,k}+v_{2,k}$ and $\gamma'_{2,k}(0)=w_{1,k}+w_{2,k}$,
where $v_{1,k},w_{1,k}\in V_1$ and $v_{2,k},w_{2,k}\in V_2$. Then 
\begin{equation}\label{e: gamma1 gamma2 less than pi/2}
    \langle\gamma'_{1,k}(0),\gamma'_{2,k}(0)\rangle=\langle v_{1,k},w_{1,k}\rangle+\langle v_{2,k},w_{2,k}\rangle.
\end{equation}
By replacing $\gamma_{1,k}$ with its image under a suitable isometry in $U(n-1)$, we can keep invariant $v_{2,k}$ and replace $v_{1,k}$ by any vector of the same norm in $V_1$, and thus assume that $\langle v_{1,k},w_{1,k}\rangle\ge0$ and that equality is achieved if and only if $v_{1,k}=0$ or $w_{1,k}=0$. Similarly, by replacing $\gamma_{2,k}$ with its image under a suitable isometry in $U(1)$, we can keep invariant $w_{1,k}$ and replace $w_{2,k}$ by any vector in $V_2$, and thus assume that $\langle v_{2,k},w_{2,k}\rangle\ge0$ and that equality is achieved if and only if $v_{2,k}=0$ or $w_{2,k}=0$.

We claim that, after replacing $x_{k,\infty},\,y_{k,\infty},$ with their images under suitable $U(n-1)$- and $U(1)$-actions respectively, we can assume that
$\langle v_{i,k},w_{i,k}\rangle\ge 0$ for $i=1,2$, and thus
the angle formed by $\gamma_{1,k},\gamma_{2,k},$ at $p_{\infty,1}$ satisfies
\begin{equation}\label{e: angle at p}
    \measuredangle (\gamma'_{1,k}(0),\gamma'_{2,k}(0))\le\tfrac{\pi}{2}.
\end{equation}
Indeed, by replacing $\gamma_{1,k}$ with its image under the $U(n-1)$-action mapping $v_{1,k}$ to $-v_{1,k}$ if necessary, we have that $\langle v_{1,k},w_{1,k}\rangle\ge0$. Then by replacing $\gamma_{2,k}$ with its image under the $U(1)$-action mapping $w_{2,k}$ to $-w_{2,k}$ if necessary, we have that $\langle v_{2,k},w_{2,k}\rangle\ge0$, so we obtain \eqref{e: angle at p}.  The monotonicity of comparison angles then gives us that
\begin{equation}\label{e: smaller than pi/2}
    \measuredangle \gamma_1\gamma_2=\lim_{k\to\infty}\widetilde{\measuredangle}x_{k,\infty}p_{\infty,1}y_{k,\infty}\leq\limsup_{k\to\infty}\measuredangle (\gamma'_{1,k}(0),\gamma'_{2,k}(0))\le\tfrac{\pi}{2}.
\end{equation}
Combining \eqref{e: bigger than pi/2} and \eqref{e: smaller than pi/2}, we find that $\measuredangle \gamma_1\gamma_2=\tfrac{\pi}{2}$, and so
by \eqref{e: angle at p} and the monotonicity of comparison angles again, we arrive at the fact that $$\widetilde{\measuredangle}\gamma_1(s)p_{\infty, 1}\gamma_2(s)\equiv\tfrac{\pi}{2}\qquad\textnormal{for all }s\ge0.$$

We claim that the sectional curvature of the 2-plane $\sigma$ spanned by $\gamma_1'(0)$ and $\gamma_2'(0)$ at $p_{\infty,1}$ vanishes. Suppose the contrary. Then the strict positivity of the sectional curvature of $\sigma$ implies that there is a small neighborhood in the two-dimensional submanifold $\Sigma:=\exp_{p_{\infty,1}}(\sigma)$ where the Gauss curvature is positive with respect to the restricted metric $g_{\Sigma}$. However, this would imply that for some small $s>0$, the $g_{\Sigma}$-minimizing geodesic in $\Sigma$ connecting $\gamma_{1}(s)$ and $\gamma_{2}(s)$ has length strictly smaller than $\sqrt{2}s$, which implies that $\widetilde\measuredangle \gamma_1(s)p_{\infty,1}\gamma_2(s)<\frac{\pi}{2}$. This is a contradiction.
The vanishing of the sectional curvature of $\sigma$ implies that \eqref{e: gamma1 gamma2 less than pi/2} vanishes, and
\begin{equation*}
    \langle v_{1,k},w_{1,k}\rangle=\langle v_{2,k},w_{2,k}\rangle=0.
\end{equation*}
Consequently, we have that $w_{1,k}=v_{2,k}=0$ or $w_{2,k}=v_{1,k}=0$. This implies that $\gamma_1'(0)\in N_1,\gamma_2'(0)\in N_2$ or $\gamma_2'(0)\in N_1,\gamma_1'(0)\in N_2$.
In each case, we obtain
\begin{equation*}
    \Rm(v_1,v_2,v_2,v_1)=0\qquad\textnormal{for any $v_1\in V_1$ and $v_2\in V_2$}.
\end{equation*}
Since $\Rm\ge0$ is a symmetric linear operator acting on $\bigwedge^2T_{p_{\infty,1}}M_{\infty,1}$, this implies that 
\begin{equation}\label{e: v_1 v_2 curvature vanish}
    \Rm(v_1\wedge v_2,\cdot)=0\qquad\textnormal{for any $v_1\in V_1$ and $v_2\in V_2$}.
\end{equation}

Now we show
that the steady Ricci soliton splits isometrically as $M_{\infty,1}=M_1\times M_2$ for some submanifolds $M_1,M_2,$ of nonzero dimension. Indeed, note that, as a smooth limit of manifolds diffeomorphic to $\mathbb{C}^n$, $M_{\infty,1}$ is diffeomorphic to $\mathbb{C}^n$ and by the strong maximum principle we may assume that $\Ric>0$ on $M_{\infty,1}$. The results of Bryant \cite{Bry08} and Chau-Tam \cite{CT05} subsequently imply that $M_{\infty,1}$ is in fact biholomorphic to $\mathbb{C}^n$. We may then argue as in \cite[Lemma 4.1]{ChenZhu2005} or \cite[Theorem 7.34]{HaRF} to deduce that if $M_{\infty,1}$ is irreducible, then $M_{\infty,1}$ has either positive curvature operator on real $(1,\,1)$-forms or is symmetric or is Einstein. The latter two cases are impossible as they would imply constant scalar curvature on $M_{\infty,1}$, contradicting the fact that $\Delta R+\langle\nabla f,\nabla R\rangle=-2|\Ric|^2$ and $\Ric>0$. $M_{\infty,1}$ must therefore have positive curvature operator on real $(1,\,1)$-forms and positive sectional curvature (see Section 2.3), which is again absurd due to the fact that the sectional curvature vanishes on a certain 2-plane. Thus, we conclude that $M_{\infty,1}$ is indeed reducible and that by the de Rham decomposition theorem \cite{KN69}, splits isometrically and holomorphically as a product of lower dimensional steady gradient K\"ahler-Ricci solitons.

We now show that the decomposition of the tangent space $T_{p_{\infty,1}}M_{\infty,1}=V_1\perp V_2$ induces a splitting of the manifold as $M_{\infty,1}=N_1\times N_2$. First, we claim that the sectional curvature is positive on each $N_i$ at $p_{\infty,1}$ for $i=1,2$. Otherwise, suppose a 2-plane of $V_1$ has zero sectional curvature. Then we would have that the sectional curvature is zero for all 2-planes of $V_1$. Indeed, the curvature of the induced K\"{a}hler metric on $N_1$ at the fixed point $p_{\infty,1}$ of $U(n-1)$ has the form
\[
-R_{\alpha\bar \beta \gamma\bar \delta}=\kappa(g_{\alpha\bar \beta}g_{ \gamma\bar \delta}+g_{\alpha\bar \delta}g_{\gamma\bar \beta})
\]
for some constant $\kappa\ge 0$. For $\kappa>0$, the ratio between the smallest sectional curvature and the scalar curvature of the induced K\"{a}hler metric on $N_1$ at $p_{\infty,1}$ is a strictly positive dimensional constant (see \eqref{uncur}), hence zero sectional curvature at a 2-plane implies zero curvature operator and thus zero sectional curvature for all 2-planes in $T_{p_{\infty,1}}N_1$. This implies that $\Ric(v_i,\cdot)=0$ for some nonzero $v_i\in V_1$, contradicting Remark \ref{1eigv}. 
Next, since $M_{\infty,1}=M_1\times M_2$, it follows that for any arbitrary nonzero vectors $v_1+v_2\in M_1$ and $w_1+w_2\in M_2$, where $v_1,w_1\in V_1$ and $v_2,w_2\in V_2$, we have that
\begin{equation}\label{e: fixed curvature vanish}
    \Rm(v_1+v_2,w_1+w_2,w_1+w_2,v_1+v_2)=0.
\end{equation}
By \eqref{e: v_1 v_2 curvature vanish} and the first Bianchi identity, we deduce that $\Rm(v_1,w_1,w_2,v_2)=0$, which, together with \eqref{e: v_1 v_2 curvature vanish} and \eqref{e: fixed curvature vanish}, implies that
\begin{equation*}
\Rm(v_1,w_1,w_1,v_1)=\Rm(v_2,w_2,w_2,v_2)=0.\end{equation*}
Since the sectional curvatures are positive 
on $V_1,V_2$, this implies that $v_1\wedge w_1=v_2\wedge w_2=0$.
If $v_1,v_2\neq0$, then this implies that $w_1= av_1$, $w_2= bv_2$, and $w_1+w_2=av_1+bv_2$ for some $a,b\in\R$.
Hence $T_{p_{\infty,1}}M_2\subset \textnormal{Span}(v_1,v_2)$.
Since the dimension of $M_2$ is at least $2$, we have that $T_{p_{\infty,1}}M_2= \textnormal{Span}(v_1,v_2)$, and in particular $v_1+v_2\in T_{p_{\infty,1}}M_2\cap T_{p_{\infty,1}}M_1$, which is a contradiction. So we have that $v_2=0$ or $v_1=0$, which implies that $T_{p_{\infty,1}}M_1\subset V_1$ or $T_{p_{\infty,1}}M_1\subset V_2$. Without loss of generality we may assume that $T_{p_{\infty,1}}M_1\subset V_1$. 
Similarly, we can deduce either $w_1$ or $w_2$ vanishes and thus $T_{p_{\infty,1}}M_2\subset V_2$ or $T_{p_{\infty,1}}M_2\subset V_1$. The latter case cannot happen and so we must have
$T_{p_{\infty,1}}M_2\subset V_2$.
Therefore, we obtain $T_{p_{\infty,1}}M_i=V_i$ for both $i=1,2$. 
Since $N_1,N_2,$ are the totally geodesic subspaces, it follows that $M_i=N_i$ for $i=1,2$, and $M_1$ (respectively $M_2$) is a $U(n-1)$-invariant (resp.~$U(1)$-invariant) steady gradient K\"ahler-Ricci soliton. In particular, they are scalings of Cao's steady gradient K\"ahler-Ricci solitons.

For a 2-plane of $T_{p_{i,\mu}}M_{i,\mu}$ spanned by a vector in the $U(1)$-fixed subspace and another vector in the $U(n-1)$-fixed subspace, we can identify it with the 2-plane $\sigma$ of $T_{p_{\infty,1}}M_{\infty,1}$. Then the sectional curvature 
$K_{i,\mu}:=K_{g_{i,\mu}}(\sigma)$ varies smoothly in $\mu$ for each fixed $i$. We have just demonstrated that $K_{\infty,1}=0$. It is also clear that $K_{\infty,0}>0$ in the non-flat $U(n)$-invariant steady Ricci soliton $(M_{\infty,0},g_{\infty,0},p_{\infty,0})$. 
Therefore, for any given $\alpha\in(0,K_{\infty,0})$, we have that $K_{i,0}<\alpha<K_{i,1}$ for all sufficiently large $i$, and hence by the intermediate value theorem
there exists $\mu_i\in(0,1)$ such that $K_{i,\,\mu_i}=\alpha$. A subsequence of the expanding gradient Ricci solitons $(M_{i,\,\mu_i},g_{i,\,\mu_i},p_{i,\,\mu_i})$ therefore converges to a steady K\"ahler-Ricci soliton $(M_{\infty},g_{\infty},p_{\infty})$ with $K_{g_{\infty}}(\sigma)=\alpha$ and $R(p_{\infty})=1$. In particular,  the strict positivity of the sectional curvature implies that $(M_{\infty},g_{\infty})$ doesn't split. This proves Theorem \ref{t: existence of new}.

\end{proof}

\appendix\label{appaa}

\section{Curvature of doubly warped product K\"ahler metrics}

In this appendix, we present the curvature computations for the Riemannian curvature tensor of a doubly warped product K\"ahler metric $\hat{g}$ on $\widehat{M}=M\times I$ of the form
\[ \hat{g}:=ds^{2}+a^2(s)\eta\otimes \eta+b^2(s)g^{T},
\]
where $M$ is a Sasaki manifold of real dimension $(2n-1)$ with $n\geq2$ with contact one-form $\eta$ and transverse metric $g^{T}$ as described in Section \ref{T3}, $I\subseteq\mathbb{R}$ is an interval, and $a,\,b:I\to[0,\,\infty)$ are smooth functions. We assume that $a(s)=b(s)b'(s)$ so that $(\widehat{M},\,\hat{g})$
does indeed defines a K\"ahler manifold (cf.~Lemma \ref{kahlerr}) and we write $\widehat{J}$ for the complex structure of $\widehat{M}$ (see Section \ref{T3} for more details). We compute using Cartan's structure equations for the connection one-forms and corresponding curvature two-forms.

\subsection{Connection one-forms}
We denote the Levi-Civita connection of $\hat{g}$ by $\widehat{\nabla}$. We compute in this subsection the connection one-forms of $\widehat{\nabla}$.
Recall that a differential form $\sigma$ on a Sasaki manifold is \emph{basic} if $\xi\lrcorner\sigma=0$ and $\mathcal{L}_{\xi}\sigma=0$, where $\xi$ is the Reeb vector field.

To this end, let $\theta_{1},\ldots,\theta_{2n-2}$ be a local basic orthonormal coframe of $g^{T}$ satisfying $\theta_{i}\circ \widehat{J}=-\theta_{i+1}$ for $i$ odd
so that $d\eta=2\omega^{T}=2\sum_{j=1}^{n-1}{\theta}_{2j-1}\wedge {\theta}_{2j},$ and let
$(\omega_{ij})_{1\,\leq\, i,\,j\,\leq\,2n-2}$ denote the matrix of connection one-forms of $g^{T}$. Then
$(\omega_{ij})$ solves the Cartan structure equations
\begin{equation*}
\left\{ \begin{array}{ll}
d{\theta}_{i}=\sum_{j\,=\,1}^{2n-2}{\omega}_{ji}\wedge\theta_{j},\\
{\omega}_{ij}+{\omega}_{ji}=0.
\end{array} \right.
\end{equation*}
Next set
$$\hat{\theta}_{i}:=b(s)\theta_{i}\qquad\textrm{for $i=1,\ldots,2n-2$},\qquad\hat{\theta}_{2n-1}:=ds,
\qquad\textrm{and}\qquad
\hat{\theta}_{2n}:=a(s)\eta.$$
Then the matrix of connection one-forms $(\widehat{\omega}_{ij})_{1\,\leq\, i,\,j\,\leq\, 2n}$ of $\hat{g}$ with respect to this coframe is given by
$$\widehat{\omega}_{ij}=\omega_{ij}-\delta_{j,\,i+1}\frac{a(s)}{b(s)^{2}}\hat{\theta}_{2n}\qquad\textrm{$1\leq i\leq 2n-2$ odd,\qquad$1\leq j\leq 2n-2$},$$
$$\widehat{\omega}_{ij}=\omega_{ij}+\delta_{j,\,i-1}\frac{a(s)}{b(s)^{2}}\hat{\theta}_{2n}\qquad\textrm{$1\leq i\leq 2n-2$ even,\qquad$1\leq j\leq 2n-2$},$$

\begin{equation*}
\begin{split}
\widehat{\omega}_{2n-1,\,i}&=\left(-\frac{b'(s)}{b(s)}\right)\hat{\theta}_{i}=-b'(s)\theta_{i},\qquad 1\leq i\leq 2n-2,\\
\widehat{\omega}_{2n-1,\,2n}&=\left(-\frac{a'(s)}{a(s)}\right)\hat{\theta}_{2n}=-a'(s)\eta,\\
\widehat{\omega}_{2n,\,i}&=\left\{ \begin{array}{ll}
\frac{a(s)}{b(s)^{2}}\hat{\theta}_{i+1}=\frac{a(s)}{b(s)}\theta_{i+1},\qquad\textrm{$1\leq i\leq 2n-2$ odd},\\
-\frac{a(s)}{b(s)^{2}}\hat{\theta}_{i-1}=-\frac{a(s)}{b(s)}\theta_{i-1},\qquad\textrm{$1\leq i\leq 2n-2$ even}.\\
\end{array} \right.\\
\end{split}
\end{equation*}
This translates to the following for $i=1,2,\dots, n-1$:
\begin{equation*}
\begin{split}
\widehat{\nabla}\hat{\theta}_{2i-1}=\sum_{k\,=\,1}^{2n}\widehat{\omega}_{k\, 2i-1}\otimes\hat{\theta}_{k}
&=\sum_{k\,=\,1}^{2n-2}\left(\omega_{k\, 2i-1}+\delta_{k,2i}\left(\frac{a(s)}{b(s)^{2}}\right)\hat{\theta}_{2n}\right)\otimes\hat{\theta}_{k}
-\frac{b'(s)}{b(s)}\hat{\theta}_{2i-1}\otimes\hat{\theta}_{2n-1}\\
&\qquad+\frac{a(s)}{b(s)^{2}}\hat{\theta}_{2i}\otimes\hat{\theta}_{2n},\\
\widehat{\nabla}\hat{\theta}_{2i}=\sum_{k\,=\,1}^{2n}\widehat{\omega}_{k\, 2i}\otimes\hat{\theta}_{k}
&=\sum_{k\,=\,1}^{2n-2}\left(\omega_{k\, 2i}-\delta_{k,2i-1}\left(\frac{a(s)}{b(s)^{2}}\right)\hat{\theta}_{2n}\right)\otimes\hat{\theta}_{k}
-\frac{b'(s)}{b(s)}\hat{\theta}_{2i}\otimes\hat{\theta}_{2n-1}\\
&\qquad-\frac{a(s)}{b(s)^{2}}\hat{\theta}_{2i-1}\otimes\hat{\theta}_{2n},\\
\widehat{\nabla}\hat{\theta}_{2n-1}=\sum_{k\,=\,1}^{2n}\widehat{\omega}_{k,\,2n-1}\otimes\hat{\theta}_{k}
&=\sum_{k\,=\,1}^{2n-2}\frac{b'(s)}{b(s)}\hat{\theta}_{k}\otimes\hat{\theta}_{k}+\frac{a'(s)}{a(s)}\hat{\theta}_{2n}\otimes\hat{\theta}_{2n},\\
\widehat{\nabla}\hat{\theta}_{2n}=\sum_{k\,=\,1}^{2n}\widehat{\omega}_{k,\,2n}\otimes\hat{\theta}_{k}&=
\sum_{k\,=\,1}^{n-1}-\frac{a(s)}{b(s)^{2}}\hat{\theta}_{2k}\otimes\hat{\theta}_{2k-1}+\sum_{k\,=\,1}^{n-1}\frac{a(s)}{b(s)^{2}}\hat{\theta}_{2k-1}\otimes\hat{\theta}_{2k}\\
&\qquad-\frac{a'(s)}{a(s)}\hat{\theta}_{2n}\otimes\hat{\theta}_{2n-1}.\\
\end{split}
\end{equation*}

\subsection{Curvature two-forms}

Recall that $a(s)=b(s)b'(s)$ so that
$(\widehat{M},\,\hat{g},\,\widehat{J})$ is K\"ahler.
We define the curvature tensor $\hat{R}(\cdot\,,\,\cdot)(\cdot)$
of $\hat{g}$ by
$$\hat{R}(X,\,Y)Z:=\widehat{\nabla}_{X}\widehat{\nabla}_{Y}Z-\widehat{\nabla}_{Y}\widehat{\nabla}_{X}Z-\widehat{\nabla}_{[X,\,Y]}Z\qquad\textrm{for $X,\,Y,\,Z\in\Gamma(T\widehat{M})$.}$$
With this, we realize $\hat{R}$ as a $(0,\,4)$-tensor via
$$\hat{R}(X,\,Y,\,Z,\,W)=\langle\hat{R}(X,\,Y)W,\,Z\rangle_{\hat{g}}\qquad
\textrm{for $X,\,Y,\,Z,\,W\in\Gamma(T\widehat{M})$},$$
and write
\begin{equation*}
\begin{split}
\hat{R}&=\hat{R}(\hat{e}_{i},\,\hat{e}_{j},\,\hat{e}_{k},\,\hat{e}_{l})\,\hat{\theta}_{i}\otimes\hat{\theta}_{j}
\otimes\hat{\theta}_{k}\otimes\hat{\theta}_{l}\\
&=\langle\hat{R}(\hat{e}_{i},\,\hat{e}_{j})\hat{e}_{l},\,\hat{e}_{k}\rangle_{\hat{g}}\,
\hat{\theta}_{i}\otimes\hat{\theta}_{j}
\otimes\hat{\theta}_{k}\otimes\hat{\theta}_{l}\\
&=\hat{R}_{ijkl}\,
\hat{\theta}_{i}\otimes\hat{\theta}_{j}
\otimes\hat{\theta}_{k}\otimes\hat{\theta}_{l},\\
\end{split}
\end{equation*}
where
\begin{equation}\label{conven}
\hat{R}_{ijkl}:=\langle\hat{R}(\hat{e}_{i},\,\hat{e}_{j})\hat{e}_{l},\,\hat{e}_{k}\rangle_{\hat{g}}.
\end{equation}
We define the curvature two-forms $\widehat{\Omega}_{ij},\,1\leq i,\,j\leq 2n,$ of
$\hat{R}$ by $$\hat{R}=\hat{\theta}_{i}\otimes\hat{\theta}_{j}\otimes\widehat{\Omega}_{ij}$$
so that
$$\widehat{\Omega}_{ij}=\hat{R}_{ijkl}
\hat{\theta}_{k}\otimes\hat{\theta}_{l}=\frac{1}{2}\hat{R}_{ijkl}\hat{\theta}_{k}\wedge\hat{\theta}_{l}.$$
Then
\begin{equation*}
\hat{R}_{ijkl}=\langle\hat{R}(\hat{e}_{i},\,\hat{e}_{j})\hat{e}_{l},\,\hat{e}_{k}\rangle_{\hat{g}}=\widehat{\Omega}_{ij}(\hat{e}_{k},\,\hat{e}_{l})
\end{equation*}
and $\widehat{\Omega}_{ij}$ satisfies the Cartan structure equations
$$\widehat{\Omega}_{ij}=d\widehat{\omega}_{ij}+\widehat{\omega}_{ik}\wedge\widehat{\omega}_{kj}.$$
By the symmetries of the curvature tensor, we also see that the two-forms $\widehat{\omega}_{ij}$ are real $(1,\,1)$-forms. In this subsection, we compute the curvature two-forms $(\widehat{\Omega})_{1\leq i,\,j\leq n}$ of $\hat{g}$.

Let $\Omega_{ij},\,1\leq i,\,j\leq 2n-2,$ denote the curvature two-forms of $g^{T}$. These satisfy
$$\Omega_{ij}=d\omega_{ij}+\sum_{k=1}^{2n-2}\omega_{ik}\wedge\omega_{kj}\qquad\textrm{
for $1\leq i,\,j\leq 2n-2$.}$$
We have the freedom to choose the local basic orthonormal coframe $\{\theta_{i}\}_{i=1}^{2n-2}$ of $g^T$ to be parallel at a point so that $(\omega_{ij})_{1\leq i,\,j\leq2n-2}=0$ at this point. This we do. With this simplification, a computation shows that at this point, the two-forms $\Omega_{ij},\,1\leq i,\,j\leq 2n-2,$
are given by
\begin{equation*}
\begin{split}
&\widehat{\Omega}_{ij}=\Omega_{ij}-\frac{b'(s)^{2}}{b(s)^{2}}\left(\hat{\theta}_{i}\wedge\hat{\theta}_{j}+\hat{\theta}_{i+1}\wedge\hat{\theta}_{j+1}\right),\qquad\textrm{$1\leq i,\,j\leq 2n-2$ odd},\\
&\widehat{\Omega}_{ij}=\Omega_{ij}-\frac{b'(s)^{2}}{b(s)^{2}}\left(\hat\theta_{i}\wedge\hat{\theta}_{j}+\hat{\theta}_{i-1}\wedge\hat{\theta}_{j-1}\right),\qquad\textrm{$1\leq i,\,j\leq 2n-2$ even},\\
&\widehat{\Omega}_{ij}=\Omega_{ij}-\frac{b'(s)^{2}}{b(s)^{2}}\left(\hat{\theta}_{i}\wedge\hat{\theta}_{j}-\hat{\theta}_{i+1}\wedge\hat{\theta}_{j-1}\right)-2\delta_{j,\,i+1}\biggl[\left(\frac{a'(s)}{b(s)^{2}}-\frac{a(s)b'(s)}{b(s)^{3}}\right)\hat{\theta}_{2n-1}\wedge\hat{\theta}_{2n}\\
&\qquad+\left(\frac{a(s)^{2}}{b(s)^{2}}\right)\omega^{T}\biggr],\qquad\textrm{$1\leq i\leq 2n-2$ odd, $1\leq j\leq 2n-2$ even},\\
&\widehat{\Omega}_{i,\,\,2n-1}=\left(\frac{b''(s)}{b(s)}\right)\left(\hat{\theta}_{2n-1}\wedge\hat{\theta}_{i}-\hat{\theta}_{i+1}\wedge\hat{\theta}_{2n}\right),\qquad i=1,\ldots,2n-2,\textrm{$i$ odd},\\
&\widehat{\Omega}_{i,\,\,2n-1}=\left(\frac{b''(s)}{b(s)}\right)\left(\hat{\theta}_{2n-1}\wedge\hat{\theta}_{i}+\hat{\theta}_{i-1}\wedge\hat{\theta}_{2n}\right),\qquad i=1,\ldots,2n-2,\textrm{$i$ even},\\
&\widehat{\Omega}_{2n,\,i}=\left(\frac{a'(s)}{b(s)^{2}}-\frac{a(s)b'(s)}{b(s)^{3}}\right)\left(
\hat{\theta}_{2n-1}\wedge\hat{\theta}_{i+1}+\hat{\theta}_{i}\wedge\hat{\theta}_{2n}\right),\qquad i=1,\ldots,2n-2,\textrm{$i$ odd},\\
&\widehat{\Omega}_{2n,\,i}=-\left(\frac{a'(s)}{b(s)^{2}}-\frac{a(s)b'(s)}{b(s)^{3}}\right)\left(\hat{\theta}_{2n-1}\wedge\hat{\theta}_{i-1}-\hat{\theta}_{i}\wedge\hat{\theta}_{2n}\right),\qquad i=1,\ldots,2n-2,\textrm{$i$ even},\\
&\widehat{\Omega}_{2n,\,2n-1}=
2\left(a'(s)-\frac{a(s)b'(s)}{b(s)}\right)\omega^{T}+\frac{a''(s)}{a(s)}\hat{\theta}_{2n-1}\wedge\hat{\theta}_{2n}.\\
\end{split}
\end{equation*}

Henceforth working at the point where the basis $\{\theta_{i}\}_{i=1}^{2n-2}$ of $g^{T}$ is parallel, we have that $\widehat{J}\hat{e}_{2k-1}=\hat{e}_{2k}$ because $\hat{\theta}_{k}$ is dual to $\hat{e}_{k}$. We now define for $k=1,\ldots,n$,
$$\hat{\Theta}_{k}:=\hat{\theta}_{2k-1}+i\,\hat{\theta}_{2k},\qquad
\hat{\Theta}_{\bar{k}}:=\hat{\theta}_{2k-1}-i\,\hat{\theta}_{2k},$$
$$ \hat{E}_{k}:=\frac{1}{2}\left(\hat{e}_{2k-1}-i\,\hat{e}_{2k}\right),\qquad
 \hat{E}_{\bar{k}}:=\frac{1}{2}\left(\hat{e}_{2k-1}+i\,\hat{e}_{2k}\right).$$
Similarly, for $k=1,\ldots,n-1$, we define the corresponding frame for $g^T$:
$$E_{k}:=\frac{1}{2}\left(e_{2k-1}-i\,e_{2k}\right),\qquad
 E_{\bar{k}}:=\frac{1}{2}\left(e_{2k-1}+i\,e_{2k}\right).$$
In particular, observe that
$\hat{E}_{n}=\frac{1}{2}\left(\partial_r-i\left(\frac{\xi}{a(s)}\right)\right)$
because by definition, $J_0\partial_r=\frac{\xi}{a(s)}$.
We extend the $(0,\,4)$-tensor $\hat{R}$ by complex multi-linearity and write the corresponding components as
$$\hat{R}_{\alpha\bar{\beta}\gamma\bar{\delta}}:=\hat{R}(\hat{E}_{\alpha},\,\hat{E}_{\bar{\beta}},\,\hat{E}_{\gamma},\,
\hat{E}_{\bar{\delta}}).$$
Then
$$\hat{R}=\left(\hat\Theta_{\alpha}\wedge\hat\Theta_{\bar{\beta}}\right)\otimes\widehat{\Omega}_{\alpha\bar{\beta}},$$
where the complex-valued $(1,\,1)$-forms
$\widehat{\Omega}_{\alpha\bar{\beta}},\,1\leq\alpha,\,\beta\leq n,$
satisfy
$$\widehat{\Omega}_{\alpha\bar{\beta}}(\hat{E}_{\gamma},\,\hat{E}_{\bar{\delta}})=\hat{R}_{\alpha\bar{\beta}\gamma\bar{\delta}},$$
and are related to the real $(1,\,1)$-forms $\widehat{\Omega}_{ij},\,1\leq i,\,j\leq 2n,$
by
$$\widehat{\Omega}_{\alpha\bar{\beta}}=\frac{1}{2}\left(\widehat{\Omega}_{2\alpha-1,\,2\beta-1}+i\,\widehat{\Omega}_{2\alpha-1,\,2\beta}\right).$$

For $1\leq\alpha,\,\beta\leq n-1$, 
we have that
$$\widehat{\Omega}_{2\alpha-1, 2\beta-1}
=\Omega_{2\alpha-1, 2\beta-1}-\frac{1}{2}\left(\frac{b'(s)}{b(s)}\right)^2\left[\hat{\Theta}_{\alpha}\wedge \hat{\Theta}_{\bar{\beta}}+\hat{\Theta}_{\bar{\alpha}}\wedge \hat{\Theta}_{\beta}\right]$$
and
\begin{equation*}
\begin{split}
\widehat{\Omega}_{2\alpha-1, 2\beta}&=\Omega_{2\alpha-1, 2\beta}-\frac{i}{2}\left(\frac{b'(s)}{b(s)}\right)^2\left[\hat{\Theta}_{\alpha}\wedge \hat{\Theta}_{\bar{\beta}}-\hat{\Theta}_{\bar{\alpha}}\wedge \hat{\Theta}_{\beta}\right]\\
&\qquad-\delta_{\alpha\beta}\left[i\left(\frac{b''(s)}{b(s)}\right)\hat{\Theta}_{n}\wedge\hat{\Theta}_{\bar{n}}
+i\left(\frac{a(s)^{2}}{b(s)^{4}}\right)\sum_{l=1}^{n-1}\hat{\Theta}_{l}\wedge\hat{\Theta}_{\bar{l}}\right].\\
\end{split}
\end{equation*}
It therefore follows that
$$\widehat{\Omega}_{\alpha\bar{\beta}}=\frac{\Omega_{\alpha\bar{\beta}}}{b(s)^{2}}-\frac{1}{2}\left(\frac{b'(s)}{b(s)}\right)^2\hat{\Theta}_{\bar{\alpha}}\wedge \hat{\Theta}_{\beta}+\frac{\delta_{\alpha\beta}}{2}\left[\left(\frac{b''(s)}{b(s)}\right)\hat{\Theta}_{n}\wedge\hat{\Theta}_{\bar{n}}
+\left(\frac{a(s)^{2}}{b(s)^{4}}\right)\sum_{l=1}^{n-1}\hat{\Theta}_{l}\wedge\hat{\Theta}_{\bar{l}}\right]$$
for $1\leq\alpha,\,\beta\leq n-1$.

Next, for $1\leq\alpha\leq n-1$, $\beta=n$, 
we have
\begin{equation*}
\begin{split}
\widehat{\Omega}_{2\alpha-1,2n-1}
&=\frac{b''(s)}{2b(s)}\left(\hat{\Theta}_{n}\wedge\hat{\Theta}_{\bar{\alpha}}+\hat{\Theta}_{\bar{n}}\wedge\hat{\Theta}_{\alpha}\right)\\
\end{split}
\end{equation*}
and
\begin{equation*}
\begin{split}
\widehat{\Omega}_{2\alpha-1,2n}
&=-\frac{i}{2}\left(\frac{b''(s)}{b(s)}\right)\left(\hat{\Theta}_{n}\wedge\hat{\Theta}_{\bar{\alpha}}-\hat{\Theta}_{\bar{n}}\wedge\hat{\Theta}_{\alpha}\right).\\
\end{split}
\end{equation*}
Hence in this case we see that
\begin{equation*}
\begin{split}
\widehat{\Omega}_{\alpha\bar{n}}&=\frac{1}{2}\left(\widehat{\Omega}_{2\alpha-1,\,2n-1}+i\,\widehat{\Omega}_{2\alpha-1,\,2n}\right)=\frac{b''(s)}{2b(s)}\hat{\Theta}_{n}\wedge\hat{\Theta}_{\bar{\alpha}}.\\
\end{split}
\end{equation*}

Finally, for $\alpha=n$, $\beta=n$, we have that
\begin{equation*}
\begin{split}
\widehat{\Omega}_{n\bar{n}}
&=\frac{1}{2}\left(\widehat{\Omega}_{2n-1,\,2n-1}+i\,\widehat{\Omega}_{2n-1,\,2n}\right)\\
&=\frac{i}{2}\widehat{\Omega}_{2n-1,\,2n}\\
&=\frac{1}{2}\left[\left(\frac{b''(s)}{b(s)}\right)\sum_{l=1}^{n-1}\hat{\Theta}_{l}\wedge\hat{\Theta}_{\bar{l}}+\frac{a''(s)}{2 a(s)}\hat{\Theta}_{n}\wedge\hat{\Theta}_{\bar{n}}\right].\\
\end{split}
\end{equation*}

In summary, $\widehat{\Omega}_{\alpha\bar{\beta}}$ is equal to
\begin{displaymath}
\left\{ \begin{array}{ll}
\frac{\Omega_{\alpha\bar{\beta}}}{b(s)^{2}}-\frac{1}{2}\left(\frac{b'(s)}{b(s)}\right)^2\hat{\Theta}_{\bar{\alpha}}\wedge \hat{\Theta}_{\beta}+\frac{\delta_{\alpha\beta}}{2}\left[\left(\frac{b''(s)}{b(s)}\right)\hat{\Theta}_{n}\wedge\hat{\Theta}_{\bar{n}}
+\left(\frac{a(s)^{2}}{b(s)^{4}}\right)\sum_{l=1}^{n-1}\hat{\Theta}_{l}\wedge\hat{\Theta}_{\bar{l}}\right] & \textrm{$1\leq\alpha,\,\beta\leq n-1$},\\
\left(\frac{b''(s)}{2b(s)}\right)\hat{\Theta}_{n}\wedge\hat{\Theta}_{\bar{\alpha}} & \textrm{$1\leq\alpha\leq n-1$, $\beta=n$},\\
\frac{1}{2}\left[\left(\frac{b''(s)}{b(s)}\right)\sum_{l=1}^{n-1}\hat{\Theta}_{l}\wedge\hat{\Theta}_{\bar{l}}+\frac{a''(s)}{2 a(s)}\hat{\Theta}_{n}\wedge\hat{\Theta}_{\bar{n}}\right] & \textrm{$\alpha=\beta=n$}.
\end{array} \right.
\end{displaymath}
In particular, this gives us that
\begin{equation}\label{curform}
\hat R_{\alpha\bar{\beta}\gamma\bar{\delta}} = \left\{ \begin{array}{ll}
\frac{\tilde{R}_{\alpha\bar{\beta}\gamma\bar{\delta}}}{b(s)^{2}}+\frac{1}{2}\left(\frac{b'(s)}{b(s)}\right)^2\delta_{\alpha\delta}\delta_{\beta\gamma}
+\left(\frac{a(s)^{2}}{2b(s)^{4}}\right)
\delta_{\alpha\beta}\delta_{\gamma\delta} & \textrm{$1\leq\alpha,\,\beta,\,\gamma,\,\delta\leq n-1$},\\
\left(\frac{b''(s)}{2b(s)}\right)\delta_{\gamma n}\delta_{\delta\alpha} & \textrm{$1\leq\alpha\leq n-1$, $\beta=n$},\\
\frac{1}{2}\left[\left(\frac{b''(s)}{b(s)}\right)\left(\delta_{\gamma\delta}-\delta_{\gamma n}\delta_{\delta n}\right)+
\left(\frac{a''(s)}{2 a(s)}\right)\delta_{\gamma n}\delta_{\delta n}\right] & \textrm{$\alpha=n$, $\beta=n$},
\end{array} \right.
\end{equation}
where $\tilde R$ is the curvature tensor of $g^T$ and $\tilde R_{\alpha\bar{\beta}\gamma\bar{\delta}}=\tilde R\left(E_\alpha, E_{\bar\beta}, E_{\gamma}, E_{\bar\delta}\right)$.

\subsection{Curvature of a K\"ahler cone}
Recalling the construction of the $n$-dimensional
K\"ahler cone from Example \ref{cone}, we have that $a(s)=b(s)=s$, and so from the above we determine that for this example,
\begin{displaymath}
\widehat{\Omega}_{\alpha\bar{\beta}} = \left\{ \begin{array}{ll}
\frac{1}{s^{2}}\left(\Omega_{\alpha\bar{\beta}}-\frac{1}{2}\hat{\Theta}_{\bar{\alpha}}\wedge \hat{\Theta}_{\beta}+\frac{\delta_{\alpha\beta}}{2}\sum_{l=1}^{n-1}\hat{\Theta}_{l}\wedge\hat{\Theta}_{\bar{l}}\right) & \textrm{if $1\leq\alpha,\,\beta\leq n-1$},\\
0 & \textrm{if $1\leq\alpha\leq n$, $\beta=n$}.\\
\end{array} \right.
\end{displaymath}

It is well-known that a Riemannian cone has $\Rm\ge0$ if and only if the link has $\Rm\ge1$. The K\"ahler version of this fact is the following.
\begin{theorem}[{\cite{book:Boyer}}]\label{cone vs base}
A K\"ahler cone has strictly positive curvature operator on real $(1,\,1)$-forms in the transverse directions if and only if the link is a Sasaki manifold whose base satisfies \eqref{e: unusual} with $\lambda=1$.
\end{theorem}

\begin{proof}
It suffices to show that  for all nonzero real $(1,\,1)$-forms involving only transverse components, namely those of the form $\sum_{\alpha,\,\beta=1}^{n-1}i\,u_{\alpha\bar\beta}\hat\Theta_{\alpha}\wedge \hat\Theta_{\bar\beta}$, we have that
\[
0<-\hat R_{\alpha\bar\beta\gamma\bar\eta}u^{\alpha\bar\beta}u^{\gamma\bar\eta}=\frac{\left(-\tilde R_{\alpha\bar\beta\gamma\bar\eta}-\tfrac{1}{2}(\delta_{\alpha\beta}\delta_{\gamma\eta}+\delta_{\alpha\eta}\delta_{\beta\gamma})\right)}{s^2}u^{\alpha\bar\beta}u^{\gamma\bar\eta},
\]
where $u^{\alpha\bar \beta}=\hat g^{\alpha\bar \gamma}\hat g^{\eta\bar\beta}u_{\eta\bar\gamma}$. Since $g^T\left(E_{\alpha},\,E_{\bar\beta}\right)=g^T_{\alpha\bar\beta}=\frac{\delta_{\alpha\beta}}{2}$, this happens precisely when the curvature tensor of the transverse metric satisfies
\begin{eqnarray*}
-\tilde R_{\alpha\bar\beta\gamma\bar\eta}u^{\alpha\bar\beta}u^{\gamma\bar\eta}&>&\frac{1}{2}\left(\delta_{\alpha\beta}\delta_{\gamma\eta}+\delta_{\alpha\eta}\delta_{\beta\gamma}\right)u^{\alpha\bar\beta}u^{\gamma\bar\eta}\\
&=& 2\left(g^T_{\alpha\bar\beta}g^T_{\gamma\bar\eta}+g^T_{\alpha\bar\eta}g^T_{\gamma\bar\beta}\right)u^{\alpha\bar\beta}u^{\gamma\bar\eta}.
\end{eqnarray*}
Thus, the transverse metric has curvature operator strictly greater than $2$ on real $(1,1)$-forms. This proves Theorem \ref{cone vs base}.
\end{proof}

In the next lemma, we consider a slightly more general situation than that of the K\"ahler cone.
\begin{lem}\label{warp vs curv} Let $\lambda\in \mathbb{R}$ and let $\hat g$ be a doubly warped product K\"ahler metric on $(0,\,L)\times\nolinebreak S^{2n-1}$, $L>0$, of the form
$$\hat g=ds^{2}+a^2(s)\eta^{2}+b^2(s)g^{T},$$
where $a,\,b:(0,\,L)\to(0,\,\infty)$ are smooth functions
and $(S^{2n-1},\,\eta,\,g^{T})$ is the round Sasaki structure on $S^{2n-1}$ as described in Example \ref{flat2} with $g^T=\frac{1}{2}g_{FS}$, where $g_{FS}$ is the Fubini-Study metric on $\mathbb{P}^{n-1}$ normalized so that $\Ric(g_{FS})=n g_{FS}$.

Let $p=(s,\,w_0)\in (0,\,L)\times S^{2n-1}$. 
Then $\hat g$ has curvature operator strictly greater than $2\lambda$ on $(1,1)$-forms at $p$ if and only if for any nonzero hermitian matrix $u^{\alpha\bar\beta}$,
\begin{eqnarray*}
    0&<&\frac{2(1-\lambda (b(s))^2-(b'(s))^2)}{(b(s))^2}\left(\frac14\left(\sum_{\alpha=1}^{n-1}u^{\alpha\bar\alpha}\right)^2+\frac14\sum_{\alpha, \beta=1}^{n-1}\left|u^{\alpha\bar\beta}\right|^2\right)+\left(-\frac{a''(s)}{4a(s)}-\lambda\right)\left|u^{n\overline{n}}\right|^2\\
    &&\quad +\left(-\frac{b''(s)}{b(s)}-\lambda\right)\sum_{\alpha=1}^{n-1}\left|u^{\alpha\overline{n}}\right|^2+\left(-\frac{b''(s)}{b(s)}-\lambda\right)u^{n\overline{n}}\sum_{\alpha=1}^{n-1}u^{\alpha\bar\alpha}.
\end{eqnarray*}
\end{lem}

\begin{proof}
From the formulas in \eqref{curform}, it suffices to express the difference $$-\hat R_{\alpha\bar\beta\gamma\bar\eta}-2\lambda \left(\hat g_{\alpha\bar\beta}\hat g_{\gamma\bar\eta}+\hat g_{\alpha\bar\eta}\hat g_{\gamma\bar\beta}\right)$$ in terms of the warping functions $a$ and $b$. To this end, let $p=(s,\,w_0)\in (0,L)\times S^{2n-1}$ 
 and for any nonzero real $(1,\,1)$-form $\sum_{\alpha,\beta=1}^{n}i\,u_{\alpha\bar\beta}\hat\Theta_{\alpha}\wedge \hat\Theta_{\bar\beta}$ at $p$, let $u^{\alpha\bar\beta}=\hat g^{\alpha\bar\varepsilon}\hat g^{\eta\bar\beta}u_{\eta\bar \varepsilon}=4u_{\beta\bar \alpha}$. Then
\begin{equation}\label{curf3}
\begin{split}
2\lambda(\hat g_{\alpha\bar\beta}\hat g_{\gamma\bar\eta}+\hat g_{\alpha\bar\eta}\hat g_{\gamma\bar\beta})&u^{\alpha\bar\beta}u^{\gamma\bar\eta}
=\frac{\lambda}{2}(\delta_{\alpha\beta}\delta_{\gamma\eta}+\delta_{\alpha\eta}\delta_{\beta\gamma})u^{\alpha\bar\beta}u^{\gamma\bar\eta}\\
&=\frac{\lambda}{2}\left[\left(u^{n\overline{n}}+\sum_{\alpha=1}^{n-1}u^{\alpha\bar\alpha}\right)^2+\sum_{\alpha,\beta=1}^{n}\left|u^{\alpha\bar\beta}\right|^2\right]\\
&=\frac{\lambda}{2}\Bigg[2\left|u^{n\overline{n}}\right|^2+2u^{n\overline{n}}\sum_{\alpha=1}^{n-1}u^{\alpha\bar\alpha}+\left(\sum_{\alpha=1}^{n-1}u^{\alpha\bar\alpha}\right)^2+\sum_{\alpha,\beta=1}^{n-1}\left|u^{\alpha\bar\beta}\right|^2+2\sum_{\alpha=1}^{n-1}\left|u^{\alpha\overline{n}}\right|^2\Bigg].
\end{split}
\end{equation}
Here we have used the fact that $\overline{u^{\alpha\overline{n}}}=u^{n\overline{\alpha}}$ in the last equality. By the symmetries of the curvature tensor, we moreover have that
\begin{eqnarray*}
-\sum_{\alpha, \beta, \gamma,\eta=1}^{n}\hat  R_{\alpha\bar\beta\gamma\bar\eta}u^{\alpha\bar\beta}u^{\gamma\bar\eta}&=&\sum_{\alpha, \beta, \gamma,\eta=1}^{n-1}-\hat R_{\alpha\bar\beta\gamma\bar\eta}u^{\alpha\bar\beta}u^{\gamma\bar\eta}-\hat R_{n\overline{n}n\overline{n}}u^{n\overline{n}}u^{n\overline{n}}\\
&& -2\sum_{\alpha=1}^{n-1}u^{\alpha\overline{n}}u^{n\overline{n}}\hat R_{\alpha\overline{n}n\overline{n}}-2\sum_{\alpha=1}^{n-1}u^{n\overline{\alpha}}u^{n\overline{n}}\hat R_{n\overline{\alpha}n\overline{n}}\\
&&-2\sum_{\alpha,\beta=1}^{n-1}u^{\alpha\overline{\beta}}u^{n\overline{n}}\hat R_{\alpha\overline{\beta}n\overline{n}}-2\sum_{\alpha,\beta=1}^{n-1}u^{\alpha\overline{n}}u^{n\overline{\beta}}\hat R_{\alpha\overline{n}n\overline{\beta}}\\
&&-\sum_{\alpha,\beta=1}^{n-1}u^{n\overline{\alpha}}u^{n\overline{\beta}}\hat R_{n\overline{\alpha}n\overline{\beta}}-\sum_{\alpha,\beta=1}^{n-1}u^{\alpha\overline{n}}u^{\beta\overline{n}}\hat R_{\alpha\overline{n}\beta\overline{n}}\\
&&-2\sum_{\alpha,\beta,\gamma=1}^{n-1}u^{\alpha\overline{\beta}}u^{n\overline{\gamma}}\hat R_{\alpha\overline{\beta}n\overline{\gamma}}-2\sum_{\alpha,\beta,\gamma=1}^{n-1}u^{\alpha\overline{\beta}}u^{\gamma\overline{n}}\hat R_{\alpha\overline{\beta}\gamma\overline{n}}.
\end{eqnarray*}
This expression can be simplified using the fact that
$$
\hat R_{\alpha\overline{\beta}n\overline{\gamma}}=\hat R_{\alpha\overline{\beta}\gamma\overline{n}}=\hat R_{\alpha\overline{n}n\overline{n}}=\hat R_{n\overline{\alpha}n\overline{n}}=\hat R_{\alpha\overline{n}\beta\overline{n}}=\hat R_{n\overline{\alpha}n\overline{\beta}}=0.
$$
Since the curvature tensor of $\frac{1}{2}g_{FS}$ is given by $-\tfrac{1}{2}(\delta_{\alpha\beta}\delta_{\gamma\eta}+\delta_{\alpha\eta}\delta_{\beta\gamma})$, we see that
\begin{eqnarray*}
\sum_{\alpha, \beta, \gamma,\eta=1}^{n-1}-\hat R_{\alpha\bar\beta\gamma\bar\eta}u^{\alpha\bar\beta}u^{\gamma\bar\eta}&=&\sum_{\alpha, \beta, \gamma,\eta=1}^{n-1}
\left[\frac{(\delta_{\alpha\beta}\delta_{\gamma\eta}+\delta_{\alpha\eta}\delta_{\beta\gamma})}{2b^2}-\frac{(b')^2(\delta_{\alpha\beta}\delta_{\gamma\eta}+\delta_{\alpha\eta}\delta_{\beta\gamma})}{2b^2}\right]u^{\alpha\bar\beta}u^{\gamma\bar\eta}\\
&=&\frac{1-(b')^2}{2b^2}\left[\left(\sum_{\alpha=1}^{n-1}u^{\alpha\bar\alpha}\right)^2+\sum_{\alpha,\beta=1}^{n-1}\left|u^{\alpha\bar\beta}\right|^2\right].
\end{eqnarray*}
By the curvature formulas given in \eqref{curform}, we find that
\[
-\hat R_{n\overline{n}n\overline{n}}u^{n\overline{n}}u^{n\overline{n}}=-\frac{a''}{4a}\left|u^{n\overline{n}}\right|^2,
\]
\[
-2\sum_{\alpha,\beta=1}^{n-1}u^{\alpha\overline{\beta}}u^{n\overline{n}}\hat R_{\alpha\overline{\beta}n\overline{n}}=-2\sum_{\alpha,\beta=1}^{n-1}u^{\alpha\overline{\beta}}u^{n\overline{n}}\frac{b''}{2b}\delta_{\alpha\beta}=-\frac{b''}{b}u^{n\overline{n}}\sum_{\alpha=1}^{n-1}u^{\alpha\bar\alpha},
\]
\[
-2\sum_{\alpha,\beta=1}^{n-1}u^{\alpha\overline{n}}u^{n\overline{\beta}}\hat R_{\alpha\overline{n}n\overline{\beta}}=-\sum_{\alpha,\beta=1}^{n-1}u^{\alpha\overline{n}}u^{n\overline{\beta}}\frac{b''}{b}\delta_{\alpha\beta}=-\frac{b''}{b}\sum_{\alpha=1}^{n-1}\left|u^{\alpha\overline{n}}\right|^2,
\]
and so
\begin{eqnarray}
-\sum_{\alpha, \beta, \gamma,\eta=1}^{n}\hat R_{\alpha\bar\beta\gamma\bar\eta}u^{\alpha\bar\beta}u^{\gamma\bar\eta}
\notag &=&\frac{1-(b')^2}{2b^2}\left[\left(\sum_{\alpha=1}^{n-1}u^{\alpha\bar\alpha}\right)^2+\sum_{\alpha,\beta=1}^{n-1}\left|u^{\alpha\bar\beta}\right|^2\right]-\frac{a''}{4a}\left|u^{n\overline{n}}\right|^2\\
\label{curf2} &&-\frac{b''}{b}u^{n\overline{n}}\sum_{\alpha=1}^{n-1}u^{\alpha\bar\alpha}-\frac{b''}{b}\sum_{\alpha=1}^{n-1}\left|u^{\alpha\overline{n}}\right|^2.
\end{eqnarray}
The lemma now follows by combining \eqref{curf3} and \eqref{curf2}, together with the fact that $u^{\alpha\bar\beta}$ is an arbitrary nonzero real $(1,\,1)$-vector.
\end{proof}

\section{Comparison with the steady solitons of Apostolov-Cifarelli}\label{s: appendix B}

Let $(z_{1},\ldots,z_{n})$ denote coordinates on $\mathbb{C}^{n}$ with $z_{k}=x_{k}+iy_{k}$ and 
suppose that $(\mathbb{C}^{n},\,g,\,X)$ is a steady K\"ahler-Ricci soliton with $\operatorname{Ric}_{g}=\frac{1}{2}\mathcal{L}_{X}g$,
where 
\begin{equation}\label{standard}
X=2\operatorname{Re}\left(\sum_{\alpha=1}^{n}a_{\alpha}z_{\alpha}\partial_{z_{\alpha}}\right)=
\sum_{\alpha=1}^{n}a_{\alpha}\left(x_{\alpha}\partial_{x_{\alpha}}+y_{\alpha}\partial_{y_{\alpha}}\right)\qquad\textrm{with $a_{\alpha}\in\mathbb{R}$}.
\end{equation}

We begin with a general lemma.

\begin{lem}\label{positive}
If $\operatorname{Ric}_{g}>0$ at $0\in\mathbb{C}^{n}$, then $a_{\alpha}>0$ for all $\alpha=1,\ldots,n$.
\end{lem}

\begin{proof}
Without loss of generality, suppose that $a_{1}\leq 0$ at $0\in\mathbb{C}^{n}$. Let $\nabla$ denote the Levi-Civita connection of $g$. Then 
\begin{equation*}
\begin{split}
0<\operatorname{Ric}_{g}(\partial_{x_{1}},\,\partial_{x_{1}})&=\frac{1}{2}\mathcal{L}_{X}g(\partial_{x_{1}},\,\partial_{x_{1}})\\
&=g\left(\nabla_{\partial_{x_{1}}}X,\,\partial_{x_{1}}\right)\\
&=\sum_{\alpha=1}^{n}a_{\alpha}g\left(\nabla_{\partial_{x_{1}}}\left(x_{\alpha}\partial_{x_{\alpha}}+y_{\alpha}\partial_{y_{\alpha}}\right),\,\partial_{x_{1}}\right)\\
&=\sum_{\alpha=1}^{n}a_{\alpha}g\left(\delta_{\alpha1}\partial_{x_{\alpha}}+x_{\alpha}\nabla_{\partial_{x_{1}}}\partial_{x_{\alpha}},\,\partial_{x_{1}}\right)
+\sum_{\alpha=1}^{n}a_{\alpha}g\left(y_{\alpha}\nabla_{\partial_{x_{1}}}\partial_{y_{\alpha}},\,\partial_{x_{1}}\right)\\
&=a_{1}g(\partial_{x_{1}},\,\partial_{x_{1}})\qquad\textrm{at $0\in\mathbb{C}^{n}$}.
\end{split}
\end{equation*}
This is a contradiction.
\end{proof}

Next, we show that there exist holomorphic coordinates that linearize the soliton vector field of the K\"ahler flying wing steady solitons.

\begin{lem}
For $(\mathbb{C}^{n},\,g,\,X)$ the K\"ahler flying wing steady solitons, holomorphic coordinates exist so that $X$ takes the form given in \eqref{standard}.
\end{lem}

\begin{proof}
The expanding K\"ahler-Ricci solitons in \cite{con-der} 
are asymptotically conical, hence the scalar curvature has a maximum value in the interior of the manifold. Since these expanding solitons have strictly positive Ricci curvature, it follows from the soliton identities that the maximum value of the scalar curvature is achieved at a zero of the soliton vector field. By construction, this is a single point. We denote this point by $0\in\mathbb{C}^{n}$.

Now, the K\"ahler flying wing steady solitons are obtained as the pointed Cheeger-Gromov limits of the aforementioned expanding solitons based at the point $0$, after rescaling the scalar curvature at $0$ to be equal to $1$. It follows that the scalar curvature of the limiting steady soliton $g$ also has a maximum value of $1$ at its base point. Since $\operatorname{Ric}_{g}>0$, by the soliton identities $X$ vanishes at this point. The fact that the limiting soliton potential is convex because $\operatorname{Ric}_{g}>0$ then implies that this zero of $X$ is unique.

Finally, by 
\cite[Proposition 6]{Bry08}, we can choose local holomorphic coordinates on $\mathbb{C}^{n}$ centred at the zero of $X$
so that $X$ takes the form \eqref{standard}. Since $X$  is complete \cite{zhzhang} and has a unique zero, 
these coordinates can be extended globally on $\mathbb{C}^{n}$ using
\cite[Proposition 2.27]{CDS19} so that
the representation \eqref{standard} of $X$ is global.
\end{proof}

Finally, we show that the soliton vector field of our K\"ahler flying wing steady solitons is a multiple of the Euler vector field on $\mathbb{C}^{n}$.

\begin{prop}\label{diagonalisable}
For $(\mathbb{C}^{n},\,g,\,X)$ the K\"ahler flying wing steady solitons in Theorem \ref{t: existence of new}, we have $a_{1}=\ldots=a_{n}>0$ in \eqref{standard}.
\end{prop}

\begin{proof}
By Remark \ref{1eigv}, the Ricci tensor of the expanding K\"ahler-Ricci solitons constructed in \cite{con-der} is diagonal at $0\in\mathbb{C}^{n}$, 
the unique critical point of the soliton vector field. It follows that $\operatorname{Ric}_{g}=\mu g$ for some $\mu\in\mathbb{R}$ at $0\in\mathbb{C}^{n}$ for the K\"ahler flying wing steady solitons. Since $\operatorname{Ric}_{g}>0$, we see from Lemma \ref{positive} above that $\mu>0$. For any index $\alpha=1,\ldots,n$, we then have that
\begin{equation*}
\begin{split}
\mu g(\partial_{x_{\alpha}},\,\partial_{x_{\alpha}})&=\operatorname{Ric}_{g}(\partial_{x_{\alpha}},\,\partial_{x_{\alpha}})=\frac{1}{2}\mathcal{L}_{X}g(\partial_{x_{\alpha}},\,\partial_{x_{\alpha}})\\
&=g\left(\nabla_{\partial_{x_{\alpha}}}X,\,\partial_{x_{\alpha}}\right)\\
&=\sum_{\beta=1}^{n}a_{\beta}g\left(\nabla_{\partial_{x_{\alpha}}}\left(x_{\beta}\partial_{x_{\beta}}+y_{\beta}\partial_{y_{\beta}}\right),\,\partial_{x_{\alpha}}\right)\\
&=\sum_{\beta=1}^{n}a_{\beta}g\left(\delta_{\alpha\beta}\partial_{x_{\beta}}+x_{\beta}\nabla_{\partial_{x_{\alpha}}}\partial_{x_{\beta}},\,\partial_{x_{\alpha}}\right)
+\sum_{\beta=1}^{n}a_{\beta}g\left(y_{\beta}\nabla_{\partial_{x_{\alpha}}}\partial_{y_{\beta}},\,\partial_{x_{\alpha}}\right)\\
&=a_{\alpha}g(\partial_{x_{\alpha}},\,\partial_{x_{\alpha}})\qquad\textrm{at $0\in\mathbb{C}^{n}$},
\end{split}
\end{equation*}
i.e., $a_{\alpha}=\mu>0$ for all $\alpha=1,\ldots,n$.
\end{proof}

Next recall the steady solitons from \cite[Theorem 1.2]{apostolov2023hamiltonian}. We consider this theorem 
with $\ell=2$, $d_{1}=0$, and $d_{2}=n-2$. In this case, this theorem yields complete $U(1)\times U(n-1)$-invariant steady solitons on $\mathbb{C}^{n}$. 
By \cite[Lemma 5.1]{apostolov2023hamiltonian}, the soliton vector field $X$ is given by $$X=-\frac{a}{(\alpha_{2}-\alpha_{1})}\left(-\frac{q(\alpha_{1})}{2}X_{1}+\frac{q(\alpha_{2})}{2(n-1)}X_{2}\right),$$
where $X_{1}$ is the Euler vector field on the first $\mathbb{C}$-factor and 
$X_{2}$ is the Euler vector field on the second $\mathbb{C}^{n-1}$-factor of $\mathbb{C}^{n}$, where $q(t)$ is a polynomial of degree one, and $-\infty<\alpha_{1}<\alpha_{2}<\infty$ (cf.~\cite[Section 5.1]{apostolov2023hamiltonian}). Write $q(t)=q_{0}+q_{1}t$ and note that
we can always normalize by translation and scaling so that $\alpha_{1}=0$ and $\alpha_{2}=1$. Then it is clear that:
\begin{lem}\label{hello}
$X$ coincides with a multiple of the Euler vector field on $\mathbb{C}^{n}$ if and only if $q_{1}=-n q_0$.
\end{lem}

The next proposition shows that this can never be the case.
\begin{prop}\label{bye}
The vector field $X$ of the steady solitons of \cite[Theorem 1.2]{apostolov2023hamiltonian} with\linebreak $U(1)\times U(n-1)$-symmetry can never  
be a multiple of the Euler vector field on $\mathbb{C}^{n}$.
\end{prop}

\begin{proof}
In this case we have that $p_{c}(t)=(t-1)^{n-2}$ (cf.~\cite[p.535]{apostolov2023hamiltonian}). Suppose that the proposition is false. Then by Lemma 
\ref{hello} we have that $q(t)=1-nt$ up to a nonzero scalar multiple. By
construction, the solution $F$ of the ODE satisfies \cite[equation (5.1)]{apostolov2023hamiltonian}:
$$0=F(1)=e^{-2a}\int_{0}^{1}e^{2ax}q(x)p_{c}(x)\,dx,$$
which by definition of $p_{c}(t)$ and $q(t)$ equates to 
$$0=\int_{0}^{1}e^{2ax}(1-nx)(x-1)^{n}\,dx.$$
Here $a>0$ is a real constant. Then
\begin{equation*}
\begin{split}
0&=\int_{0}^{1}\left(x(x-1)^{n-1}e^{2ax}\right)'\,dx\\
&=\int_{0}^{1}(x-1)^{n-1}e^{2ax}\,dx+(n-1)\int_{0}^{1}x(x-1)^{n-2}e^{2ax}\,dx+2a\int_{0}^{1}x(x-1)^{n-1}e^{2ax}\,dx\\
&=\int_{0}^{1}\left((x-1)+(n-1)x\right)(x-1)^{n-2}e^{2ax}\,dx+2a\int_{0}^{1}x(x-1)^{n-1}e^{2ax}\,dx\\
&=-\underbrace{\int_{0}^{1}(1-nx)(x-1)^{n-2}e^{2ax}\,dx}_{=\,0}+2a\int_{0}^{1}x(x-1)^{n-1}e^{2ax}\,dx.
\end{split}
\end{equation*}
This is a contradiction because $x(x-1)^{n-1}e^{2ax}$ has a sign on $(0,\,1)$.   
\end{proof}

Finally, the steady solitons from \cite[Theorem 1.4]{apostolov2023hamiltonian} 
can never have a soliton vector field a scalar multiple of the Euler vector field on $\mathbb{C}^{n}$; cf.~\cite[Section 5.3]{apostolov2023hamiltonian}. As a result of these observations, we obtain

\begin{cor}\label{noniso}
The solitons of Theorem \ref{t: existence of new} are not isometric to those of 
\cite{apostolov2023hamiltonian}.
\end{cor}

\begin{proof}
  If a steady soliton of \cite{apostolov2023hamiltonian} doesn't have strictly positive sectional curvature or is not $U(1)\times U(n-1)$-invariant, then the result is clear. So suppose that it has both of these properties. Then any isometry will map a critical point of the scalar curvature to a critical point of the scalar curvature. By the soliton identities, the fact that the Ricci curvature is strictly positive implies that for each steady soliton, the critical points of the scalar curvature and the zero set of the soliton vector field coincide. Each soliton vector field having one zero at the origin and strictly positive Ricci curvature implies that this zero is unique, hence any isometry will map the origin to the origin. At every point, the steady soliton equation implies that the eigenvalues of the Ricci tensor coincide with those of the Hessian of the potential function. In particular, at the origin, being a zero of the soliton vector field, the Hessian of the soliton potential is independent of the choice of metric. This yields a contradiction. Indeed, at the origin, Proposition \ref{diagonalisable} implies that the Ricci tensor of the steady solitons of Theorem \ref{t: existence of new} is a multiple of the identity, whereas the other observations above imply that this is not the case for those steady solitons of \cite{apostolov2023hamiltonian} admitting $U(1)\times U(n-1)$-symmetry.
\end{proof}

\newpage
\bibliography{bib}
\bibliographystyle{abbrv}

\end{document}